\documentclass[a4paper]{amsart}

\usepackage[T1]{fontenc}
\usepackage[utf8]{inputenc}
\usepackage[english]{babel}
\usepackage[includeheadfoot, margin=1in]{geometry}
\usepackage{ifthen}

\usepackage[bookmarksnumbered]{hyperref}

\usepackage{amsmath}
\usepackage{amssymb}
\usepackage{mathrsfs}
\usepackage{eucal}			
\usepackage{tensor}

\usepackage{amsthm}
\usepackage{enumitem}  
\usepackage[normalem]{ulem} 

\usepackage{tikz}  \usetikzlibrary{arrows,positioning}

\usepackage[numbers]{natbib}
\usepackage{doi}

\newcounter{generalnumbering}   \numberwithin{generalnumbering}{section}

\theoremstyle{plain}    \newtheorem{theorem}[generalnumbering]{Theorem}
\theoremstyle{plain}    \newtheorem{corollary}[generalnumbering]{Corollary}
\theoremstyle{definition}   \newtheorem{definition}[generalnumbering]{Definition}
\theoremstyle{definition}   \newtheorem{example}[generalnumbering]{Example}
\theoremstyle{plain}    \newtheorem{proposition}[generalnumbering]{Proposition}
\theoremstyle{plain}    \newtheorem{lemma}[generalnumbering]{Lemma}

\newcommand{\namefordifferentenvironment}{}

\theoremstyle{plain}    \newtheorem{plainstyle}[generalnumbering]{\namefordifferentenvironment}
\theoremstyle{plain}    \newtheorem*{plainstyle*}{\namefordifferentenvironment}
\theoremstyle{definition}    \newtheorem{definitionstyle}[generalnumbering]{\namefordifferentenvironment}
\theoremstyle{definition}    \newtheorem*{definitionstyle*}{\namefordifferentenvironment}

\newenvironment{penv*}[1]{\renewcommand{\namefordifferentenvironment}{#1}\begin{plainstyle*}}{\end{plainstyle*}}

\newenvironment{denv*}[1]{\renewcommand{\namefordifferentenvironment}{#1}\begin{definitionstyle*}}{\end{definitionstyle*}}

\newenvironment{remark}{\begin{denv*}{Remark}}{\end{denv*}}

\newcommand{\ntag}{\tag{\thegeneralnumbering}\stepcounter{generalnumbering}}

\DeclareMathOperator{\so}{\mathfrak{s}}
\DeclareMathOperator{\ra}{\mathfrak{r}}

\DeclareMathOperator{\id}{id}
\DeclareMathOperator{\dom}{dom}
\DeclareMathOperator{\ran}{ran}

\DeclareMathOperator{\cod}{cod}

\newcommand*{\defeq}{\mathrel{\vcenter{\baselineskip0.5ex \lineskiplimit0pt\hbox{\scriptsize.}\hbox{\scriptsize.}}}=}
\newcommand*{\eqdef}{=\mathrel{\vcenter{\baselineskip0.5ex \lineskiplimit0pt\hbox{\scriptsize.}\hbox{\scriptsize.}}}}

\makeatletter
\DeclareFontEncoding{LS1}{}{}
\DeclareFontSubstitution{LS1}{stix}{m}{n}
\DeclareMathAlphabet{\mathfancy}{LS1}{stixscr}{m}{n}
\makeatother

\newcommand{\cat}[1]{\normalfont{\textsc{\textbf{#1}}}}
\DeclareMathOperator{\IG}{\mathfancy{I\mkern-2mu G}\mkern-1mu}

\makeatletter
\def\G{\@ifnextchar[{\@Gwithbrak}{\@Gwithoutbrak}}
\def\@Gwithbrak[#1]{\mathcal{G}^{(#1)}}
\def\@Gwithoutbrak{\mathcal{G}}
\makeatother

\makeatletter
\def\H{\@ifnextchar[{\@Hwithbrak}{\@Hwithoutbrak}}
\def\@Hwithbrak[#1]{\mathcal{H}^{(#1)}}
\def\@Hwithoutbrak{\mathcal{H}}
\makeatother

\title{Étale inverse semigroupoids - the fundamentals}
\author{Luiz Gustavo Cordeiro
}
\thanks{The author was supported by the ANR project GAMME (ANR-14-CE25-0004)}

\address{UMPA, UMR 5669 CNRS -- École Normale Supérieure de Lyon\\
46 allée d'Italie, 69364 Lyon Cedex 07, France}
\email{luizgc6@gmail.com, luis-gustavo.cordeiro@ens-lyon.fr}				


\begin{document}


\begin{abstract}
    In this article we will study semigroupoids, and more specifically inverse semigroupoids. These are a common generalization to both inverse semigroups and groupoids, and provide a natural language on which several types of dynamical structures may be described. Moreover, this theory allows us to precisely compare and simultaneously generalize aspects of both the theories of inverse semigroups and groupoids.
    
    We begin by comparing and settling the differences between two notions of semigroupoids which appear in the literature (one by Tilson and another by Exel). We specialize this study to inverse semigroupoids, and in particular an analogue of the Vagner-Preston Theorem is obtained. This representation theorem leads to natural notions of actions, and more generally $\land$-preactions and partial actions, of étale inverse semigroupoids, which generalize topological actions of inverse semigroups and continuous actions of étale groupoids. Many constructions which are commonplace in the theories of inverse semigroups and groupoids are also generalized, and their categorical properties made explicit. We finish this paper with a version of non-commutative Stone duality for ample inverse semigroupoids, which utilizes several of the aforementioned constructions.
\end{abstract}

\maketitle

\section{Introduction}

Groups, semigroups and categories are ubiquitous in Mathematics. The dynamical systems associated with such structures are usually assumed to satisfy some form of associativity, roughly stating that ``performing certain related operations in any order will always yield the same result''. This is clearly the case, for example, when considering group actions. Moreover, different components of a given system play different roles: Given a left action of a group $G$ on a set $X$, the elements of $G$ can be composed both with elements of $G$ (on either side) and with elements of $X$ on their right (via the action), however elements of $X$ can only be composed with elements of $G$ on their left. This leads us to consider partially defined operations.

Semigroupoids are an algebraic abstraction of these concepts: they are sets with partially defined binary operations with are associative in a precise way. In particular, the two most prominent examples of semigroupoids are categories and semigroups.

Whereas the most common categories one first encounters consist of unrelated objects of a given signature, such as the categories of rings, topological spaces, posets, etc\ldots and certain structure-preserving maps between them, we consider a more algebraic-geometric approach: categories (and semigroupoids) ought to be regarded as collections of transformations between subobjects of a single object. The classical version of this approach is well-established, where one seeks to understand a given object by looking at its automorphism group. This is the first motivation for this work.

This approach -- to study an object by analyzing its ``partial automorphisms'' -- has been popularized in the contexts of Topological Dynamics and Operator Algebras mainly after the introduction of \emph{partial actions} of groups by Exel on his study of the structure of C*-algebras endowed with circle actions, \cite{MR1276163}. This turned out to be a rich an fruitful research direction, since it allows one to consider algebras and topological dynamical systems which are not induced (in any natural fashion) by a group action, but rather from transformations among its subsets (see \cite{arxiv1804.00396,MR3789176,MR2645883}, for example).

A variation on the ``partial transformation'' approach was taken by Sieben in \cite{MR1671944,MR1456588}, who instead considered \emph{actions of inverse semigroups} on topological spaces and algebras. Inverse semigroups have been thoroughly studied in the last century, which provides solid foundation for this approach.

Another approach to Topological Dynamics and Operator Algebras, and which in fact predates the works above, is the usage of \emph{groupoids}, which have played a central role in the theory of $C^*$-algebras since Renault's seminal work \cite{MR584266}, since they provide a geometric counterpart to a large class of (non-commutative) $C^*$-algebras. In simple terms, and as an interpretation of \cite[Theorem 5.9]{MR2460017}, groupoids are the ``non-commutative spectra'' of ``non-commutative, dynamical $C^*$-algebras''.

Groupoids, partial actions of groups, and actions of inverse semigroups -- or more generally \emph{partial actions of inverse semigroups} -- are related via the constructions of ``groupoids of germs'' and ``crossed products''. We refer to \cite{MR2045419,arxiv1804.00396,MR3743184,MR3851326,MR2799098,MR1724106,MR1671944,MR2565546}. In fact, groupoids and inverse semigroups may be seen as dual to each other, via the \emph{non-commutative Stone duality} of Lawson-Lenz, \cite{MR3077869}. We also refer to \cite{MR2969047,MR2644910,MR2304314} for related studies.

Actions of groupoids have also appeared throughout the literature, in the context of Lie theory (\cite{MR2012261}), algebraic topology (\cite{MR2273730}), operator algebras and topological dynamics (\cite{MR2969047,MR2941279,MR2982887}). A generalization of groupoids and inverse monoids, called \emph{inverse categories}, were initially introduced in \cite{MR0506554} and have been considered in recent work, e.g.\ in the study of \emph{tilings} (\cite{MR1736698}), logic and recursion theory (\cite{MR1871071,MR3605681}) and crossed product algebras (\cite{MR2959793}).

The main goal of this article is to study the structure of \emph{inverse semigroupoids}, which are a common generalization of both inverse semigroups and groupoids. First we begin by comparing the notions of semigroupoid used in \cite{MR3597709,MR915990} to that of \cite{MR2754831}, which is more general. After that, we introduce inverse semigroupoids and prove some basic facts about their structure (some which had already been proven in \cite{MR3597709}). We follow with a representation theorem, which is an analogue of Cayley's Theorem for groups and the Vagner-Preston theorem for inverse semigroups. It motivates a notion of action, and more generally $\land$-preactions and partial actions of inverse semigroupoids on semigroupoids, in a manner which covers previously used notions for inverse semigroups and for groupoids. These are used to construct ``semidirect products'' which generalize transformation groupoids of partial group actions, semidirect products of inverse semigroups, semidirect products of groupoids, among others.

In the third section we look at \emph{topological semigroupoids}, specializing to the \emph{étale inverse semigroupoid} case, and generalize some results known for topological groupoids to this setting. The fourth section generalizes some constructions from the theory of inverse semigroups, and we describe their categorical properties explicitly. The fifth and last section contains a generalization of non-commutative Stone duality to the context of inverse semigroupoids.

\subsection*{Notation and running conventions}

The \emph{domain} and \emph{codomain} of a map $f$ are denoted by $\dom(f)$ and $\cod(f)$, respectively. The image of a subset $A\subseteq\dom(f)$ is denoted as $f(A)$, and the range of $f$ is $\ran(f)\defeq f(\dom(f))$. Note that $\cod(f)=\ran(f)$ if and only if $f$ is surjective.

Given maps $f_i\colon X_i\to Y_i$ ($i=1,2$), we define $f_1\times f_2\colon X_1\times X_2\to Y_2\times Y_2$ as $(f_1\times f_2)(x_1,x_2)=(f_1(x_1),f_2(x_2))$. If $X_1=X_2\eqdef X$, we define $(f_1,f_2)\colon X\to Y_1\times Y_2$ as $(f_1,f_2)(x)=(f_1(x),f_2(x))$.
The identity map of a set $X$ is denoted as $\id_X$, and more generally $\id_{\mathcal{C}}$ stands for the identity functor of a category $\mathcal{C}$.

We denote as $\mathbb{N}=\left\{0,1,2,\ldots,\right\}$ the set of non-negative integers, and as $\mathbb{N}_{\geq 1}=\mathbb{N}\setminus\left\{0\right\}$ the set of positive integers.

We assume familiarity with inverse semigroups and étale groupoids. The main references are \cite{MR1694900,MR1724106} (see also the author's PhD thesis \cite{cordeirothesis}). A \emph{semilattice} is a poset $(P,\leq)$ admiting binary meets (infima), and is always regarded as an inverse semigroup under meets: $ab=a\land b=\inf\left\{a,b\right\}$ for all $a,b\in P$.

\subsection{Setoids, bundles and partitions}
    \begin{denv*}{Setoids}
A \emph{setoid} is a pair $(X,R)$, where $X$ is a set and $R$ is an equivalence relation on $X$. Equivalently, a setoid is a principal groupoid. A morphism of setoids $(X,R_X)$ and $(Y,R_Y)$ is simply a relational morphism, that is, a map $f\colon X\to Y$ such that $(f\times f)(R_X)\subseteq R_Y$. Equivalently, a morphism of setoids, seen as principal groupoids, is simply a groupoid homomorphism.
\end{denv*}

\begin{denv*}{Bundles/Fibrations}
A \emph{bundle} or \emph{fibration} is a function $\pi\colon X^{(1)}\to X^{(0)}$. $X^{(0)}$ is called the \emph{base space}, and we say that $X^{(1)}$ is \emph{fibred over $X^{(0)}$}, or that $\pi$ is a \emph{bundle over $X^{(0)}$}. Although of most insterest are the surjective bundles, we do not make such an assumption.

A morphism between bundles $(X^{(0)},X^{(1)},\pi_X)$ and $(Y^{(0)},Y^{(1)},\pi_Y)$ is a pair $f=(f^{(0)},f^{(1)})$ of maps $f^{(i)}\colon X^{(i)}\to Y^{(i)}$ ($i=0,1$) such that $f^{(0)}\circ\pi_X=\pi_Y\circ f^{(1)}$. In other words, bundles and their morphisms are simply the category of arrows of $\cat{Set}$, the category of sets and functions (see \cite[p.\ 40]{MR1712872}).
\end{denv*}

\begin{denv*}{Partitions}
A \emph{partition} is a pair $(X,\mathscr{P})$, where $X$ is a set and $\mathscr{P}$ is a partition of $X$ (i.e., a collection of nonempty, pairwise disjoint subsets of $X$ such that $\bigcup\mathscr{P}=X$). A \emph{morphism} between partitions $(X,\mathscr{P}_X)$ and $(Y,\mathscr{P}_Y)$ is a map $f\colon X\to Y$ such that for every $A\in\mathscr{P}_X$, there exists $B\in\mathscr{P}(Y)$ such that $f(A)\subseteq B$.
\end{denv*}

Setoids, surjective bundles and partitions are equivalent concepts. More precisely, the categories $\cat{Std}$, $\cat{Bdl}_{\mathrm{sur}}$ and $\cat{Part}$ which they respectively define are equivalent:
\begin{enumerate}
    \item Given a setoid $(X,R)$, we construct the bundle $(X/R,X,\pi_X)$, where $X/R$ is the quotient space and $\pi_X\colon X\to X/R$ is the quotient map. Given a morphism of setoids $f\colon (X,R_X)\to (Y,R_Y)$, there exists a unique map $f^{(0)}\colon X/R_X\to Y/R_Y$ such that $f^{(0)}\circ\pi_X=\pi_Y\circ f$. Then $(f^{(0)},f)$ is a morphism of bundles.
    
    \item Given a bundle $(X^{(0)},X^{(1)},\pi)$, we consider the partition $\mathscr{P}_X=\left\{\pi^{-1}(x):x\in\pi(X^{(1)})\right\}$ of $X^{(1)}$. Given a bundle morphism $f=(f^{(0)},f^{(1)})\colon(X^{(0)},X^{(1)},\pi_X)\to (Y^{(0)},Y^{(1)},\pi_Y)$, the map $f^{(1)}$ is a morphism of partitions $f^{(1)}\colon (X^{(1)},\mathscr{P}_{X})\to (Y^{(1)},\mathscr{P}_{Y})$.
    
    \item Any partition $(X,\mathscr{P})$ induces an equivalence relation $R_{\mathscr{P}}$ on $X$ as
    \[R_{\mathscr{P}}=\left\{(x,y)\in X\times X:\text{there exists }A\in\mathscr{P}\text{ such that }x,y\in A\right\}\]
    Given a morphism of partitions $f\colon (X,\mathscr{P}_X)\to (Y,\mathscr{P}_Y)$, the same map $f\colon X\to Y$ is also a morphism of setoids.
\end{enumerate}

All of these constructions are functorial. If we denote $F\colon \cat{Std}\to\cat{Bdl}_{\mathrm{sur}}$, $G\colon \cat{Bdl}_{\mathrm{sur}}\to\cat{Part}$ and $H\colon\cat{Part}\to\cat{Std}$ the functors described above, then $H\circ G\circ F=\id_{\cat{Std}}$, $G\circ F\circ H$ is equivalent to $\id_{\cat{Bdl}_{\mathrm{sur}}}$, and $F\circ H\circ G=\id_{\cat{Part}}$. In particular, these three categories are equivalent ($\cat{Std}$ and $\cat{Part}$ being in fact isomorphic). We will, therefore, not make any meaningful distinction between these concepts.

Bundles are generally easier to describe in the topological setting: A \emph{continuous} or \emph{topological bundle} is a continuous map $\pi\colon X^{(1)}\to X^{(0)}$ between topological spaces. In this  setting, we also consider only bundle morphisms $f=(f^{(1)},f^{(0)})$ such that $f^{(0)}$ and $f^{(1)}$ are continuous. Equivalently, continuous bundles and their morphisms form the category of arrows of $\cat{Top}$, the category of topological spaces and continuous maps.

Given bundles (fibrations) $\pi_i\colon X_i\to X^{(0)}$ ($i=1,2,$) over the same base space $X^{(0)}$, the \emph{fibred product} of $\pi_1$ and $\pi_2$ is
\[X_1\tensor[_{\pi_1}]{\ast}{_{\pi_2}} X_2=\left\{(x_1,x_2)\in X_1\times X_2:\pi_1(x_1)=\pi_2(x_2)\right\}.\]

\subsection{Graphs}
    Graphs will be used in the description of semigroupoids, and provide a geometric picture which will be useful throughout this paper. The graphs we consider are sometimes called \emph{directed multigraphs}, since all edges come with a direction and we allow multiple edges between points.

\begin{definition}
A \emph{graph} is a tuple $G=(G^{(0)},G^{(1)},\so,\ra)$, where $G^{(0)}$ and $G^{(1)}$ are classes of \emph{vertices} and \emph{arrows}, respectively, and $\so,\ra\colon G^{(1)}\to G^{(0)}$ are functions, called the \emph{source} and \emph{range} maps.
\end{definition}

Alternative terminology is sometimes employed. Elements of $G^{(0)}$ are also called \emph{objects} or \emph{units}; elements of $G^{(1)}$ are called \emph{edges}; The source map is also called the \emph{domain} map, and the range map the \emph{target} or \emph{codomain} map. We may alternate between these terminologies depending on the context. If necessary, we will use subscripts to specify the graph $G$, as in writing $\so_G$ and $\ra_G$.

We usually write simply $G$ in place of $G^{(1)}$, so that an inclusion of the form $g\in G$ means that $g$ is an arrow of $G$.

Note that the source and range maps give fibred structures on $G^{(1)}$ over $G^{(0)}$. Moreover, we will generally assume that $G=\so(G)\cup\ra(G)$, in the same manner that surjective bundles are the ones of interest.

A \emph{graph morphism} $\phi\colon G\to H$ between graphs $G$ and $H$ is a pair $\phi=(\phi^{(0)},\phi^{(1)})$ of maps $\phi^{(0)}\colon G^{(0)}\to H^{(0)}$ and $\phi^{(1)}\colon G^{(1)}\to H^{(1)}$ such that $\so_H\circ\phi^{(1)}=\phi^{(0)}\circ\so_G$ and $\ra_H\circ\phi^{(1)}=\phi^{(0)}\circ\ra_G$. In other words, it is a simultaneous fibred morphism from $G$ to $H$ over their respective source and range maps. A \emph{graph isomorphism} is a graph morphism $\phi$ such that both $\phi^{(0)}$ and $\phi^{(1)}$ are bijective, and in this case $\phi^{-1}=((\phi^{(0)})^{-1},(\phi^{(1)})^{-1})$ is also a graph morphism.
\[\begin{tikzpicture}
\node (G1) at (0,1) {$G^{(1)}$};
\node (G2) at ([shift={+(4,0)}]G1) {$G^{(1)}$};
\foreach \i in {1,2}
{\node (G\i2) at ([shift={+(1.5,0)}]G\i) {$H^{(1)}$};
\node (G\i3) at ([shift={+(0,-1.5)}]G\i) {$G^{(0)}$};
\node (G\i4) at ([shift={+(1.5,0)}]G\i3) {$H^{(0)}$};
\draw[->] (G\i)--(G\i2) node[above,midway] {$\phi^{(1)}$};
\draw[->] (G\i)--(G\i3) node[left,midway] {${\ifthenelse{\i=1}{\so}{\ra}}_G$};
\draw[->] (G\i3)--(G\i4) node[below,midway] {$\phi^{(0)}$};
\draw[->] (G\i2)--(G\i4) node[right,midway] {${\ifthenelse{\i=1}{\so}{\ra}}_H$};}
\node (G3) at ([shift={+(4,0)}]G2) {$\bullet$};
\node (G32) at ([shift={+(3,0)}]G3) {$\bullet$};
\node (G33) at ([shift={+(0,-1.5)}]G3) {$\bullet$};
\node (G34) at ([shift={+(0,-1.5)}]G32) {$\bullet$};
\draw[->] (G3) to[out=15,in=165] node[midway] (M1) {} (G32);
\draw[->] (G33) to[out=-15,in=195] node[midway] (M2) {} (G34);
\draw[->,dashed,thin] (M1)--(M2) node[midway,left] {$\phi^{(1)}$};
\draw[->,dashed,thin] (G3)--(G33) node[midway,left] {$\phi^{(0)}$};
\draw[->,dashed,thin] (G32)--(G34) node[midway,right] {$\phi^{(0)}$};
\node at ([shift={+(1,0)}]G32) {$G$};
\node at ([shift={+(1,0)}]G34) {$H$};
\end{tikzpicture}\]

A \emph{sink} in a graph is a vertex $v\in G^{(0)}$ such that $\so^{-1}(v)=\varnothing$. A \emph{source} in a graph is a vertex $v\in G^{(0)}$ such that $\ra^{-1}(v)=\varnothing$.
\[
    \begin{tikzpicture}
    \node (N) at (0,0) {$\bullet$};
    \node at (-0.5,-1) {A sink};
    \foreach \x in {150,180,210}
    {\draw[->] (N)+(\x:1)--(N);}
    \node (B) at (4,0) {$\bullet$};
    \node at (4.5,-1) {A source};
    \foreach \x in {0,30,330}
    {\draw[->] (B)--+(\x:1);}
    \end{tikzpicture}
\]
We will interpret the arrows of a graph as functions, and thus they will be composed from right to left, as functions are. Given $k\in\mathbb{N}_{\geq 1}$, we consider the set of \emph{$k$-paths} of a graph $G$ as
\[G^{(k)}=\left\{(g_1,\ldots,g_k)\in (G^{(1)})^ k:\so(g_i)=\ra(g_{i+1})\text{ for all }i\right\}\]
Naturally, edges of $G$ are identified with $1$-paths, so the notation $G^{(1)}$ is unambiguous in this manner. We regard $G^ {(k)}$ as a graph itself, with vertex set $G^{(0)}$ and the source and range maps given by
\[\so(g_1,\ldots,g_k)=\so(g_k)\qquad\text{and}\qquad\ra(g_1,\ldots,g_k)=\ra(g_1).\]

We also regard $G^{(0)}$ as a trivial graph with vertex set $G^{(0)}$, and the source and range maps the identity function: $\so,\ra=\id_{G^{(0)}}$.

If $G$ and $H$ are graphs over the same vertex set $G^{(0)}=H^{(0)}$, we make the fibred product $G\tensor[_{\so}]{\ast}{_{\ra}} H$ into a graph over that same vertex set, with source and range maps $\ra(g,h)=\ra(g)$ and $\so(g,h)=\so(h)$.  The construction of graphs of paths obey the ``rules of exponentiation'', where the fibred product $\tensor[_{\so}]{\ast}{_{\ra}}$ takes the role of the product: given $k,p\in\mathbb{N}_{\geq 0}$, we have natural isomorphisms
\[(G^{(k)})^{(p)}\cong G^{(pk)}\qquad\text{and}\qquad G^{(k)}\tensor[_{\so}]{\ast}{_{\ra}}G^{(p)}\cong G^{(k+p)}.\]
    
\section{Inverse semigroupoids}

    \subsection{Exel and graphed semigroupoids}
    We will now define semigroupoids, which are generalizations of both semigroups and categories. Every category $\mathcal{C}$ comes with an underlying graph structure, where vertices and arrows correspond respectively to objects and morphisms of $\mathcal{C}$. However, the vertex set of $\mathcal{C}$ can always be recovered from the arrow set, by identifying each object to its corresponding identity morphism, and so categories may be defined purely in terms of their arrow space. This becomes an issue in the case of semigroupoids, where we do not have identity elements (or something similar) anymore. We have two working definitions of semigroupoids: One purely algebraic, introduced by Exel in \cite{MR2754831}, and one where we assume an underlying graph structure, introduced by Tilson in \cite{MR915990}. It should be noted that every semigroupoid in the sense of Tilson is a semigroupoid in the sense of Exel (Proposition \ref{prop:graphedsemigroupoidsaresemigroupoids}).

To avoid any confusion, we will use capital greek letters $\Lambda,\Gamma,\ldots$ to denote semigroupoids in the sense of Exel, without a priori underlying graph structures. These are called simply \emph{semigroupoids}, or \emph{Exel semigroupoids} whenever such precision is warranted. Capital caligraphic latin letters $\mathcal{S},\mathcal{T},\ldots$ will be used to denote semigroupoids in the sense of Tilson, with an underlying graph structure, and these will always be called \emph{graphed semigroupoids}.

\begin{definition}[{\cite[Definition 2.1]{MR2754831}}]\label{def:semigroupoid}
An \emph{Exel semigroupoid} or simply \emph{semigroupoid} is a set $\Lambda$ equipped with a subset $\Lambda^{[2]}\subseteq\Lambda\times\Lambda$ and a \emph{product} map $\mu\colon\Lambda^{[2]}\to\Lambda$, denoted by concatenation, $\mu(f,g)=fg$, which is associative in the following sense: For all $f,g,h\in\Lambda$, the statements
\begin{enumerate}[label=(\roman*)]
    \item\label{def:semigroupoiditem1} $(f,g)\in\Lambda^{[2]}$ and $(g,h)\in\Lambda^{[2]}$;
    \item\label{def:semigroupoiditem2} $(f,g)\in\Lambda^{[2]}$ and $(fg,h)\in\Lambda^{[2]}$;
    \item\label{def:semigroupoiditem3} $(g,h)\in\Lambda^{[2]}$ and $(f,gh)\in\Lambda^{[2]}$;
\end{enumerate}
are equivalent and in case any (all) of them holds, we have $(fg)h=f(gh)$. If necessary for precision, we will say instead that the triple $(\Lambda,\Lambda^{[2]},\mu)$ is a semigroupoid.
\end{definition}

We use square brackets when denoting the subset $\Lambda^{[2]}$ of $\Lambda\times\Lambda$ to stress the fact that $\Lambda$ has no graph structure. Moreover, instead of saying that a pair $(f,g)$ belongs to $\Lambda^{[2]}$ we may simply state that ``the product $fg$ is (well-)defined''.

The product of subsets of a semigroupoid $\Lambda$ is regarded in the standard manner: If $A,B\subseteq\Lambda$, the set $AB$ consists of all products $ab$ which are defined, where $a\in A$ and $b\in B$ - that is, $AB=\mu((A\times B)\cap\Lambda^{[2]})$. If $a\in\Lambda$, the products $aA$ and $Aa$ are defined similarly.

The associativity condition on the product may be regarded as follows: if any of the products $(fg)h$ or $f(gh)$ are defined, then the other product is also defined and they are equal, so we simply denote it $fgh$. Also, if $fg$ and $gh$ are both defined, $fgh$ is also defined. This allows us to move parentheses at will during computations.

\begin{example}
Let $\theta$ be a left action of a group $G$ on a set $X$. Consider the disjoint union $\Lambda\defeq G\sqcup X$. We make $\Lambda$ into a semigroupoid by extending the product of $G$ to a partial product on $\Lambda$ with the action, viz.\ $gx=\theta_g(x)$ for all $g\in G$ and $x\in X$.

This is an alternative to the more useful construction of a \emph{transformation groupoid}, which is a particular case of a semidirect product of semigroupoids (see Subection \ref{sec:dualprehomomorphisms}).
\end{example}

Let us look at some counter-examples for the associativity condition, where some of items \ref{def:semigroupoiditem1}-\ref{def:semigroupoiditem3} are valid but the others are not.

\begin{example}
Let $\Lambda=\left\{f,g,h\right\}$, and define $fg=f$ and $gh=h$. Then $fg$ and $gh$ are defined, but neither $(fg)h$ nor $f(gh)$ are defined.
\end{example}

\begin{example}
Let $\Lambda=\left\{f,g,h\right\}$, and define $fg=g$, $gh=h$. Then $fg$, $gh$ and $(fg)h$ are defined, but $f(gh)$ is not. A similar example may be constructed, where $f(gh)$ is defined but $(fg)h$ is not.
\end{example}

\begin{example}
Let $\Lambda=\left\{f,g,h\right\}$ with product $fg=hh=h$. Then $(fg)h$ is defined, but $gh$ is not. A similar example may be constructed where $f(gh)$ is defined but $fg$ is not.
\end{example}

\begin{example}
If $\Lambda$ is any set endowed with a non-associative binary operation, then \ref{def:semigroupoiditem1}-\ref{def:semigroupoiditem3} of Definition \ref{def:semigroupoid} are always valid, and in particular equivalent, but $\Lambda$ is not a semigroupoid. For example, take $\Lambda=\left\{a,b\right\}$, $aa=ab=b$, $bb=ba=a$. Then $(aa)a=a$ but $a(aa)=b$.
\end{example}

Sub-semigroupoids and semigroupoid homomorphisms are defined in the natural manner.

\begin{definition}
A \emph{sub-semigroupoid} of a semigroupoid $\Lambda$ is a subset $\Delta\subseteq\Lambda$ such that $\Delta\Delta\subseteq\Delta$.
\end{definition}

\begin{definition}
A \emph{homomorphism} between semigroupoids $\Lambda$ and $\Gamma$ is a map $\phi\colon\Lambda\to\Gamma$ such that $(\phi\times\phi)(\Lambda^{[2]})\subseteq\Gamma^{[2]}$ and $\phi(ab)=\phi(a)\phi(b)$ for all $(a,b)\in\Lambda^{[2]}$.
\end{definition}

It is immediate to verify that if $\phi\colon \Lambda\to\Gamma$ is a homomorphism of semigroupoids and $\Delta$ is a sub-semigroupoid of $\Gamma$, then $\phi^{-1}(\Delta)$ is a sub-semigroupoid of $\Lambda$. However, contrary to the cases of semigroups and groupoids, the image of a homomorphism $\phi\colon\Lambda\to\Gamma$ between semigroupoids is not necessarily a sub-semigroupoid of $\Gamma$.

\begin{example}\label{ex:imageofhomomorphismisnotasubsemigroupoid}
    Let $\Lambda=\left\{e,f\right\}$, with operations $ee=e$ and $ff=f$ and $\Gamma=\left\{e,f,g\right\}$ the semigroup with product $ee=e$, $ff=f$, and all other products $xy=g$. Then the inclusion $\iota\colon\Lambda\hookrightarrow\Gamma$ is a semigroupoid homomorphism, but the image $\iota(\Lambda)=\left\{e,f\right\}$ is not sub-semigroupoid of $\Gamma$.
\end{example}

We will now consider semigroupoids in the sense of Tilson.

\begin{definition}[{\cite[p.\ 194]{MR915990}}]\label{def:graphedsemigroupoid}
A \emph{graphed semigroupoid} is a tuple $(\mathcal{S}^{(0)},\mathcal{S},\so,\ra,\mu)$, where $(\mathcal{S}^{(0)},\mathcal{S},\so,\ra)$ is a graph and $\mu\colon\mathcal{S}^{(2)}\to\mathcal{S}$ is a \emph{product map}, denoted by concatenation, $\mu(a,b)=ab$, satisfying:
\begin{enumerate}[label=(\roman*)]
    \item\label{def:graphedsemigroupoiditem1} for all $(a,b)\in \mathcal{S}^{(2)}$, $\so(ab)=\so(b)$ and $\ra(ab)=\ra(a)$ -- i.e., $(\id_{\mathcal{S}^{(0)}},\mu)$ is a graph morphism;
    \item\label{def:graphedsemigroupoiditem2} for all $(a,b,c)\in \mathcal{S}^{(3)}$, $(ab)c=a(bc)$.
\end{enumerate}
We will always assume that a graphed semigroupoid $\mathcal{S}$ satisfies $\mathcal{S}^{(0)}=\so(\mathcal{S})\cup\ra(\mathcal{S})$. Otherwise, restricting the vertex set to $\so(\mathcal{S})\cup\ra(\mathcal{S})$ yields another graphed semigroupoid with the same underlying arrow set and product.
\end{definition}

The next proposition shows that graphed semigroupoids are, in particular, Exel semigroupoids.

\begin{proposition}\label{prop:graphedsemigroupoidsaresemigroupoids}
Let $(\mathcal{S}^{(0)},\mathcal{S},\so,\ra)$ be a graph and $\mu\colon\mathcal{S}^{(2)}\to \mathcal{S}$, $\mu(a,b)=ab$, a fixed function. Consider the following assertions:
\begin{enumerate}[label=(\arabic*)]
    \item\label{prop:graphedsemigroupoidsaresemigroupoidsitem1} $(\mathcal{S}^{(0)},\mathcal{S},\so,\ra,\mu)$ is a graphed semigroupoid;
    \item\label{prop:graphedsemigroupoidsaresemigroupoidsitem2} $(\mathcal{S},\mathcal{S}^{(2)},\mu)$ is an Exel semigroupoid.
    \item\label{prop:graphedsemigroupoidsaresemigroupoidsitem3} for all $(a,b)\in \mathcal{S}^{(2)}$, $\so(ab)=\so(b)$ and $\ra(ab)=\ra(b)$.
\end{enumerate}
Then \ref{prop:graphedsemigroupoidsaresemigroupoidsitem1} is equivalent to \ref{prop:graphedsemigroupoidsaresemigroupoidsitem2}+\ref{prop:graphedsemigroupoidsaresemigroupoidsitem3} (that is, their logical conjunction).

If $\mathcal{S}$ has no sources nor sinks, then \ref{prop:graphedsemigroupoidsaresemigroupoidsitem2} implies \ref{prop:graphedsemigroupoidsaresemigroupoidsitem3}. (Thus \ref{prop:graphedsemigroupoidsaresemigroupoidsitem1} is equivalent to \ref{prop:graphedsemigroupoidsaresemigroupoidsitem2} in this case.)
\end{proposition}
\begin{proof}
    The implication \ref{prop:graphedsemigroupoidsaresemigroupoidsitem2}+\ref{prop:graphedsemigroupoidsaresemigroupoidsitem3}$\Rightarrow$\ref{prop:graphedsemigroupoidsaresemigroupoidsitem1} is trivial.
    
    The only nontrivial part of the implication \ref{prop:graphedsemigroupoidsaresemigroupoidsitem1}$\Rightarrow$\ref{prop:graphedsemigroupoidsaresemigroupoidsitem2}+\ref{prop:graphedsemigroupoidsaresemigroupoidsitem3} is the verification that items \ref{def:semigroupoiditem1}-\ref{def:semigroupoiditem3} of Definition \ref{def:semigroupoid} are equivalent: In other words, to compare when products $ab$, $bc$, $a(bc)$ and $(ab)c$ are defined.
    
    For example, assume that \ref{def:semigroupoid}\ref{def:semigroupoiditem1} is valid, i.e., $(a,b)\in\mathcal{S
    }^{(2)}$ and $(b,c)\in\mathcal{S
    }^{(2)}$. Then $\so(ab)=\so(b)=\ra(c)$, so $(ab,c)\in\mathcal{S
    }^{(2)}$, which means that \ref{def:semigroupoid}\ref{def:semigroupoiditem2} is valid. The other implications are proven similarly. Therefore \ref{prop:graphedsemigroupoidsaresemigroupoidsitem1} implies \ref{prop:graphedsemigroupoidsaresemigroupoidsitem2}+\ref{prop:graphedsemigroupoidsaresemigroupoidsitem3}.
    
    For the last part, assume that $(\mathcal{S
    },\mathcal{S
    }^{(2)},\mu)$ is an Exel semigroupoid and that the graph $\mathcal{S
    }$ has no sources nor sinks, and let us verify \ref{prop:graphedsemigroupoidsaresemigroupoidsitem3}. Let $(a,b)\in\mathcal{S
    }^{(2)}$. As the vertex $\so(b)$ is not a source of $\mathcal{S
    }$, choose $z\in\mathcal{S
    }$ such that $(b,z)\in\mathcal{S
    }^{(2)}$. Then $(ab,z)\in\mathcal{S
    }^{(2)}$, because $(\mathcal{S
    },\mathcal{S
    }^{(2)},\mu)$ is an Exel semigroupoid, which means that $\so(ab)=\ra(z)=\so(b)$. Similarly, $\mathcal{S
    }$ not having sinks implies that $\ra(ab)=\ra(a)$. This is precisely property \ref{prop:graphedsemigroupoidsaresemigroupoidsitem3}.\qedhere
\end{proof}

\begin{example}
Every semigroup $S$ may be regarded as a graphed semigroupoid with a singleton vertex set, $S^{(0)}=\left\{\ast\right\}$. Conversely, every graphed semigroupoid (or rather its arrow set) with singleton vertex set is a semigroup.
\end{example}

We will now characterize Exel semigroupoids which admit a compatible graph structure (Proposition \ref{prop:equivalencegraphableandcategoricalsemigroupoids}). These were already considered in \cite{MR2419901}. We will moreover classify all compatible graph structures to such an Exel semigroupoid (Proposition \ref{prop:classificationofgraphsoncategorical}).

\begin{definition}[{\cite[Definition 19.1]{MR2419901}}]
Let $\Lambda$ be an Exel semigroupoid. For every $a\in\Lambda$, we define
\[\Lambda^a=\left\{b\in\Lambda:(a,b)\in\Lambda^{[2]}\right\},\qquad\Lambda_a=\left\{b\in\Lambda:(b,a)\in\Lambda^{[2]}\right\}.\]
We say that $\Lambda$ is \emph{categorical} if for all $a,b\in\Lambda$, the sets $\Lambda^a$ and $\Lambda^b$ are either disjoint or equal.
\end{definition}

Note that $a\in\Lambda^b$ if and only if $b\in\Lambda_a$. Moreover, $\Lambda^{ab}=\Lambda^b$ and $\Lambda_{ab}=\Lambda_a$ for all $(a,b)\in\Lambda^{[2]}$.

\begin{proposition}\label{prop:equivalencecategorical}
The following are equivalent for an Exel semigroupoid $\Lambda$:
\begin{enumerate}[label=(\arabic*)]
  \item\label{prop:equivalencecategoricalitem1} $\Lambda$ is categorical.
  \item\label{prop:equivalencecategoricalitem2} For all $a,b\in\Lambda$, $\Lambda_a$ and $\Lambda_b$ are either disjoint or equal.
\end{enumerate}
\end{proposition}
\begin{proof}
  Assume that $\Lambda$ is categorical and let $a,b\in\Lambda$. We need to prove that if $\Lambda_a$ and $\Lambda_b$ are not disjoint, then they are equal. Suppose $p\in\Lambda_a\cap\Lambda_b$. By symmetry, it is sufficient to prove that $\Lambda_a\subseteq\Lambda_b$.
  
  Given $r\in\Lambda_a$, we have $a\in\Lambda^p\cap\Lambda^r$. Therefore, $\Lambda^p=\Lambda^r$, as $\Lambda$ is categorical. As $b\in\Lambda^p=\Lambda^r$, then $r\in\Lambda_b$. Therefore \ref{prop:equivalencecategoricalitem2} is valid.
  
  The implication \ref{prop:equivalencecategoricalitem2}$\Rightarrow$\ref{prop:equivalencecategoricalitem1} is completely analogous.\qedhere
\end{proof}

We will now prove that categorical semigroupoids are precisely those with a compatible graph structure. All that is needed is to define an appropriate vertex set and source and range maps. The main idea is that if $\mathcal{S}$ is a graphed semigroupoid and $(a,z)\in\mathcal{S
    }^{(2)}$, then for any other arrow $c\in\mathcal{S
    }$, we have $\so(c)=\so(a)$ if and only if $c\in\mathcal{S
    }_z$. This suggests us to identify $\so(a)$ with $\mathcal{S
    }_z$, and this is where the categorical property comes into play. If $\mathcal{S
    }^a=\varnothing$ (i.e., there is no $z$ such that $(a,z)\in\mathcal{S
    }^{(2)}$), then we add a ``dummy'' vertex as the source of $a$.

\begin{theorem}\label{prop:equivalencegraphableandcategoricalsemigroupoids}
Every graphed semigroupoid is categorical. Conversely, every categorical Exel semigroupoid $(\Lambda,\Lambda^{[2]},\mu)$ can be graphed, i.e., we may construct a graphed semigroupoid $(\Lambda^{(0)},\Lambda,\so,\ra,\mu)$, in such a way that $\Lambda^{[2]}=\Lambda^{(2)}$.
\end{theorem}
\begin{proof}
Suppose that $(\mathcal{S
    }^{(0)},\mathcal{S
    },\so,\ra,\mu)$ is a graphed semigroupoid and $a,b\in\mathcal{S
    }$. If $\mathcal{S
    }^a$ and $\mathcal{S
    }^b$ are not disjoint, take any $p\in\mathcal{S
    }^a\cap\mathcal{S
    }^b$. For all $q\in\mathcal{S
    }^a$, we have
\[\ra(q)=\so(a)=\ra(p)=\so(b)\]
thus $q\in\mathcal{S}^b$. This proves $\mathcal{S
    }^a\subseteq\mathcal{S
    }^b$, and the reverse inclusion is proven similarly. Hence $\mathcal{S
    }$ is categorical.

Conversely, assume that $(\Lambda,\Lambda^{[2]},\mu)$ is a categorical semigroupoid. We will first construct the vertex set of the underlying graph of $\Lambda$.

Consider two collections of symbols \[V_0=\left\{v_0(a):a\in\Lambda,\Lambda^a=\varnothing\right\}\qquad\text{and}\qquad V_1=\left\{v_1(a):a\in\Lambda,\Lambda_a=\varnothing\right\}.\]
Let $R_0$ and $R_1$ be any two equivalence relations on $V_0$ and $V_1$, respectively, satisfying, for all $(a,b)\in\Lambda^{[2]}$,
\[\text{if }\Lambda^b=\varnothing\text{ then }(v_0(b),v_0(ab))\in R_0\qquad\text{and}\qquad\text{if }\Lambda_a=\varnothing\text{ then }(v_1(a),v_1(ab))\in R_1.\ntag\label{eq:additionalpropertyofequivalencerelations}\]
We denote the $R_i$-equivalence class of a symbol $v_i(a)\in V_i$ as $[v_i(a)]$ (where $a\in\Lambda)$.

Define the vertex set $\Lambda_{R_0,R_1}^{(0)}$ as the disjoint union $\left\{\Lambda_a:a\in\Lambda,\Lambda_a\neq\varnothing\right\}\sqcup(V_0/R_0)\sqcup(V_1/R_1)$, and the source and range maps $\so_{R_0,R_1},\ra_{R_0,R_1}\colon\Lambda\to\Lambda_{R_0,R_1}^{(0)}$ as
\[\so_{R_0,R_1}(a)=\begin{cases}
\Lambda_z,&\text{if }z\text{ is any element of } \Lambda^a\\
[v_0(a)],&\text{if }\Lambda^a=\varnothing,
\end{cases}\qquad\text{and}\qquad \ra_{R_0,R_1}(a)=\begin{cases}
\Lambda_a,&\text{if }\Lambda_a\neq\varnothing,\\
[v_1(a)],&\text{if }\Lambda_a=\varnothing.
\end{cases}
\]
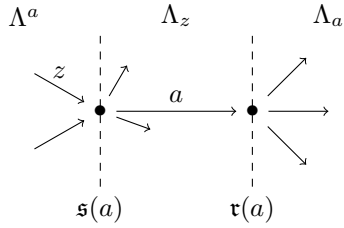
\[\begin{tikzpicture}
\node (Sa) at (0,0) {$\bullet$};
\node (Ra) at (2,0) {$\bullet$};
\draw[->] (Sa)--(Ra) node[above,align=center,midway] {$a$};
\draw[dashed] (Sa)+(-90:1) node[below] {$\so(a)$} -- +(90:1);
\draw[dashed] (Ra)+(-90:1) node[below] {$\ra(a)$}-- +(90:1);
\draw[->] (Sa)--+(60:0.7);
\draw[->] (Sa)--+(-20:0.7);
\node[above] at (1,1) {$\Lambda_z$};
\draw[->] (Ra) -- +(45:1);
\draw[->] (Ra) -- +(-45:1);
\draw[->] (Ra) -- +(0:1);
\draw[->] (Sa)+(150:1) -- (Sa) node[midway,above] {$z$};
\draw[->] (Sa)+(210:1) -- (Sa);
\node[above] at (3,1) {$\Lambda_a$};
\node[above] at (-1,1) {$\Lambda^a$};
\end{tikzpicture}
\]
We need to check that the source map is well-defined: if $\Lambda^a$ is nonempty and $z_1,z_2\in\Lambda^a$, then $a\in\Lambda_{z_1}\cap\Lambda_{z_2}$, which is therefore nonempty and thus $\Lambda_{z_1}=\Lambda_{z_2}$ because $\Lambda$ is categorical. Moreover, note that if $z\in\Lambda^a$, then $a\in\Lambda_z\neq\varnothing$, so $\Lambda_z\in\Lambda_{R_0,R_1}^{(0)}$.

This defines a graph structure on $\Lambda$. Since it depends on $R_0$ and $R_1$, we denote the set of $2$-paths as $\Lambda_{R_0,R_1}^{(2)}$. We need to prove that $\Lambda^{[2]}=\Lambda_{R_0,R_1}^{(2)}$. 

If $(a,b)\in\Lambda^{[2]}$, then $b\in\Lambda^a$, so $\so(a)=\Lambda_b$. Also, $a\in\Lambda_b$, which is nonempty and thus $\ra(b)=\Lambda_b=\so(a)$. This proves $\Lambda^{[2]}\subseteq\Lambda_{R_0,R_1}^{(2)}$.

Conversely, suppose that $(a,b)\in\Lambda_{R_0,R_1}^{(2)}$, that is, that $\so(a)=\ra(b)$. By the definitions of $\Lambda_{R_0,R_1}^{(2)}$ (as a disjoint union) and of the source and range maps, we necessarily have $\so(a)=\ra(b)\in\left\{\Lambda_z:z\in\Lambda\right\}$, that is, that $\so(a)=\Lambda_z$ for some $z\in\Lambda^a$, and that $\ra(b)=\Lambda_b$. Then
\[a\in\Lambda_z=\so(a)=\ra(b)=\Lambda_b,\]
which means that $(a,b)\in\Lambda^{[2]}$.

Finally, since $\Lambda^{ab}=\Lambda^b$ and $\Lambda_{ab}=\Lambda_a$ for all $(a,b)\in\Lambda^{(2)}$, we obtain $\ra(ab)=\ra(a)$ and $\so(ab)=\so(b)$ (this is where Equation \eqref{eq:additionalpropertyofequivalencerelations} is necessary). By Proposition \ref{prop:graphedsemigroupoidsaresemigroupoids}, $(\Lambda_{R_0,R_1}^{(0)},\Lambda,\so,\ra,\mu)$ is a graphed semigroupoid.\qedhere
\end{proof}

The construction above, in fact, yields \emph{all} the compatible graph structures on $\Lambda$, as we now prove. Let $V_0$ and $V_1$ be the two sets considered in the proof above.

\begin{proposition}\label{prop:classificationofgraphsoncategorical}
Let $\Lambda$ be a categorical Exel semigroupoid, endowed with a graph structure which makes it a graphed semigroupoid $(\Lambda^{(0)},\Lambda,\so,\ra,\mu)$. Suppose, moreover, that $\Lambda^{(0)}=\so(\Lambda)\cup\ra(\Lambda)$.

Then there exist unique equivalence relations $R_0$ and $R_1$ on $V_0$ and $V_1$, respectively, and a bijection $I\colon\Lambda_{R_0,R_1}^{(0)}\to\Lambda^{(0)}$ such that $(I,\id_{\Lambda})$ is a graph isomorphism.
\end{proposition}
\begin{proof}
Consider the sets \[W_0=\left\{\so(a):\Lambda^a=\varnothing\right\},\qquad W_1=\left\{\ra(a):\Lambda_a=\varnothing\right\}\qquad\text{and}\qquad W=\Lambda^{(0)}\setminus(W_0\cup W_1).\]
Note that $W_0$ and $W_1$ are disjoint, since otherwise there would be $a,b\in\Lambda$ such that $\so(a)=\ra(b)$ and $\Lambda_b=\varnothing$. However, this would imply that $(a,b)\in\Lambda^{(2)}$, so $a\in\Lambda_b$, a contradiction.

First we need to construct the equivalence relations $R_0$ and $R_1$ as in Proposition \ref{prop:equivalencegraphableandcategoricalsemigroupoids}. Consider the functions $F_0\colon V_0\to W_0$ and $F_1\colon V_1\to W_1$ given by
\[F_0(v_0(a))=\so(a)\qquad\text{and}\qquad F_1(v_1(b))=\ra(b)\]
for all $a,b\in\Lambda$ such that $\Lambda^a=\Lambda_b=\varnothing$. Note that these maps are surjective. Consider the equivalence relations $R_i=\ker F_i$ ($i=1,2$), that is,
\[R_i=\left\{(v_i(a),v_i(b)):F_i(v_i(a))=F_i(v_i(b))\right\}.\]
Then Equation \eqref{eq:additionalpropertyofequivalencerelations} is clearly satisfied, by the definition of $F_0$ and $F_1$ and because $\so$ and $\ra$ are compatible with the semigroupoid structure of $\Lambda$ (see Definition \ref{def:graphedsemigroupoid}\ref{def:graphedsemigroupoiditem1}).

Therefore the maps $F_i$ factor through the quotient to bijections $V_i/R_i\to W_i$. This gives us a bijection
\[I\colon(V_0/R_0)\sqcup(V_1/R_1)\to W_0\cup W_1,\qquad I([v_i(a)])=F_i(v_i(a)).\]
We need to extend $I$ to a bijection from $\Lambda_{R_0,R_1}^{(0)}$ to $\Lambda^{(0)}$, that is, we need to define $I(\Lambda_a)$ when $\Lambda_a\neq\varnothing$ in order to obtain a bijection from $\left\{\Lambda_a:\Lambda_a\neq\varnothing\right\}$ to $W$.

If $\Lambda_a\neq\varnothing$, we let $I(\Lambda_a)=\ra(a)$. In order to prove that $I$ is well-defined, suppose $\Lambda_{a_1}=\Lambda_{a_2}\neq\varnothing$. Choose any $z\in\Lambda_{a_1}=\Lambda_{a_2}$, so $(z,a_i)\in\Lambda^{(2)}$, that is, $\ra(a_1)=\so(z)=\ra(a_2)$, so $I(\Lambda_a)$ is uniquely defined.

Now let us prove that $I(\Lambda_a)\in W$ whenever $\Lambda_a\neq\varnothing$, or equivalently that $I(\Lambda_a)\not\in W_0\cup W_1$.
\begin{itemize}
    \item If $\ra(a)=\so(b)$ then $a\in\Lambda^b$ and in particular $\Lambda^b\neq\varnothing$. This proves that $I(\Lambda_a)\not\in W_0$;
    \item If $\ra(a)=\ra(b)$, then $\Lambda_a=\Lambda_b$. As $\Lambda_a\neq\varnothing$, this proves that $I(\Lambda_a)\not\in W_1$.
\end{itemize}
Therefore $I(\Lambda_a)\in\Lambda^{(0)}\setminus(W_0\cup W_1)=W$.

Let us now prove that every $v\in W$ is of the form $I(\Lambda_a)$ for some $a\in\Lambda$ with $\Lambda_a\neq\varnothing$. As we assume that $\Lambda^{(0)}=\so(\Lambda)\cup\ra(\Lambda)$, we have two possibilities:
\begin{itemize}
    \item If $v=\so(b)$ for some $b$, then $\Lambda^b\neq\varnothing$, as $v\not\in W_0$, so we may take any $a\in\Lambda^b$. We have $b\in\Lambda_a\neq\varnothing$, and $I(\Lambda_a)=\ra(a)=\so(b)=v$.
    \item If $v=\ra(a)$ for some $a$, then $\Lambda_a\neq\varnothing$ as $v\not\in W_1$, and so $v=I(\Lambda_a)$.
\end{itemize}

We have thus obtained a surjection $I\colon\left\{\Lambda_a:\Lambda_a\neq\varnothing\right\}\to W$. We still need to prove that $I$ is injective on this set. Suppose $I(\Lambda_a)=I(\Lambda_b)$, that is, $\ra(a)=\ra(b)$, where $\Lambda_a,\Lambda_b\neq\varnothing$. Choose any $z\in\Lambda_a$. Then $\so(z)=\ra(a)=\ra(b)$, so $z\in\Lambda_a\cap\Lambda_b$, and therefore $\Lambda_a=\Lambda_b$ as $\Lambda$ is categorical.

We have, therefore, a well-defined bijection $I\colon\Lambda_{R_0,R_1}^{(0)}\to\Lambda^{(0)}$. We are done if we verify that $\so(a)=I(\so_{R_0,R_1}(a))$ and $\ra(a)=I(\ra_{R_0,R_1}(a))$ for all $a\in\Lambda$. Let $a\in\Lambda$. If $\Lambda^a=\varnothing$, then
\[I(\so_{R_0,R_1}(a))=I([v_0(a)])=F_0(v_0(a))=\so(a).\]
If $\Lambda^a\neq\varnothing$, then we choose $z\in\Lambda^a$, so that $a\in\Lambda_z\neq\varnothing$, and
\[I(\so_{R_0,R_1}(a))=I(\Lambda_z)=\ra(z)=\so(a).\]

This proves that $I\circ\so_{R_0,R_1}=\so$, and similarly $I\circ\ra_{R_0,R_1}=\ra$. Therefore $(I,\id_\Lambda)$ is a graph isomorphism.

Suppose now that $R_0'$ and $R_1'$ are any other equivalence relations on $V_0$ and $V_1$ and $J\colon\Lambda_{R_0',R_1'}^{(0)}\to\Lambda^{(0)}$ is any other bijection for which $(J,\id_{\Lambda})$ is a graph isomorphism. We denote the equivalence classes of $V_i/R_i'$ as $[v_i(a)]'$. For each $v_0(a),v_0(b)\in V_0$, where $\Lambda^a=\Lambda^b=\varnothing$, we have
\[\begin{array}{c c c c c}
    v_0(a)R_0'v_0(b)&\iff&[v_0(a)]'=[v_0(b)]'&\iff& J([v_0(a)]')=J([v_0(b)]')\\
    &\iff& J(\so_{R_0',R_1'}(a))=J(\so_{R_0',R_1'}(b))&\iff& \so(a)=\so(b)\\
    &\iff& F_0(v_0(a))=F_0(v_0(b))&\iff& v_0(a)R_0 v_0(b),
\end{array}\]
where we use the facts that $J$ is injective, the definition of the source map $\so_{R_0',R_1'}$, the fact that $(J,\id_{\Lambda})$ is a graph isomorphism, and the definitions of $F_0$ and of $R_0=\ker F_0$. Therefore $R_0'=R_0$. Similarly, $R_1'=R_1$.\qedhere
\end{proof}

\begin{example}
Let $\Lambda=\left\{a,b,x,y,z\right\}$, with product defined by
\[aa=a,\quad xa=y,\quad ya=y,\quad bb=b, \quad xb=z,\quad zb=z\]
Then $\Lambda$ is a non-categorical semigroupoid, since $\Lambda_a=\left\{x,y\right\}$ and $\Lambda_b=\left\{x,z\right\}$ are different but not disjoint. Thus $\Lambda$ cannot be fully realized as a graphed semigroupoid.
\end{example}

Note that every Exel semigroupoid $\Lambda$ may be identified, in an injective and homomorphic manner, as a subset of some semigroup $S$: Namely, the power set $S=2^\Lambda$ of $\Lambda$ is a semigroup under the product of sets, and the map $\phi\colon\Lambda\to S$, $\phi(a)=\left\{a\right\}$, is an injective semigroupoid homomorphisms. However, $\phi(\Lambda)$ is not a sub-semigroup(oid) of $S$ if $\Lambda$ is not a semigroup itself.

We finish this introduction by providing a condition that allows us to ``extend'' semigroupoid homomorphisms between graphed semigroupoids to graph homomorphisms.

\begin{proposition}\label{prop:vertexmap}
Suppose that $G$ and $H$ are graphs, that $G$ has no sources nor sinks, and that $\phi\colon G\to H$ is a map such that $(\phi\times\phi)(G^{(2)})\subseteq H^{(2)}$. Then there exists a unique ``vertex map'' $\phi^{(0)}\colon G^{(0)}\to H^{(0)}$ such that $(\phi^{(0)},\phi)$ is a graph morphism.
\end{proposition}
\begin{proof}
Uniqueness of such $\phi^{(0)}$ is immediate, since we require that $\so_H\circ\phi=\phi^{(0)}\circ\so_G$, and $\so_G$ is surjective as $G$ has no sinks. This same equation yields us the only possible formula for $\phi^{(0)}$, viz.\ $\phi^{(0)}(v)=\so_H(\phi(a))$ where $a\in G$ is any arrow with $\so_G(a)=v$. Thus we need to verify that $\so_G(a)=\so_G(b)$, where $a,b\in G$, implies $\so_H(\phi(a))=\so_H(\phi(b))$. As $\so_G(a)$ is not a source, there exists $z\in G$ with $\ra_G(z)=\so_G(a)=\so_G(b)$, that is $(a,z)$ and $(b,z)\in G^{(2)}$. Then $(\phi(a),\phi(z))$ and $(\phi(b),\phi(z))\in H^{(2)}$, which means that
\[\so_H(\phi(a))=\ra_H(\phi(z))=\so_H(\phi(b)).\]
Therefore, we obtain a unique function $\phi^{(0)}\colon G^{(0)}\to H^{(0)}$ satisfying $\so_H\circ\phi=\phi^{(0)}\circ\so_G$. The proof that $\ra_H\circ\phi=\phi^{(0)}\circ\ra_G$ follows a similar argument as in the paragraph above, but using instead the fact that $G$ has no sinks.\qedhere
\end{proof}

\begin{corollary}\label{cor:graphedhomomorphisminducesvertexmap}
Suppose that $\mathcal{S}$ and $\mathcal{T}$ are graphed semigroupoids, $\phi\colon\mathcal{S}\to\mathcal{T}$ is a semigroupoid homomorphism, and that $\mathcal{S}$ has no sources nor sinks. Then there exists a unique ``vertex map'' $\phi^{(0)}\colon\mathcal{S}^{(0)}\to\mathcal{T}$ such that $(\phi^{(0)},\phi)$ is a graph homomorphism.
\end{corollary}

\begin{example}
The proposition above is not valid if we just assume that $\mathcal{S}$ has no sinks (or no sources). Let us associate a semigroupoid $A$ to the \emph{strict} order of $\mathbb{N}$ in the same manner as we associate categories to (non-strict) orders -- namely $A=\left\{(n,m)\in\mathbb{N}\times\mathbb{N}:m<n\right\}$, with product $(n,m)(m,k)=(n,k)$, vertex set $A^{(0)}=\mathbb{N}$ and source and range maps $\so(n,m)=m$, $\ra(n,m)=n$.

Now let $\mathcal{S}_1=A\sqcup A$ consist of two distinct copies of $A$, so $\mathcal{S}_1$ is a graphed semigroupoid over $\mathcal{S}_1^{(0)}=\mathbb{N}\sqcup\mathbb{N}$ (the product of $\mathcal{S}_1$ is defined only for elements in the same copy of $A$). We let $\mathcal{S}_2^{(0)}$ be the set obtained from $\mathcal{S}_1^{(0)}=\mathbb{N}\sqcup\mathbb{N}$ by identifying both copies of $0$, and let $\mathcal{S}_2$ be the graphed semigroupoid obtained from $\mathcal{S}_1$ by composing the source and range maps with the canonical quotient map $\mathcal{S}_1^{(0)}\to\mathcal{S}_2^{(0)}$.
\[
\begin{tikzpicture}
\node (S2) at (0,0) {$\mathcal{S}_2$:};
\node (S20) at ([shift={+(0.5,0)}]S2) {$0$};
\node (S211) at ([shift={+(1,0.5)}]S20) {$1$};
\node (S212) at ([shift={+(1,-0.5)}]S20) {$1$};
\node (S221) at ([shift={+(1,0)}]S211) {$2$};
\node (S222) at ([shift={+(1,0)}]S212) {$2$};
\node (S231) at ([shift={+(1,0)}]S221) {$\cdots$};
\node (S232) at ([shift={+(1,0)}]S222) {$\cdots$};
\draw[->] (S20)--(S211);
\draw[->] (S211)--(S221);
\draw[->] (S221)--(S231);
\draw[->] (S20)--(S212);
\draw[->] (S212)--(S222);
\draw[->] (S222)--(S232);

\node (S1) at ([shift={+(-6,0)}]S2) {$\mathcal{S}_1$:};
\node (S101) at ([shift={+(0.5,0.5)}]S1) {$0$};
\node (S102) at ([shift={+(0.5,-0.5)}]S1) {$0$};
\node (S111) at ([shift={+(1,0)}]S101) {$1$};
\node (S112) at ([shift={+(1,0)}]S102) {$1$};
\node (S121) at ([shift={+(1,0)}]S111) {$2$};
\node (S122) at ([shift={+(1,0)}]S112) {$2$};
\node (S131) at ([shift={+(1,0)}]S121) {$\cdots$};
\node (S132) at ([shift={+(1,0)}]S122) {$\cdots$};
\draw[->] (S101)--(S111);
\draw[->] (S111)--(S121);
\draw[->] (S121)--(S131);
\draw[->] (S102)--(S112);
\draw[->] (S112)--(S122);
\draw[->] (S122)--(S132);
\end{tikzpicture}
\]
Then neither $\mathcal{S}_1$ nor $\mathcal{S}_2$ have sinks, but the identity map $\mathcal{S}_2\to\mathcal{S}_1$ is an Exel semigroupoid isomorphism which cannot be extended to a graph homomorphism.
\end{example}

    \subsection{Inverse semigroupoids}
    The main goal of this subsection if to define inverse semigroupoids and to extend the basic elements of the theory of inverse semigroups to this more general setting.

\begin{definition}\label{def:inversesemigroupoid}
An Exel semigroupoid $\Lambda$ is \emph{regular} if for every $a\in\Lambda$ there exists $b\in\Lambda$ such that $(a,b),(b,a)\in\Lambda^{[2]}$ and
\[aba=a\qquad\text{and}\qquad bab=b.\label{eq:inversesemigroupoid}\]
Such an element $b$ is called an \emph{inverse} of $a$.

If every element $a\in\Lambda$ admits a unique inverse, then $\Lambda$ is called an \emph{inverse semigroupoid}. In this case, the unique inverse of $a$ is denoted $a^*$.
\end{definition}

All groupoids and all inverse semigroups are inverse semigroupoids.

Just as in the case of semigroups, the condition for regular semigroupoids can be weakened to the following: For every $a\in\Lambda$, there exists $b\in\Lambda$ such that $aba=a$. Such an element $b$ is called a \emph{pseudoinverse} of $a$. It follows that $bab$ is an inverse of $a$, and therefore $\Lambda$ is regular.

\begin{example}
Suppose that $\Lambda$ is an inverse graphed semigroupoid, endowed with some compatible graph structure. Then for every $x\in\Lambda^{(0)}$, the \emph{isotropy semigroup} $\Lambda_x^x=\ra^{-1}(x)\cap\so^{-1}(x)$ is in fact an inverse semigroup, with the product induced by $\Lambda$.
\end{example}

Given an Exel semigroupoid $\Lambda$, denote by $E(\Lambda)=\left\{e\in \Lambda:(e,e)\in\Lambda^{[2]}\text{ and }ee=e\right\}$ the set of \emph{idempotents} of $\Lambda$. Note that if $t$ is an inverse of $s$, then $st$ and $ts$ are idempotents.

We will now proceed to prove that every inverse semigroupoid admits a unique compatible graph structure (Corollary \ref{cor:inversesemigroupoidhasauniquegraphstructure}).

\begin{lemma}\label{lem:exelinverseiscategorical}
Let $\Lambda$ be an inverse semigroupoid. The following are equivalent:
\begin{enumerate}[label=(\arabic*)]
\item\label{lem:exelinverseiscategorical1} $\Lambda$ is categorical;
\item\label{lem:exelinverseiscategorical2} For all $e,f\in E(\Lambda)$, if $ef$ is defined, then $fe$ is also defined and $ef=fe$.
\end{enumerate}
\end{lemma}
\begin{proof}
  \ref{lem:exelinverseiscategorical1}$\Rightarrow$\ref{lem:exelinverseiscategorical2}: Suppose that $\Lambda$ is categorical, and that $ef$ is defined, where $e,f\in E(\Lambda)$. We consider any compatible graph structure on $\Lambda$. Then $\so(f)=\ra(f)=\so(e)=\ra(e)$, so $fe$ is defined. Let $x=\so(e)$. Then $e$ and $f$ belong to the isotropy semigroup $\Lambda_x^x$ and are idempotents, therefore $ef=fe$ (\cite[Theorem 5.1.1]{MR1455373}).
  
  \ref{lem:exelinverseiscategorical2}$\Rightarrow$\ref{lem:exelinverseiscategorical1}: Assume \ref{lem:exelinverseiscategorical2} holds. We will construct an explicit compatible graph structure on $\Lambda$, although similar arguments may be used to prove directly that $\Lambda$ is categorical. We define the following equivalence relation $\sim$ on $E(\Lambda)$:
  \[e\sim f\iff ef\text{ is defined}.\]
  The only nontrivial part about $\sim$ being an equivalence relation is transitivity. If $ef$ and $fg$ are defined ($e,f,g\in E(\Lambda)$), then $efg=(ef)(fg)=(fe)(gf)$, and in particular $eg$ is defined. Thus $\sim$ is transitive.
  
  Denote the $\sim$-class of $e\in E(\Lambda)$ by $[e]$. Let $\Lambda^{(0)}=E(\Lambda)/\!\!\sim$, and define a graph structure on $\Lambda$ by setting
  \[\so(a)=[a^*a],\qquad \ra(a)=[aa^*].\]
  Since a product $ab$ is defined if and only if $a^*abb^*$ is defined, i.e., if and only if $\so(a)=\ra(b)$, we have $\Lambda^{[2]}=\Lambda^{(2)}$, as necessary. We may also compute, for all $(a,b)\in\Lambda^{(2)}$,
  \[(ab)(ab)^*=(aa^*a)b(ab)^*=(aa^*)(ab)(ab)^*,\]
  so $\ra(ab)=\ra(a)$. Similarly, $\so(ab)=\so(b)$. By Proposition \ref{prop:graphedsemigroupoidsaresemigroupoids}, $(\Lambda^{(0)},\Lambda,\so,\ra,\mu)$ is a graphed semigroupoid.\qedhere
  \end{proof}

The set $\Lambda^{(0)}$ constructed above is the \emph{space of germs} of $E(\Lambda)$ with its canonical order. This construction will be further developed in Subsection \ref{subsec:groupoidofgerms}.

\begin{theorem}\label{theo:everyinversesemigroupoidiscategorical}
Every inverse semigroupoid $\Lambda$ is categorical. Equivalently, if $e,f\in E(\Lambda)$ and $ef$ is defined, then $fe$ is also defined and $ef=fe$.
\end{theorem}
\begin{proof}
  We will prove that $\Lambda$ satisfies condition \ref{lem:exelinverseiscategorical2} of Lemma \ref{lem:exelinverseiscategorical}, which follows the same arguments as in the case of inverse semigroups (see \cite[Proposition 5.1.1]{MR1455373}), as long as we make sure that all products involved are defined. We include the details for completeness. Suppose that $e,f\in E(\Lambda)$, and that $(e,f)\in \Lambda^{[2]}$. Let $x=(ef)^*$. Then
  \[(ef)x(ef)=ef\qquad\text{and}\qquad x(ef)x=x\]
  In particular, the products $fx$ and $xe$ are defined, so $fxe$ is defined. Let us prove that $fxe=x$, by proving that $fxe$ is an inverse of $ef$: Since $ff=f$ and $ee=e$ are defined, then we may compute
  \[(fxe)(ef)(fxe)=fx(e^2)(f^2)xe=fxefxe=fxe\ntag\label{eq:exelinverseiscategorical}\]
  and similarly $(ef)(fxe)(ef)=ef$. Thus $fxe$ is the inverse of $ef$, i.e., $fxe=x$. Moreover, it follows from the last equality of \eqref{eq:exelinverseiscategorical} that
  \[x=fxe=fxefxe=x^2\]
  so $x\in E(\Lambda)$, and in particular $x=x^*=ef$ is idempotent.
  
  Since $fx=f(ef)$ is defined, then $fe$ is defined, and the same argument as above (changing the roles of $e$ and $f$) proves that $fe\in E(\Lambda)$. To finish, we use the equalities $x=fxe=ef$ to obtain
  \[fe=fefe=fxe=x=ef.\qedhere\]
\end{proof}

Therefore, every inverse semigroupoid admits a compatible graph structure. Suppose $\mathcal{S}$ is a graphed inverse semigroupoid. Recall that we assume that $\mathcal{S}^{(0)}=\so(\mathcal{S})\cup\ra(\mathcal{S})$. Given $a\in \mathcal{S}$, $a^*a$ and $aa^*$ are defined, so $\so(a)=\ra(a^*)$ and $\ra(a)=\so(a^*)$. It follows that $\mathcal{S}$ has no sources nor sinks. From Corollary \ref{cor:graphedhomomorphisminducesvertexmap} we may conclude that the compatible graph structure is unique.

\begin{corollary}\label{cor:inversesemigroupoidhasauniquegraphstructure}
If $\Lambda$ is an inverse semigroupoid, then there exists a unique compatible graph structure on $\Lambda$.
\end{corollary}

We will thus always regard inverse semigroupoids as graphed semigroupoids, with their unique compatible graph structure, and for every inverse semigroupoid homomorphism $\phi\colon\mathcal{S}\to\mathcal{T}$, we denote $\phi^{(0)}\colon\mathcal{S}^{(0)}\to\mathcal{T}^{(0)}$ the unique map for which $(\phi^{(0)},\phi)$ is a graph homomorphism, as in Corollary \ref{cor:graphedhomomorphisminducesvertexmap}.

Let $\mathcal{S}$ be an inverse semigroupoid. We will now list the remaining algebraic properties of $\mathcal{S}$ which will be used in the remainder of this article.

\begin{definition}
The \emph{canonical order} of an inverse semigroup $\mathcal{S}$ is the relation $\leq$ defined as
\[a\leq b\iff a=ba^*a.\]
(In particular, we require that $\so(a)=\so(b)$ and $\ra(a)=\ra(b)$.)
\end{definition}

Note that inverse semigroupoid homomorphisms preserve the order. 

If a product $ab$ is defined in the inverse semigroupoid $\mathcal{S}$, then $\ra(a^*)=\so(a)=\ra(b)=\so(b^*)$, so $b^*a^*$ is also defined. Using this and the commutativity of $E(\mathcal{S})$, the properties below can be proven by the same computations as in the case of inverse semigroups. See \cite[Chapter 5]{MR1455373} for details.

\begin{proposition}\label{prop:propertiesoperationsofinversesemigroupoid}
Let $\mathcal{S}$ be an inverse semigroupoid. Then
\begin{enumerate}[label=(\alph*)]
    \item\label{prop:propertiesoperationsofinversesemigroupoid1} $(ab)^*=b^*a^*$, for all $(a,b)\in\mathcal{S}^{(2)}$.
    \item\label{prop:propertiesoperationsofinversesemigroupoid2} If $e\in E(\mathcal{S})$ and $(b,e)\in \mathcal{S}^{(2)}$. then $beb^*\in E(\mathcal{S})$.
    \item\label{prop:propertiesoperationsofinversesemigroupoid3} the following are equivalent:
    \begin{enumerate}[label=(\alph{enumi}.\roman*)]
        \item\label{prop:propertiesoperationsofinversesemigroupoid31} $a\leq b$;
        \item\label{prop:propertiesoperationsofinversesemigroupoid32} there exists $e\in E(\mathcal{S})$ such that $a=be$;
        \item\label{prop:propertiesoperationsofinversesemigroupoid33} there exists $f\in E(\mathcal{S})$ such that $a=fb$;
        \item\label{prop:propertiesoperationsofinversesemigroupoid34} $a^*\leq b^*$.
    \end{enumerate}
    \item\label{prop:propertiesoperationsofinversesemigroupoid4} $\leq$ is a partial order on $\mathcal{S}$;
    \item\label{prop:propertiesoperationsofinversesemigroupoid5} If $a\leq b$, $c\leq d$ and $(a,c)\in \mathcal{S}^{(2)}$, then $ac\leq bd$.
\end{enumerate}
\end{proposition}

We may determine subclasses of the class of semigroupoids algebraically as follows: An inverse semigroupoid $\mathcal{S}$ is a
\begin{itemize}
    \item \emph{semigroup} if and only if any two idempotents may be multiplied (and in this case it is an inverse semigroup).
    \item \emph{groupoid} if the product of two idempotents is defined if and only if they are equal; Alternatively, $\mathcal{S}$ is a groupoid if and only if the canonical order is equality.
    \item \emph{group} if it is both a semigroup and a groupoid, or equivalently if it has a unique idempotent.
\end{itemize}
    
    \subsection{A representation theorem}
    Cayley's theorem states that every group can be represented as a permutation group on some set. More generally, the Vagner-Preston theorem states that every inverse semigroup can be represented as a semigroup of partial bijections on some set. In this subsection, we will further generalize this to represent semigroupoids as certain semigroupoids of partial bijections.

Representation results have been obtained in the context of \emph{inverse categories}, i.e., (possibly large) categories which behave as inverse semigroupoids. Namely, it is proven that every (locally small) inverse category may be faithfully embedded into the category $\cat{PInj}$ of sets and partial bijections. For more details, see \cite[p.\ 87]{MR0506554}, \cite[Theorem 3.8]{MR1871071}, \cite[Proposition 3.11]{MR3093088}. However we adopt a fundamentally different perspective from those works, by considering inverse semigroupoids as innate to bundles. This approach has the advantage of making the dynamical nature of inverse semigroupoids more explicit, and in particular to naturally motivating a natural notion of \emph{action} for inverse semigroupoids.

Let $\pi\colon X\to X^{(0)}$ be a bundle. A \emph{partial bijection between fibers} of $\pi$ is a triple $(y,f,x)$, where $x,y\in X^{(0)}$ and $f\colon\dom(f)\to\ran(f)$ is a bijection with $\dom(f)\subseteq\pi^{-1}(x)$ and $\ran(f)\subseteq\pi^{-1}(y)$ (in other words, $f$ is a \emph{partial bijection} from $\pi^{-1}(x)$ to $\pi^{-1}(y)$). We denote by $\mathcal{I}(\pi)$ the set of all partial bijections between fibers of $\pi$.

We make $\mathcal{I}(\pi)$ into a graph over $X^{(0)}$ by setting
\[\so(y,f,x)=x\qquad\text{and}\qquad\ra(y,f,x)=y\]
and then induce a product structure on $\mathcal{I}(\pi)$ as
\[(z,g,y)(y,f,x)=(z,g\circ f,x)\]
where $g\circ f$ is the usual partial composition of functions
\[g\circ f\colon f^{-1}(\ran(f)\cap\dom(g))\to g(\dom(g)\cap\ran(f)),\qquad (g\circ f)(x)=g(f(x)).\]

The verification that this products makes $\mathcal{I}(\pi)$ into an inverse semigroupoid is straightforward. The inverse of an element $(y,f,x)$ is $(x,f^{-1},y)$, where $f^{-1}\colon\ran(f)\to\dom(f)$ is the inverse function of $f$. The idempotent set of $\mathcal{I}(\pi)$ is
\[E(\mathcal{I}(\pi))=\left\{(x,\id_A,x):x\in X^{(0)},A\subseteq\pi^{-1}(x)\right\}.\]

In particular cases, this construction leads to well-known examples of semigroups and groupoids.

\begin{example}
Suppose that $X$ is a set, seen as a bundle over a singleton set $X^{(0)}=\left\{\ast\right\}$, i.e., we consider the bundle $\pi\colon X\to\left\{\ast\right\}$, $\pi_X(x)=\ast$ for all $x\in X$. Then $\mathcal{I}(X)\defeq\mathcal{I}(\pi)$ is simply the inverse semigroup of partial bijections of $X$.
\end{example}

\begin{example}
Consider the identity function $\id_X$ of a set $X$. Let $L_2=\left\{0,1\right\}$ be the lattice with two elements $0<1$, which is an inverse semigroup under meets. Let $X\times X$ be the transitive equivalence relation on $X$, seen as a groupoid. Then $\mathcal{I}(\id_X)$ is isomorphic to the product inverse semigroupoid $(X\times X)\times L_2$ (where the product semigroupoid structure is defined in the obvious manner). Namely, to an element $(y,f,x)$ of $\mathcal{I}(\id_X)$ we associate the element $(y,x,0)$ of $(X\times X)\times L_2$ if $f=\varnothing$, the empty function, and $(y,x,1)$ otherwise.

Note that the equivalence relation $X\times X$ is isomorphic to the subsemigroupoid of maximal elements (with respect to the canonical order) of $(X\times X)\times L_2$, or to the initial groupoid of $\mathcal{I}(\id_X)$ (see Subsection \ref{subsec:groupoidofgerms}).
\end{example}

We now state our representation theorem for inverse semigroupoids.

\begin{theorem}\label{thm:representationtheoremoid}
Let $\mathcal{S}$ be an inverse semigroupoid. Then $\mathcal{S}$ is isomorphic to a sub-inverse semigroupoid of $\mathcal{I}(\pi)$ for some bundle $\pi$.
\end{theorem}
\begin{proof}
Considering the range map $\ra\colon\mathcal{S}\to \mathcal{S}^{(0)}$, we will define an embedding $\alpha\colon\mathcal{S}\to\mathcal{I}(\ra)$. Namely, given $a\in\mathcal{S}$, $\alpha(a)$ is a triple of the form $(y,\alpha_a,x)$, where $x,y\in\mathcal{S}^{(0)}$ and $\alpha_a$ is a partial bijection from $\ra^{-1}(x)$ to $\ra^{-1}(y)$. We first describe the map $\alpha_a$.

For every $a\in\mathcal{S}$, let $D_a=\left\{t\in \mathcal{S}:tt^*\leq aa^*\right\}$. If $t\in D_{a^*}$, then
\[(at)(at)^*=att^*a^*\leq aa^*,\]
so $at\in D_a$. Therefore we may define $\alpha_a\colon D_{a^*}\to D_a$ as $\alpha_a(t)=at$. It is a bijection, since if $t\in D_{a^*}$, then $\alpha_a(t)\in D_a$, so we may apply $\alpha_{a^*}$ and obtain
\[\alpha_{a^*}(\alpha_a(t))=\alpha_{a^*}(at)=a^*at=t,\]
because $tt^*\leq a^*a$. This proves that $\alpha_{a^*}\circ\alpha_a=\id_{D_{a^*}}$. Changing the roles of $a^*$ and $a$ we conclude that $\alpha_a$ is invertible, with $\alpha_a^{-1}=\alpha_a^*$.

Moreover, we have $D_a\subseteq\ra^{-1}(\ra(a))$ and thus we may define $\alpha\colon \mathcal{S}\to\mathcal{I}(\ra)$ by
\[\alpha(a)=(\ra(a),\alpha_a,\so(a)).\]
Notice that, for $a,b\in\mathcal{S}$,
\[ab\text{ is defined}\iff\so(a)=\ra(b)\iff\so(\alpha(a))=\ra(\alpha(b))\iff\alpha(a)\alpha(b)\text{ is defined}.\ntag\label{eq:equivalenceproductdefinedivagnerprestonoid}\]
Using the equivalence above and the same arguments as in the proof of the Vagner-Preston theorem (see \cite[Proposition 2.1.3]{MR1724106}, for example), it follows that $\alpha$ is an injective semigroupoid homomorphism. In fact, the right-to-left implications of \ref{eq:equivalenceproductdefinedivagnerprestonoid} imply that the image of $\alpha$ is a sub-semigroupoid of $\mathcal{I}(\ra)$, onto which $\alpha$ is an isomorphism.\qedhere
\end{proof}

Still considering the map $\alpha$ of the proof above, note that the vertex set of $\mathcal{I}(\ra)$ is $\mathcal{S}^{(0)}$, the vertex set of $\mathcal{S}$. We have $\so\circ\alpha=\so$, so the ``vertex map'' (as in Corollary \ref{cor:graphedhomomorphisminducesvertexmap}) associated to $\alpha$ is the identity of $\mathcal{S}^{(0)}$. This motivates the definition of an \emph{action} of an inverse semigroupoid $\mathcal{S}$ on a set $X$.

\begin{definition}\label{def:action}
A \emph{global action} of an inverse semigroupoid $\mathcal{S}$ on a set $X$ consists of a map $\pi\colon X\to\mathcal{S}^{(0)}$, called the \emph{anchor map} and a semigroupoid homomorphism $\theta\colon\mathcal{S}\to\mathcal{I}(\pi)$ such that the vertex map $\theta^{(0)}
\colon\mathcal{S}^{(0)}\to\mathcal{I}(\pi)^{(0)}=\mathcal{S}^{(0)}$ is the identity $\id_{\mathcal{S}^{(0)}}$.

More explicitly, it consists of a map $\pi\colon X\to\mathcal{S}^{(0)}$ and a family $\left\{\theta_a:a\in\mathcal{S}\right\}$ of bijections $\theta_a\colon\dom(\theta_a)\to\ran(\theta_a)$ such that
\begin{enumerate}[label=(\roman*)]
    \item $\dom(\theta_a)\subseteq\pi^{-1}(\so(a))$ and $\ran(\theta_a)\subseteq\pi^{-1}(\ra(a))$ for all $a\in\mathcal{S}$;
    \item $\theta_{a}\circ\theta_b=\theta_{ab}$ for all $(a,b)\in\mathcal{S}^{(2)}$.
\end{enumerate}
\end{definition}

However, as stated in the introduction, we will be interested in more general notions than actions, which is the content of the next subsection.
    
    \subsection{\texorpdfstring{$\land$}{∧}-prehomomorphisms, partial homomorphisms, and actions}\label{sec:dualprehomomorphisms}
    
    Generalizations of homomorphisms, initially called $\lor$ and $\land$-\emph{prehomomorphisms}, were respectively introduced in \cite[Definition 1.1]{MR0424979} and \cite[Definition 4.1]{MR0470123} by McAlister and Reilly on their study of $E$-unitary inverse semigroups. We will focus on $\land$-prehomomorphisms, which are appropriate for the constructions of semidirect products.

\begin{remark}
    The terminology ``prehomomorphism'' has been used to describe  $\lor$-prehomomorphisms in \cite[p.\ 80]{MR1694900}, but used to describe $\land$-prehomomorphisms in \cite[VI.7.2]{MR752899} and \cite{Mikola2017}. In order to avoid confusion, we will use the original terminology of ``$\land$-prehomomorphism''.
\end{remark}

A more specific case of $\land$-prehomomorphisms are \emph{partial homomorphisms}, whose study was initiated by Exel in \cite{MR1276163}, in order to describe the structure of C*-algebras endowed with an action of the circle. This turned out to be a rich area of research, with several applications in the theory of topological dynamical systems and C*-algebras (see \cite{MR2799098,MR3699795}). Let us mention that partial homomorphisms of general inverse semigroups were first defined in \cite{MR3231479}, while partial actions of groupoids (on rings) were defined in \cite{MR2982887}.

We will now define $\land$-prehomomorphisms and partial homomorphisms in the context of semigroupoids, connecting in a precise manner all the notions described above. From them, we will construct semidirect products in a manner which generalizes semidirect products of groupoids (\cite[Exercise 11.5.1]{MR2273730}), transformation groupoids of partial group actions (\cite{MR2045419}), semidirect products of inverse semigroups acting on semilattices (see \cite[Lemma 1.1]{MR0357660.1} and \cite[VI.7.6-7]{MR752899}). Similarly, groupoids of germs of (partial) inverse semigroup actions are simply the underlying groupoids of the associated semidirect product (see Section \ref{sec:categoricalconstructions}).

\begin{definition}\label{def:dualprehomomorphismandpartialhomomorphism}
Let $\mathcal{S}$ and $\mathcal{T}$ be inverse semigroupoids and $\theta\colon\mathcal{S}\to\mathcal{T}$ a function. Consider the following statements:
\begin{enumerate}[label=(\roman*)]
\item\label{def:dualprehomomorphismandpartialhomomorphism1} For all $a\in\mathcal{S}$ we have $\theta(a^*)=\theta(a)^*$;
\item\label{def:dualprehomomorphismandpartialhomomorphism2} For all $a,b\in\mathcal{S}^{(2)}$, we have $(\theta(a),\theta(b))\in\mathcal{T}^{(2)}$ and $\theta(a)\theta(b)\leq\theta(ab)$;
\item\label{def:dualprehomomorphismandpartialhomomorphism3} If $a\leq b$ in $\mathcal{S}$, then $\theta(a)\leq\theta(b)$ in $\mathcal{T}$.
\end{enumerate}
If $\theta$ satisfies \ref{def:dualprehomomorphismandpartialhomomorphism1}-\ref{def:dualprehomomorphismandpartialhomomorphism2}, then $\theta$ is called a \emph{$\land$-prehomomorphism}. If $\theta$ satisfies all of \ref{def:dualprehomomorphismandpartialhomomorphism1}-\ref{def:dualprehomomorphismandpartialhomomorphism3}, then $\theta$ is called a \emph{partial homomorphism}.
\end{definition}.

\begin{example}
Let $\mathcal{S}$ be any inverse semigroupoid, $X$ any set, and consider $2^X$, the power set of $X$, as a semigroup under intersection.

If $X$ has sufficiently large cardinality (e.g.\ $|X|\geq|\mathcal{S}|$), then there is a map $\theta\colon\mathcal{S}\to 2^X\setminus\left\{\varnothing\right\}$ such that $\theta(a)\cap\theta(b)=\varnothing$ if $a\neq b$. Then $\theta$ is a $\land$-prehomomorphism, but not a partial homomorphism as long as $\mathcal{S}$ is not a groupoid.
\end{example}

If $\theta\colon\mathcal{S}\to\mathcal{T}$ is a $\land$-prehomomorphism, then $\theta$ takes idempotents to idempotents. Indeed, if $e\in E(\mathcal{S})$, then $\theta(e)^*\theta(e)=\theta(e^*)\theta(e)\leq\theta(e^*e)=\theta(e)$. Applying the definition of the order of $\mathcal{T}$ implies $\theta(e)=\theta(e)^*\theta(e)$, which is idempotent.

Following \cite[p.\ 12]{MR1694900}, two elements $a,b\in\mathcal{S}$ are \emph{compatible} if $a^*b$ and $ab^*$ are idempotent. Any $\land$-prehomomorphism $\theta\colon\mathcal{S}\to\mathcal{T}$ takes compatible elements to compatible elements, if $a,b\in\mathcal{S}$ are compatible, then $\theta(a)^*\theta(b)\leq\theta(a^*b)$, which is idempotent. Similarly, $\theta(a)\theta(b)^*$ is idempotent.

The statements \ref{def:dualprehomomorphismandpartialhomomorphism}.\ref{def:dualprehomomorphismandpartialhomomorphism1}-\ref{def:dualprehomomorphismandpartialhomomorphism3} are independent, that is, there exist maps $\theta\colon\mathcal{S}\to\mathcal{T}$ between inverse semigroupoids which satisfy any two of them but not the third one. In fact, we can obtain such examples for inverse semigroups, which we present for the sake of reference.

\begin{example}
Let $G=\left\{1,g\right\}$ be the cyclic group of order $2$ and $X=\left\{a,b\right\}$. Define $\theta\colon G\to\mathcal{I}(X)$ as $\theta_1=\id_X$ and \[\theta_g\colon\left\{a\right\}\to\left\{b\right\},\qquad\theta_g(a)=b.\]
Then $\theta$ satisfies \ref{def:dualprehomomorphismandpartialhomomorphism}\ref{def:dualprehomomorphismandpartialhomomorphism2} and \ref{def:dualprehomomorphismandpartialhomomorphism3}, but not \ref{def:dualprehomomorphismandpartialhomomorphism1}.
\end{example}

\begin{example}
Let $G=\left\{1\right\}$ be the trivial group, and $H=\left\{1,g\right\}$ the cyclic group of order $2$. The map $\theta\colon G\to H$, $1\mapsto g$ satisfies  \ref{def:dualprehomomorphismandpartialhomomorphism}\ref{def:dualprehomomorphismandpartialhomomorphism1} and \ref{def:dualprehomomorphismandpartialhomomorphism3}, but not \ref{def:dualprehomomorphismandpartialhomomorphism2}.
\end{example}

\begin{example}\label{ex:landprehomomorphismnotpartialaction}
Let $L_2=\left\{0,1\right\}$ be the lattice with two elements, $0<1$, and define $\theta\colon L_2\to L_2$ as $\theta(a)=b$ and $\theta(b)=\theta(a)$. Then \ref{def:dualprehomomorphismandpartialhomomorphism}\ref{def:dualprehomomorphismandpartialhomomorphism1} and \ref{def:dualprehomomorphismandpartialhomomorphism2} are satisfied, but \ref{def:dualprehomomorphismandpartialhomomorphism3} is not (i.e., $\theta$ is a $\land$-prehomomorphisms but not a partial homomorphism).
\end{example}

We should remark that $\land$-prehomomorphisms are not stable under composition.

\begin{example}
Let $S=\left\{a,b\right\}$ and $\theta$ be as in Example \ref{ex:landprehomomorphismnotpartialaction}. Let $G=\left\{1,g\right\}$ be the group with two elements and $\eta\colon G\to S$ be given by $\eta(1)=b$, $\eta(g)=a$. Then $\theta$ is a $\land$-prehomomorphism and $\eta$ is a partial homomorphism, however $\theta\circ\eta$ is not a $\land$-prehomomorphism since
\[\theta(\eta(g))\theta(\eta(g))=\theta(a)\theta(a)=bb=b\qquad\text{and}\qquad\theta(\eta(gg))=\theta(\eta(1))=\theta(b)=a.\]
\end{example}

On the other hand, it is easy to verify that if $\eta$ is a $\land$-prehomomorphism and $\theta$ is a partial homomorphism, then $\theta\circ\eta$ is a $\land$-prehomomorphism, and also that the composition of two partial homomorphisms is a partial homomorphism.

We may now define $\land$-preactions and partial actions. In particular from property \ref{def:dualprehomomorphismandpartialhomomorphism}\ref{def:dualprehomomorphismandpartialhomomorphism2}, we may apply Proposition \ref{prop:vertexmap} to obtain a unique map $\phi^{(0)}\colon\mathcal{S}^{(0)}\to\mathcal{T}^{(0)}$ such that $(\phi^{(0)},\phi)$ is a graph homomorphism.

\begin{definition}\label{def:partialactiononset}
    A \emph{$\land$-preaction} of an inverse semigroupoid $\mathcal{S}$ on a set $X$ consists of an \emph{anchor map} $\pi\colon X\to\mathcal{S}^{(0)}$ and a $\land$-prehomomorphism $\theta\colon \mathcal{S}\to\mathcal{I}(\pi)$ such that the vertex map $\theta^{(0)}$ is the identity of $\mathcal{S}^{(0)}$. If $\theta$ is a partial homomorphism, we call $(\pi,\theta)$ a \emph{partial action}. We say that $(\pi,\theta)$ is \emph{non-degenerate} if $X=\bigcup_{a\in\mathcal{S}}\dom(\theta_a)$.
    
    We use the notation $(\pi,\theta)\colon\mathcal{S}\curvearrowright X$ to denote a $\land$-preaction of $\mathcal{S}$ on $X$.
\end{definition}

As $\land$-prehomomorphisms take idempotents to idempotents, then whenever $(\pi,\theta)\colon\mathcal{S}\curvearrowright X$ is a $\land$-preaction and $e\in E(\mathcal{S})$, the map $\theta_e$ is the identity of its domain. Moreover, if $a\in\mathcal{S}$, then
\[\theta_a=\theta_a\theta_a^{-1}\theta_a\leq\theta_a\theta_{a^*a}\]
which implies that $\dom(\theta_a)\subseteq\dom(\theta_{a^*a})$.

If $a\leq b$ and $x\in\dom(a^*a)\cap\dom(b)$, then we may compute
\[\theta_b(x)=\theta_b(\theta_{a^*a}(x))=\theta_{ba^*a}(x)=\theta_a(x).\ntag\label{eq:changepreactionbygreater}\]
Note, however, that we do not necessarily have $\dom(\theta_a)\subseteq\dom(\theta_b)$ (unless $(\pi,\theta)$ is a partial action), so we need some care with these computations.

These notions extend, simultaneously, those of (global) actions and partial actions which have commonly appeared throughout the literature.

\begin{example}\label{ex:actionofsemigroup}
	Let $S$ be an inverse semigroup, so $S^{(0)}=\left\{\ast\right\}$ is a singleton. Given any set $X$, there is a unique map $\pi\colon X\to S^{(0)}$. Therefore, partial and global actions of $S$ on $X$, as an inverse semigroupoid, are the same as partial and global actions of $S$ on $X$ in the usual sense of inverse semigroups (see e.g.\ \cite[Definition 5.1]{MR2565546} and \cite[Definition 3.3]{MR3231479}).
\end{example}

\begin{example}\label{ex:groupoidnondegenerateaction}
	If $\G$ is a groupoid, then a non-degenerate global action $(\pi,\theta)\colon\G\curvearrowright X$ is the same as a groupoid action as in \cite[Definition 3.6]{MR2969047} or \cite[Exercise 11.2.1]{MR2273730}, i.e., it satisfies $\dom(\theta_a)=\pi^{-1}(a)$ for all $a\in\mathcal{S}$.
	
	Indeed, as $\theta$ is a global action then $\dom(\theta_a)=\dom(\theta_{a^*a})=\dom(\theta_{\so(a)})$ for all $a\in\mathcal{S}$. That is, $\dom(\theta_a)$ depends only on $\so(a)$, so if $\so(a)=\so(b)$ then $\dom(\theta_a)=\dom(\theta_b)$. Non-degeneracy of $(\pi,\theta)$ means precisely that for all $a\in A$, $\pi^{-1}(\so(a))=\bigcup_{b\in\so^{-1}(\so(a))}\dom(\theta_b)=\dom(\theta_a)$, as we wanted.
\end{example}

In order to construct semidirect products, we will need to consider $\land$-preactions which preserve the structure of semigroupoids. This notion will be used to connect the topological and algebraic settings (i.e., when considering appropriate actions of semigroupoids on topological spaces and on algebras).

\begin{definition}
A \emph{left ideal} of a semigroupoid $\Lambda$ is a subset $I\subseteq\Lambda$ such that $\Lambda I\subseteq I$. Right ideals are defined similarly. An \emph{ideal} is a subset of $\Lambda$ which is simultaneously a left and a right ideal.
\end{definition}

\begin{example}
Let $\G$ be a groupoid. The map $I\mapsto\so(I)$ is an order isomorphism from the set of ideals of $\G$ to the set of invariant subsets of $\G[0]$ (a subset $A\subseteq\G[0]$ is \emph{invariant} if $A=\ra(\so^{-1}(A))$).
\end{example}

Every set $X$ is regarded as a groupoid in the trivial manner: $X^{(0)}=X$, and for all $x\in X$, $\so(x)=\ra(x)=x$, and the product is defined as $xx=x$. The groupoids constructed in this manner are called \emph{unit groupoids}. Ideals of $X$ are precisely its subsets.

If $\Lambda$ is an Exel semigroupoid, a homomorphism $\phi\colon\Lambda\to X$ is a function such that, if $(a,b)\in\Lambda^{[2]}$, then $\phi(a)=\phi(b)$. In particular, functions between sets are the same as their (semi)groupoid homomorphisms.

In the definition below, we regard the vertex set $\mathcal{S}^{(0)}$ as a unit groupoid. This is an adaptation of \cite[p. 3660]{MR2982887}

\begin{definition}\label{def:partialactiononsemigroupoid}
    A $\land$-preaction (resp.\ partial, global action) $(\pi,\theta$) of an inverse semigroupoid $\mathcal{S}$ on a semigroupoid $\Lambda$ is a $\land$-preaction (resp.\ partial, global action) of $\mathcal{S}$ on $\Lambda$, as a set, which further satisfies:
    \begin{enumerate}[label=(\roman*)]
        \item\label{def:partialactiononsemigroupoid1} The anchor map $\pi\colon\Lambda\to\mathcal{S}^{(0)}$ is a homomorphism.
        \item\label{def:partialactiononsemigroupoid2} For every $x\in\mathcal{S}^{(0)}$, $\pi^{-1}(x)$ is an ideal of $\Lambda$;
        \item\label{def:partialactiononsemigroupoid3} For every $a\in\Lambda$, $\dom(\theta_a)$ is an ideal of $\pi^{-1}(\so(a))$;
        \item\label{def:partialactiononsemigroupoid4} For every $a\in\mathcal{S}$, $\theta_a$ is a semigroupoid isomorphism from $\dom(\theta_a)$ to $\ran(\theta_a)$.
    \end{enumerate}
\end{definition}

\begin{remark}
\begin{enumerate}
    \item Alternatively to \ref{def:partialactiononsemigroupoid2} and \ref{def:partialactiononsemigroupoid3}, we could adopt the simpler (and slightly stronger) assumption that $\dom(\theta_a)$ is an ideal of $\Lambda$ for all $a\in\mathcal{S}$. This is the case, for example, when all such ideals $\dom(\theta_a)$ are \emph{idempotent} (see Definition \ref{def:idempotentsemigroupoid}).
    \item If $\Lambda$ is a set/unit groupoid, then this definition coincides with \ref{def:partialactiononset}, since the additional properties of Definition \ref{def:partialactiononsemigroupoid} become trivial. Accordingly, by a $\land$-preaction we will always mean one in the sense of Definition \ref{def:partialactiononsemigroupoid}.
\end{enumerate}
\end{remark}

As we will be more interested in \emph{continuous $\land$-preactions of topological semigroupoids}, we refer to Subsection \ref{subsec:continuouslandpreactions} for examples. In fact, every $\land$-preaction $(\pi,\theta)\colon\mathcal{S}\curvearrowright\Lambda$ may be ``extended'' to a partial action, in the sense that partial action $(\pi,\overline{\theta})\colon\mathcal{S}\curvearrowright\Lambda$ (with same anchor map) such that $\theta_a\leq\overline{\theta}_a$ for all $a\in\mathcal{S}$. See Proposition \ref{prop:extensionoflandpreactiontopartialaction}

Let us finish this subsection by proving that any $\land$-preaction may be ``extended'' to a partial action.

\begin{proposition}\label{prop:extensionoflandpreactiontopartialaction}
Let $(\pi,\theta)\colon\mathcal{S}\curvearrowright\Lambda$ be a $\land$-preaction. Then there exists a partial action $(\pi,\overline{\theta})\colon\mathcal{S}\curvearrowright\Lambda$ such that $\theta_a\leq\overline{\theta}_a$ for all $a\in\mathcal{S}$.
\end{proposition}
\begin{proof}
    As the only (possible) problem with is that $\theta$ does not necessarily satisfy $\theta_b\leq\theta_a$ when $b\leq a$, we simply ``glue'' $\theta_b$ to $\theta_a$ in this case.
    
    More precisely: Let $a\in\mathcal{S}$. If $b_1,b_2\leq a$. Then $b_1$ and $b_2$ are compatible, so $\theta_{b_1}$ and $\theta_{b_2}$ are compatible (see the paragraphs after Definition \ref{def:dualprehomomorphismandpartialhomomorphism}), which means that $\theta_{b_1}$ and $\theta_{b_2}$ coincide on the intersection of their domains, as do $\theta_{b_1}^{-1}$ and $\theta_{b_2}^{-1}$.
    
    Thus we set $\overline{\theta}_a=\bigvee_{b\leq a}\theta_b$, the join taken in $\mathcal{I}(\Lambda)$ i.e., $\overline{\theta}_a(x)$ is defined if and only if $\theta_b(x)$ is defined for some $b\leq a$, in which case $\overline{\theta}_a(x)=\theta_b(x)$. This defines an isomorphism between the ideals $\bigcup_{b\leq a}\dom(\theta_b)$ of $\pi^{-1}(\so(a))$ and $\bigcup_{b\leq a}\ran(\theta_b)$ of $\pi^{-1}(\ra(a))$. It is clear that $\overline{\theta}_{a^*}=\overline{\theta}_a^{-1}$, and that $\overline{\theta}$ is order preserving, and the verification that it is a $\land$-preaction, and thus a partial action, is straightforward.\qedhere
\end{proof}

Note that the partial action $(\pi,\overline{\theta})$ constructed above is minimal, in the sense that any other partial action $(\pi,\gamma)$ extending $(\pi,\theta)$ satisfies $\overline{\theta}_a\leq\gamma_a$ for all $a\in\mathcal{S}$.

\subsection{Semidirect products}

Throughout this subsection, we fix a $\land$-preaction $(\pi,\theta)\colon\mathcal{S}\curvearrowright \Lambda$. Consider the set
\[\mathcal{S}\ltimes\Lambda=\left\{(a,x)\in\mathcal{S}\times X:x\in\dom(\theta_a)\right\}.\]
We will endow $\mathcal{S}\ltimes\Lambda$ with a product, and call the resulting semigroupoid the \emph{semidirect product} of $(\pi,\theta)$. If $(a,x)$ and $(b,y)\in\mathcal{S}\ltimes\Lambda$, $ab$ and $x\theta_b(y)$ are defined (in $\mathcal{S}$ and $\Lambda$, respectively), we set
\[(a,x)(b,y)=(ab,\theta_{b^*}(x\theta_b(y))).\ntag\label{eq:semidirectproduct}\]
(Compare this formula with the product on partial crossed product algebras; see \cite{arxiv1804.00396,MR1456588,MR1331978}.)

To check that the right-hand side of \eqref{eq:semidirectproduct} is a well-defined element of $\mathcal{S}\ltimes\Lambda$, first note that $\theta_b(y)$ belongs to the ideal $\dom(\theta_{b^*})$ of $\Lambda$, so $\theta_{b^*}(x\theta_b(y))$ is an element of $\dom(\theta_b)$. We have
\[\theta_b(\theta_{b^*}(x\theta_b(y)))=x\theta_b(y).\]
Since $x$ belongs to the ideal $\dom(\theta_a)$, then $\theta_{b^*}(x\theta_b(y))$ belongs to $\dom(\theta_a\circ\theta_b)$, which is contained in $\dom(\theta_{ab})$ because $\theta$ is a $\land$-prehomomorphism. Therefore $(ab,\theta_{b^*}(x\theta_b(y)))\in\mathcal{S}\ltimes\Lambda$.

However, the product \eqref{eq:semidirectproduct} is not associative in general, even when considering global actions of inverse semigroups on semigroups. The example below is a slight modification of \cite[Example 3.5]{MR2115083}.

\begin{example}
Let $G=\left\{1,g\right\}$ be the cyclic group of order $2$, and $S$ the inverse semigroup obtained by adjoining a new unit $x$ to $G$.

Let $T=\left\{0,t,u,v\right\}$ be the semigroup with product defined by
\[tv=vt=u\qquad\text{and}\qquad ab=0\quad\text{if}\quad(a,b)\not\in\left\{(t,v),(v,t)\right\}.\]
Consider the ideal $I=\left\{0,u,v\right\}$ of $T$, and define an action $\theta\colon S\curvearrowright T$ as $\theta_x=\id_T$, $\quad\theta_1=\id_I$ and $\theta_g\colon I\to I$ as
\[\theta_g(0)=0,\qquad\theta_g(u)=v\qquad\text{and}\qquad\theta_g(v)=u.\]

Then $(1,t)((g,u)(1,t))=(g,0)$, but $((1,t)(g,u))(1,t)=(g,u)$.
\end{example}

\begin{definition}
	We call $\mathcal{S}\ltimes\Lambda$ with the product \eqref{eq:semidirectproduct} the \emph{semidirect product semigroupoid} (induced by the $\land$-preaction $(\pi,\theta)$), whenever the product \eqref{eq:semidirectproduct} is associative.
\end{definition}

We will therefore need to consider conditions on the $\land$-preaction $\theta\colon\mathcal{S}\curvearrowright\Lambda$ which make the product in \eqref{eq:semidirectproduct} associative, so that we obtain a semigroupoid structure on $\mathcal{S}\ltimes\Lambda$. The same ideas as in \cite[Section 3]{MR2115083}, which lie in the context of partial actions of groups on algebras, may be easily adapted to the context of semigroupoid actions and show that the product \eqref{eq:semidirectproduct} is in fact associative for a large class of inverse semigroupoid actions. See the paragraph after Proposition \ref{prop:idempotentornondegenerateimpliesLRassociative}

\begin{definition}
A \emph{multiplier} of a semigroupoid $\Lambda$ is a pair $(L,R)$, of partially defined maps $L$ and $R$ on $\Lambda$ satisfying
\begin{enumerate}[label=(\roman*)]
\item $\dom(L)$ is a right ideal of $\Lambda$ and $\dom(R)$ is a left ideal of $\Lambda$;
\item $L(ab)=L(a)b$;
\item $R(ab)=aR(b)$;
\item $R(a)b=aL(b)$,
\end{enumerate}
in the sense that each side of these equations is defined if and only if the other side is defined, in which case they coincide. $L$ and $R$ are called, respectively, a \emph{left} and a \emph{right multipliers}.
\end{definition}

\begin{example}\label{ex:multipliersofmultiplication}
Given $x\in\Lambda$, we let $L_x\colon\Lambda^x\to x\Lambda^x$ and $R_x\colon \Lambda_x\to\Lambda_xx$ be given by $L_x(a)=xa$ and $R_x(a)=ax$. Then $(L_x,R_x)$ is a multiplier. If $I$ is an ideal of $\Lambda$, then the restrictions of $L_x$ and $R_x$ to $I$ form a multiplier of $I$.
\end{example}

\begin{definition}
	$\Lambda$ is \emph{non-degenerate} if the map $a\mapsto (L_a,R_a)$ is injective.
\end{definition}

Moreover, the associative property of $\Lambda$ may be rewritten in terms of multipliers as follows: For all $a,b\in\Lambda$,
\[\dom(L_a)\cap\dom(R_b)\subseteq\dom(L_a\circ R_b)\qquad\text{and}\qquad L_a\circ R_b=R_b\circ L_a.\]

\begin{definition}
A semigroupoid $\Lambda$ is \emph{$(L,R)$-associative} if for any two multipliers $(L,R)$ and $(L',R')$ of $\Lambda$, we have
\[\dom(L)\cap\dom(R')\subseteq\dom(L\circ R')\qquad\text{and}\qquad L\circ R'=R'\circ L\]
\end{definition}

\begin{theorem}\label{thm:semidirectproductisassociativeifLRassociative}
The product of $\mathcal{S}\ltimes\Lambda$ is associative if $\dom(\theta_a)$ is $(L,R)$-associative for all $a\in\mathcal{S}$.
\end{theorem}
\begin{proof}
Let $(a,r),(b,s),(c,t)\in \mathcal{S}\ltimes\Lambda$. As long as it makes sense, we compute
\[\left((a,r)(b,s)\right)(c,t)=\left(ab,\theta_{b}^{-1}(r\theta_{b}(s))\right)(c,t)=(abc,\theta_{c}^{-1}(\theta_{b}^{-1}(r\theta_{b}(s))\theta_{c}(t)))\]
and
\[(a,r)\left((b,s),(c,t)\right)=(a,r)\left(bc,\theta_{c}^{-1}(s\theta_{c}(t))\right)=(abc,\theta_{bc}^{-1}(r\theta_{bc}(\theta_{c}^{-1}(s\theta_{c}(t)))))\]
Thus, we need to prove that for all $(a,b,c)\in\mathcal{S}^{(3)}$,
\[\theta_{c}^{-1}(\theta_{b}^{-1}(r\theta_{b}(s))\theta_{c}(t))=\theta_{bc}^{-1}(r\theta_{bc}(\theta_{c}^{-1}(s\theta_{c}(t))))\]
in the sense that either side is defined if and only if the other one is, and in which case the equation holds. As $\theta_c(t)$ runs through all of $\ran(\theta_c)$, let us rewrite it as $\theta_c(t)=w$. Let us first rewrite the right-hand side. We have
\[\theta_{bc}(\theta_c^{-1}(sw)=\theta_{bcc^*}(sw)\]
and the element $sw$ belongs to the domains of both $\theta_{bcc^*}$ and of $\theta_b$. As $bcc^*\leq b$ then $\theta_{bcc^*}(sw)=\theta_b(sw)$ (see Equation \eqref{eq:changepreactionbygreater}), thus we need instead to prove
\[\theta_{c}^{-1}(\theta_{b}^{-1}(r\theta_{b}(s))w)=\theta_{bc}^{-1}(r\theta_b(sw))\ntag\label{eq:equationwhichdeterminesassociativity1}.\]

To further simplify the expression above we would need to rewrite the term ``$\theta_{bc}^{-1}$'', at the beggining of the right-hand side above, as ``$\theta_c^{-1}\circ\theta_b^{-1}$'', however this is not so immediate. If this is the case, Equation \eqref{eq:equationwhichdeterminesassociativity1} implies
\[\theta_{b}^{-1}(r\theta_{b}(s))w=\theta_{b}^{-1}(r\theta_{b}(sw)).\ntag\label{eq:equationwhichdeterminesassociativity2}\]

Let us postpone the proof of the equivalence \eqref{eq:equationwhichdeterminesassociativity1}$\iff$\eqref{eq:equationwhichdeterminesassociativity2} and instead finish the proof of the theorem. Then \eqref{eq:equationwhichdeterminesassociativity2} may be rewritten as
\[R_{w}\circ \theta_{b}^{-1}\circ L_{r}\circ\theta_{b}=\theta_{b}^{-1}\circ L_{r}\circ \theta_{b}\circ R_{w}\ntag\label{eq:prop:semidirectproductisassociativeifLRassociative}\]
on $\dom(\theta_{b})$. Here we regard $L_r$ as a left multiplier of $\ran(\theta_b)$, as in Example \ref{ex:multipliersofmultiplication}, thus $\theta_{b}^{-1}\circ L_{r}\circ\theta_{b}$ is a left multiplier of $\dom(\theta_{b})$. Similarly, we regard $R_w$ as a right multiplier of $\dom(\theta_{b})$. As we assume that $\dom(\theta_{b})$ is $(L,R)$-associative, then Equation \eqref{eq:prop:semidirectproductisassociativeifLRassociative} holds.

In particular, this also proves that $((a,r)(b,s))(c,t)$ is defined if and only if $(a,r)((b,s)(c,t))$ is defined, which is equivalence \ref{def:semigroupoiditem2}$\iff$\ref{def:semigroupoiditem3} of Definition \ref{def:semigroupoid}. We still need to prove that these terms are defined when both $(a,r)(b,s)$ and $(b,s)(c,t)$ are defined..

Indeed, in this case, $\theta_b^{-1}(r\theta_b(s))$ and $s\theta_c(t)$ are defined, which means that
\[s\in\dom(\theta_b^{-1}\circ L_r\circ\theta_b)\cap\dom(R_{\theta_c(t)})\subseteq\dom(R_{\theta_c(t)}\circ(\theta_b^{-1}\circ  L_r\circ\theta_{b^{-1}}))\]
because $\theta_b^{-1}\circ L_r\circ\theta_b$ and $R_{\theta_c(t)}$ are left and a right multipliers on $\dom(\theta_b)$, respectively. Then
\[\theta_b^{-1}(r\theta_b(s))\theta_c(t)\]
is defined, thus $((a,r)(b,s))(c,t)$ is also defined.\qedhere
\end{proof}

\begin{proof}[Proof of the equivalence of Equations \eqref{eq:equationwhichdeterminesassociativity1} and \eqref{eq:equationwhichdeterminesassociativity2}]
First assume that \eqref{eq:equationwhichdeterminesassociativity1} holds. We may apply $\theta_c$ on the left-hand side, and hence on the right-hand side as well, to obtain
\[\theta_b^{-1}(r\theta_b(s))w=\theta_c(\theta_{(bc)^*}(r\theta_b(sw)))=\theta_{cc^*b^*}(r\theta_b(sw))=\theta_b^{-1}(r\theta_b(sw)),\]
where the last equality follows from the fact that $cc^*b^*\leq b^*$ and $r\theta_b(sw)$ belongs to $\dom(\theta_{cc^*b^*})\cap\dom(\theta_b^*)$. This is precisely \eqref{eq:equationwhichdeterminesassociativity2}.

In the other direction, note that the left-hand side of \eqref{eq:equationwhichdeterminesassociativity2} is $\theta_b^{-1}(r\theta_b(s))w$, and $w\in\ran(\theta_c)=\dom(\theta_{c^*})$. Thus we may apply $\theta_c^{-1}$ on both sides and use $\theta_c^{-1}\circ\theta_b^{-1}\leq\theta_{bc}^{-1}$ to obtain \eqref{eq:equationwhichdeterminesassociativity1}.\qedhere
\end{proof}

Proposition \ref{prop:idempotentornondegenerateimpliesLRassociative} below yields a large class of $(L,R)$-associative semigroupoids.

\begin{definition}\label{def:idempotentsemigroupoid}
	$\Lambda$ is \emph{idempotent} if $\Lambda=\Lambda\Lambda$, that is, if every $a\in\Lambda$ may be rewritten as a product $a=a_1a_2$, where $a_1,a_2\in\Lambda$.
\end{definition}

\begin{proposition}\label{prop:idempotentornondegenerateimpliesLRassociative}
A semigroupoid $\Lambda$ is $(L,R)$-associative if any of the following two conditions below holds.
\begin{enumerate}[label=(\roman*)]
	\item\label{prop:idempotentornondegenerateimpliesLRassociative1} $\Lambda$ is idempotent;
	\item\label{prop:idempotentornondegenerateimpliesLRassociative2} $\Lambda$ is non-degenerate, and for every $a\in\Lambda$ we have $\Lambda_a\neq\varnothing$.
	\item\label{prop:idempotentornondegenerateimpliesLRassociative3} $\Lambda$ is non-degenerate, and for every $a\in\Lambda$ we have $\Lambda^a\neq\varnothing$.
\end{enumerate}
\end{proposition}
\begin{proof}
Let $(L,R)$ and $(L',R')$ be multipliers of $\Lambda$.

First assume that $a=a_1a_2$ in $\Lambda$. Then
\[L(R'(a))=L(R'(a_1a_2))=L(a_1R'(a_2))=L(a_1)R'(a_2)=R'(L(a_1)a_2)=R'(L(a_1a_2))=R'(L(a)).\]
Thus if $\Lambda$ is idempotent, then $L\circ R'=R'\circ L$ as well.

Alternatively, for all $a,b$ we have
\[L(R'(a))b=L(R'(a)b)=L(aL'(b))=L(a)L'(b)=R'(L(a))b,\]
and
\[bL(R'(a))=R(b)R'(a)=R'(R(b)a)=R'(bL(a))=bR'(L(a)),\]
in the sense that, in each of these chains of equalities, one of the terms is defined if and only if the other ones are, in which case they coincide. If $\Lambda$ is non-degenerate, this implies that $L(R'(a))=R'(L(a))$ for all $a\in\Lambda$, i.e., $L\circ R'=R'\circ L$.

It remains only to prove that, under either of the hypotheses \ref{prop:idempotentornondegenerateimpliesLRassociative1} or \ref{prop:idempotentornondegenerateimpliesLRassociative2} above, we have $\dom(L)\cap\dom(R')\subseteq\dom(L\circ R')$. Fix $a\in\dom(L)\cap\dom(R')$.

First assume that $\Lambda$ is idempotent, and write $a=a_1a_2$. Then we may compute
\[L(a)=L(a_1a_2)=L(a_1)a_2\qquad\text{and}\qquad R'(a)=R'(a_1a_2)=a_1R'(a_2).\]
In particular, $a_1\in\dom(L)$, which is a right ideal and thus $a_1R'(a_2)\in\dom(L)$, so we may compute
\[L(a_1R'(a_2))=L(R'(a_1a_2))=L(R'(a)),\]
and so $a\in\dom(L\circ R')$.

Now assume that condition \ref{prop:idempotentornondegenerateimpliesLRassociative2} holds, and take $x\in\Lambda_{L(a)}$, so $xL(a)=R(x)a$ is defined. Since $a$ belongs to the left ideal $\dom(R')$, we may compute
\[R'(R(x)a)=R(x)R'(a)=xL(R'(a)),\]
and in particular $a\in\dom(L\circ R')$. The proof that condition \ref{prop:idempotentornondegenerateimpliesLRassociative3} implies $(L,R)$-associativity is similar.\qedhere
\end{proof}

We could also adapt the terminology ``$s$-unital'' from ring theory (see \cite{MR0419511}) to the setting of semigroupoids: A semigroupoid $\Lambda$ is \emph{left $s$-unital} if for every $t\in\Lambda$ there exists $u\in\Lambda_t$ such that $ut=t$. Of course every left $s$-unital semigroupoid is idempotent (and also non-degenerate). Examples of left $s$-unital semigroupoids include regular semigroupoids and categories -- in particular sets, groupoids, regular semigroups, multiplicative semigroups of $s$-unital rings, monoids, lattices etc\ldots

Note that if $\Lambda$ is a regular semigroupoid, $I$ is an ideal of $\Lambda$, $x\in I$ and $y$ is an inverse of $x$ in $\Lambda$, then $y=yxy\in I$. It follows that $I$ itself is a regular semigroupoid. In particular, a $\land$-preaction $(\pi,\theta)\colon\mathcal{S}\curvearrowright\Lambda$ of an inverse semigroupoid $\mathcal{S}$ on a regular semigroupoid $\Lambda$ will always have the product \eqref{eq:semidirectproduct} associative. In this case, $\mathcal{S}\ltimes\Lambda$ is also a regular semigroupoid, and in fact the converse implication is also true.

\begin{proposition}\label{prop:semidirectproductisinverseiffsemigroupoidisinverse}
Suppose that the $\land$-preaction $(\pi,\theta)$ is non-degenerate and $\mathcal{S}\ltimes\Lambda$ is a semigroupoid. Then $\mathcal{S}\ltimes\Lambda$ is regular (resp.\ inverse) if and only if $\Lambda$ is regular (resp.\ inverse).
\end{proposition}
\begin{proof}
We may simply prove that the inverses of $(a,x)\in\mathcal{S}\ltimes X$ are precisely the elements of the form $(a^*,\theta_a(x))$, where $y$ is an inverse of $x$ in $\Lambda$.

First assume that $(a,x)\in\mathcal{S}\ltimes\Lambda$, and that $x$ has an inverse $y$. Then as $y=yxy$ is defined, so $y\in p^{-1}(\so(a))$, and as $\dom(\theta_a)$ is an ideal of $p^{-1}(\so(a))$ then $y\in\dom(\theta_a)$. It is easy to verify that $(a^*,\theta_a(y))$ is an inverse of $(a,x)$.

Conversely, suppose that $(a,x)$ has an inverse $(b,z)$ in $\mathcal{S}\ltimes\Lambda$. This means that
\[(a,x)=(aba,\theta_{a^*}(\theta_{b^*}(x\theta_b(z))\theta_a(x)))\qquad\text{and}\qquad (z,y)=(zaz,\theta_{b^*}(\theta_{a^*}(z\theta_a(x))\theta_b(z))),\]
so $b=a^*$. In particular $\theta_b=\theta_a^{-1}$, so using the fact that $\theta_a$ is a semigroupoid homomorphism, the equations above mean that
\[(a,x)=(a,x\theta_{a^*}(z)x)\qquad\text{and}\qquad (a^*,z)=(a^*,z\theta_a(x)z).\ntag\label{eq:prop:semidirectproductisinverseiffsemigroupoidisinverse}\]
The second equation above implies that
\[\theta_{a^*}(z)=\theta_{a^*}(z\theta_a(x)z)=\theta_{a^*}(z)x\theta_{a^*}(z).\]
Together with the first equation of \eqref{eq:prop:semidirectproductisinverseiffsemigroupoidisinverse}, this means that the element $y=\theta_{a^*}(z)$ is an inverse of $x$ in $\Lambda$, and $z=\theta_a(y)$.\qedhere
\end{proof}

\subsubsection*{The graphed case}

If $\mathcal{T}$ is a graphed semigroupoid and $(\pi,\theta)\colon\mathcal{S}\curvearrowright\mathcal{T}$ is a $\land$-preaction, the semidirect product $\mathcal{S}\ltimes\mathcal{T}$ has a natural graphed structure over $\mathcal{T}^{(0)}$ (or more precisely, its subset $\bigcup_{a\in\mathcal{S}}\dom(\theta_a)$. Namely, the source and range maps of $\mathcal{S}\ltimes\mathcal{T}$ are defined as
\[\so(a,x)=\so(x)\qquad\text{ and }\ra(a,x)=\ra(\theta_a(x)).\]
    
\section{Topological and étale inverse semigroupoids}
    We will now consider semigroupoid endowed with topologies which are compatible with their algebraic structures. In general, a (partial) algebraic structure $\mathscr{A}$ consists of families $\mathscr{A}_0$ of sets and $\mathscr{A}_1$ of partial functions between sets in $\mathscr{A}_0$, and $\mathscr{A}$ is said to be \emph{topological} if all sets in $\mathscr{A}_0$ are endowed with topologies, and all functions in $\mathscr{A}_1$ are continuous. The homomorphisms between such structure are also assumed to be continuous. This applies, in particular, to graphs and (Exel/graphed/inverse) semigroupoids. We write the following definitions in full for the sake of completeness.

\begin{definition}
A \emph{topological graph} is a graph $(G^{(0)},G,\so,\ra)$, such that $G^{(0)}$ and $G$ are endowed with topologies making $\so$ and $\ra$ continuous.

A \emph{topological} (Exel) semigroupoid is a semigroupoid $\Lambda$ endowed with some topology which makes the product continuous (where $\Lambda^{[2]}$ is endowed with the product topology of $\Lambda\times\Lambda$).

A topological graphed semigroupoid is a graphed semigroupoid which is both a topological graph and a topological semigroupoid.
\end{definition}

\begin{definition}
A \emph{topological inverse semigroupoid} is an inverse semigroupoid $\mathcal{S}$ which is a topological graphed semigroupoid, and such that the inverse map $(\ )^*\colon\mathcal{S}\to\mathcal{S}$, $s\mapsto s^*$, is continuous.
\end{definition}

Note that since $(\ )^*\circ (\ )^*=\id_{\mathcal{S}}$, then $(\ )^*$ being continuous implies that it is a homeomorphism.

\begin{definition}\label{def:etaleinversesemigroupoid}
A topological inverse semigroupoid $\mathcal{S}$ is \emph{étale} if the source map $\so\colon\mathcal{S}\to\mathcal{S}^{(0)}$ is a local homeomorphism. (Equivalently, the range map $\ra=\so\circ(\ )^*$ is a local homeomorphism.)
\end{definition}

\begin{remark}
As in \cite{MR2304314}, recall that a topological groupoid $\mathcal{G}$ is \emph{étale} if the source map $\so\colon a\mapsto a^{-1}a$ is a local homeomorphism from $\G$ to $\G[0]=\so(\G)$, where \emph{$\G[0]$ is endowed with the subspace topology of $\G$}. It is not immediately clear that this coincides with the notion in Definition \ref{def:etaleinversesemigroupoid} since for étale semigroupoids we do not assume $\mathcal{S}^{(0)}$ to even be a subset of $\mathcal{S}$.

Corollary \ref{cor:propertiesopeninversesemigroupoid} below deals with this: Suppose $\G$ is a topological groupoid, étale in the sense of \ref{def:etaleinversesemigroupoid}. Let $\tau$ be the topology of $\G$, $\tau|_{\G[0]}$ ist restriction to $\G[0]$, and $\eta$ the topology of $\G[0]$. By \ref{cor:propertiesopeninversesemigroupoid}\ref{cor:propertiesopeninversesemigroupoid1}, $E(\G)=\G[0]$ is open in $\G$. The source map on $\G[0]$ is simply the identity map $(\G[0],\tau|_{\G[0]})\to(\G[0],\eta)$, and it is a bijective (local) homeomorphism, hence $\tau|_{\G[0]}=\eta$.
\end{remark}

Homomorphisms of topological semigroupoids are always assumed to be continuous. A homomorphism $\phi=(\phi^{(0)},\phi^{(1)})\colon G\to H$ of topological graphs is continuous in the sense that $\phi^{(0)}$ and $\phi^{(1)}$ are continuous. When dealing with topological graphed semigroupoids this is a condition which we need to require, but when dealing with étale inverse semigroupoids, continuity of the ``arrow map'' actually implies continuity of the ``vertex map''. For this we recall a fact from the general theory of étale spaces.

Suppose that $\pi_i\colon X_i^{(1)}\to X_i^{(0)}$ ($i=1,2$) are continuous bundles,  and that $f=(f^{(0)},f^{(1)})\colon\pi_1\to\pi_2$ is a bundle homomorphism (possibly discontinuous a priori). Suppose further that $f^{(1)}$ is continuous and that $\pi_1$ is open (and surjective). Then $f^{(0)}$ is continuous. Indeed, since $\pi_1$ is open and surjective then $X_1^{(0)}$ is endowed with the the quotient (final) topology that $\pi_1$ induces. We have the commutative square
\[\begin{tikzpicture}
\node (X11) at (0,0) {$X_1^{(1)}$};
\node (X10) at ([shift={+(0,-1)}]X11) {$X_1^{(0)}$};
\node (X21) at ([shift={+(2,0)}]X11) {$X_2^{(1)}$};
\node (X20) at ([shift={+(0,-1)}]X21) {$X_2^{(0)}$};
\draw[->] (X11)--(X21) node[midway,above] {$f^{(1)}$};
\draw[->] (X11)--(X10) node[midway,left] {$\pi_1$};
\draw[->] (X21)--(X20) node[midway,right] {$\pi_2$};
\draw[->] (X10)--(X20) node[midway,above] {$f^{(0)}$};
\end{tikzpicture}\]
which implies that $f^{(0)}$ is the quotient map induced from the continuous map $\pi_2\circ f^{(1)}$, and is therefore continuous.

We thus apply the fact above with the bundle induced by the source maps of graphed semigroupoids, and in conjunction with Corollary \ref{cor:graphedhomomorphisminducesvertexmap} to obtain the result below:

\begin{corollary}
Suppose that $\phi\colon\mathcal{S}\to\mathcal{T}$ is a continuous homomorphism between topological inverse semigroupoids and that $\mathcal{S}$ is étale. Then there exists a unique continuous map $\phi^{(0)}\colon\mathcal{S}^{(0)}\to\mathcal{T}^{(0)}$ for which $(\phi^{(0)},\phi)$ is a graphed semigroupoid homomorphism.
\end{corollary}

A \emph{bisection} of a graphed semigroupoid $\mathcal{S}$ is a subset $U\subseteq\mathcal{S}$ such that the source and range maps are injective on $U$. We denote by $\mathbf{B}(\mathcal{S})$ the set of all open bisections of an étale inverse semigroupoid $\mathcal{S}$. If $U\in\mathbf{B}(\mathcal{S})$, then the source and range maps restrict to homeomorphisms from $U$ onto open subsets of $\mathcal{S}^{(0)}$. Moreover, $\mathbf{B}(\mathcal{S})$ is a basis for the topology of $\mathcal{S}$.

\begin{proposition}\label{prop:productmapisopen}
If $\mathcal{S}$ is an étale inverse semigroupoid then the product map $\mu\colon\mathcal{S}^{(2)}\to\mathcal{S}$ is open.
\end{proposition}
\begin{proof}
We follow the proof of \cite[1.3.11]{cordeirothesis}: Let $\pi\colon \mathcal{S}^{(2)}\to\mathcal{S}$ be the projection $\pi(a,b)=a$. We prove that $\pi$ is open. If $A,B\in\mathbf{B}(\mathcal{S})$, then $\pi((A\times B)\cap \mathcal{S}^{(2)})=A\cap \so^{-1}(\ra(B))$. Since $\ra(B)$ is open in $\mathcal{S}^{(0)}$ and $\so$ is continuous, then $\pi((A\times B)\cap \mathcal{S}^{(2)})$ is open in $\mathcal{S}$. This proves that $\pi$ is open, as $\mathbf{B}(\mathcal{S})$ is a basis for $\mathcal{S}$. Moreover, $\pi$ is injective on $(A\times B)\cap\mathcal{S}^{(2)}$, because the range map is injective on $B$. Therefore, $\pi$ is a locally injective open continuous map, that is, a local homeomorphism.

Since $\ra\circ\pi=\ra\circ\mu$ on $\mathcal{S}^{(2)}$ and both $\ra$ and $\pi$ are local homeomorphisms, then $\mu$ is a local homeomorphism as well.\qedhere
\end{proof}

We thus make $\mathbf{B}(\mathcal{S})$ into a semigroup with the usual product of sets: for all $A,B\in\mathbf{B}(\mathcal{S})$,
\[AB=\left\{ab:(a,b)\in (A\times B)\cap \mathcal{S}^{(2)}\right\}.\]
Then $\mathbf{B}(\mathcal{S})$ is, in fact, an inverse semigroup. However, we are not able to recover the semigroupoid $\mathcal{S}$ from $\mathbf{B}(\mathcal{S})$ alone, since the canonical order of $\mathbf{B}(\mathcal{S})$ does not correspond to set inclusion. We have
\[A\leq B\iff A=BA^*A\iff\text{for all }a\in A\text{ there exists }b\in B\text{ such that }a\leq b.\]

In Section \ref{sec:duality}, we will prove that an inverse semigroupoid $\mathcal{S}$ may be recovered from $\mathbf{B}(\mathcal{S})$ and set inclusion $\subseteq$.

We finish this section by mentioning a few simple, but nevertheless important, consequences of Proposition \ref{prop:productmapisopen}.

\begin{corollary}\label{cor:propertiesopeninversesemigroupoid}
Let $\mathcal{S}$ be an étale inverse semigroupoid. Then
\begin{enumerate}[label=(\alph*)]
    \item\label{cor:propertiesopeninversesemigroupoid1} $E(\mathcal{S})$ is open in  $\mathcal{S}$;
    \item\label{cor:propertiesopeninversesemigroupoid2} For all open $A\subseteq \mathcal{S}$, the \emph{upper} and \emph{lower closures} \[A^{\uparrow,\leq}=\left\{b\in\mathcal{S}:b\geq a\text{ for some }a\in A\right\}\quad\text{and}\quad A^{\downarrow,\leq}=\left\{b\in\mathcal{S}:b\leq a\text{ for some }a\in A\right\}\]
    are open in $\mathcal{S}$.
\end{enumerate}
\end{corollary}
\begin{proof}
\begin{enumerate}[label=\ref{cor:propertiesopeninversesemigroupoid\arabic*}]
    \item Simply note that $E(\mathcal{S})=\bigcup\left\{A^*A:A\in \mathbf{B}(\mathcal{S})\right\}$.
    \item The lower set $A^{\downarrow,\leq}=AE(\mathcal{S})$ is the product of open sets, hence open.
    
    Suppose $b\geq a$ for some $a\in A$, i.e., $ba^*a=a$. Since the semigroupoid operations are continuous, there exist neighbourhoods $U$ of $b$ and $V$ of $a$ such that $UV^*V\subseteq A$. Let $B=U\cap\so^{-1}(\so(A\cap V))$. Then $B$ is an open neighbourhood of $b$.
    
    Given $c\in B$, we have $\so(c)=\so(a)$ for some $a\in A\cap V$, so $ca^*a$ is defined, belongs to $UV^*V\subseteq A$ and $ca^*a\leq c$, so $c\in A^{\uparrow,\leq}$. Therefore $B$ is an open neighbourhood of $b$ contained in $A^{\uparrow,\leq}$.\qedhere
\end{enumerate}
\end{proof}

\begin{example}
    For non-étale semigroupoids, upper (and lower) closures of open sets are not necessarily open.

    Let $X=[0,1]$ with its usual topology and $L_2=\left\{0,1\right\}$ the lattice with $0<1$. The inverse semigroupoid $L_2\times X$ is étale. Consider the subsemigroupoid $\Lambda=(L_2\times X)\setminus\left\{(0,1/n):n\in\mathbb{N}_{\geq 1}\right\}$. Then $\Lambda$ is a non-étale topological inverse semigroupoid.

    Let $U$ be any neighbourhood of $0$ in $X$. Then $\tilde{U}\defeq(\left\{0\right\}\times U)\cap\Lambda$ is a neighbourhood of $(0,0)$ in $\Lambda$, but the upper closure $\tilde{U}^{\uparrow,\leq}=(L_2\times U)\cap\Lambda$ contains the non-interior point $(1,0)$.
\end{example}

It follows from Corollary \ref{cor:propertiesopeninversesemigroupoid} that if $\mathcal{S}$ is an étale inverse semigroupoid, then $E(\mathcal{S})$ is also an étale inverse semigroupoid with the subspace topology. In fact a weaker version of the converse holds: If $\mathcal{S}$ is a topological inverse semigroupoid and $E(\mathcal{S})$ is open and étale, then $\mathcal{S}$ has a basis of open bisections. However this only implies that the source map is locally injective, but not necessarily open, and so $\mathcal{S}$ may be non-étale. In any case, the following generalization of \cite[Theorem 5.18]{MR2304314} holds, with simple adaptations on the proof.

\begin{theorem}
    Let $\mathcal{S}$ be a topological inverse semigroupoid. Then the following are equivalent:
    \begin{enumerate}[label=(\arabic*)]
        \item $\mathcal{S}$ is étale;
        \item $E(\mathcal{S})$ is open and étale, and the product of open sets is open.
        \item $E(\mathcal{S})$ is open and étale, and $\mathcal{S}A$ is open for each open $A\subseteq\mathcal{S}$.
        \item $E(\mathcal{S})$ is open and étale, and $A^*A$ is open for each open $A\subseteq\mathcal{S}$.
        \item $E(\mathcal{S})$ is open and étale, and the source map is open;
    \end{enumerate}
\end{theorem}

\subsection{Continuous \texorpdfstring{$\land$}{∧}-preactions and semidirect products}\label{subsec:continuouslandpreactions}

Tipically, in the topological setting we consider actions which preserve the topological structure. Let us consider the ``action map'' associated to a $\land$-preaction $(\pi,\theta)\colon\mathcal{S}\curvearrowright\Lambda$, which we also denote by $\theta$:
\[\theta\colon\mathcal{S}\ltimes\Lambda\to\Lambda,\qquad\theta(a,x)=\theta_a(x).\]
In the case that $\mathcal{S}$ and $\Lambda$ are topological semigroupoids, we endow $\mathcal{S}\ltimes\Lambda$ with the product topology. If $A\subseteq\mathcal{S}$ and $U\subseteq\Lambda$, we denote $A\ast U=(A\times U)\cap(\mathcal{S}\ltimes\Lambda)$, so that the basic open sets of $\mathcal{S}\ltimes\Lambda$ have the form $A\ast U$ where $A$ and $U$ are (basic) open subsets of $\mathcal{S}$ and $\Lambda$, respectively.

\begin{definition}
    A $\land$-preaction $(\pi,\theta)\colon\mathcal{S}\curvearrowright\Lambda$ of a topological inverse semigroupoid $\mathcal{S}$ on a topological semigroupoid $\Lambda$ is \emph{continuous} if $\pi$ and the ``action map'' $\theta\colon\mathcal{S}\ltimes\Lambda\to\Lambda$ are continuous. If the action map $\theta$ is open then we call $(\pi,\theta)$ an \emph{open} $\land$-preaction.
\end{definition}

Evidently, the semidirect product $\mathcal{S}\ltimes\Lambda$ associated to a continuous $\land$-preaction $(\pi,\theta)\colon\mathcal{S}\curvearrowright\Lambda$ is a topological semigroupoid (as long as the product \eqref{eq:semidirectproduct} is associative).

\begin{remark}
Suppose that $(\pi,\theta)\colon\mathcal{S}\curvearrowright\Lambda$ is a continuous open $\land$-preaction. Given $a\in\mathcal{S}$, let $A$ be any open bisection containing $a$. Then \[\ran(\theta_a)=\pi^{-1}(\ra(a))\cap\theta(A\ast\Lambda)\]
is open in $\pi^{-1}(\ra(a))$.
\end{remark}

If $\Lambda$ is a topological graphed or inverse semigroupoid then $\mathcal{S}\ltimes\Lambda$ will also be a topological graphed or inverse semigroupoid, where the vertex set $(\mathcal{S}\ltimes\Lambda)^{(0)}=\Lambda^{(0)}$ is endowed with its original topology.

The following proposition simplifies the verification of when an action is open.

\begin{proposition}\label{prop:preactionisopeniffprojectionisopen}
Let $(\pi,\theta)\colon\mathcal{S}\curvearrowright\Lambda$ be a continuous $\land$-preaction of topological semigroupoids. Then $(\pi,\theta)$ is open if and only if the map
\[p\colon\mathcal{S}\ltimes\Lambda\to\Lambda,\qquad(a,x)\mapsto x\]
is open.
\end{proposition}
\begin{proof}
Let $A\subseteq\mathcal{S}$ and $U\subseteq\Lambda$ be open. Assuming that $(\pi,\theta)$ is open, we have
\[p(A\ast U)=\theta(A^*\ast\theta(A\ast U)),\]
which is open as the action map $\theta$ is open.

Conversely, assume that $p$ is open. The map \[I\colon\mathcal{S}\ltimes\Lambda\to\mathcal{S}\ltimes\Lambda,\qquad (a,x)\mapsto (a^*,\theta_a(x))\]
is a self-homeomorphism of $\mathcal{S}\ltimes\Lambda$ of order $2$, and in particular it is an open map. The action map is the composition $\theta=p\circ I$, hence an open map.
\end{proof}

For example, if $(\pi,\theta)\colon\mathcal{S}\curvearrowright\Lambda$ is a continuous $\land$-preaction and $\mathcal{S}\ltimes\Lambda$ is open in $\mathcal{S}\times\Lambda$, then $(\pi,\theta)$ is also open. This is an assumption made, for example, in \cite{MR2045419}.)

Finally, if $\mathcal{S}$ and $\mathcal{T}$ are étale inverse semigroupoids and $(\pi,\theta)\colon\mathcal{S}\curvearrowright\mathcal{T}$ is an open continuous $\land$-preaction, then $\mathcal{S}\ltimes\mathcal{T}$ is also an étale inverse semigroupoid. Indeed, if $A\subseteq\mathcal{S}$ is an open bisection and $U\subseteq\mathcal{T}$ is open, then the map $p$ is injective on $A\ast U$. Indeed, if $p(a,x)=p(b,y)$, where $(a,x),(b,y)\in A\ast U$, then $x=y$ (by definition of $p$) and $\so(a)=\pi(x)=\pi(y)=\so(b)$, thus $a=b$ as $A$ is a bisection. Therefore $p$ is locally injective, continuous and open, i.e., a local homeomorphism. The range map $\ra_{\mathcal{S}\ltimes\mathcal{T}}$ of $\mathcal{S}\ltimes\mathcal{T}$ is the composition of the range map $\ra_{\mathcal{T}}$ of $\mathcal{T}$ and the action map $\theta$, which are both local homeomorphisms. Therefore $\mathcal{S}\ltimes\mathcal{T}$ is étale.

\begin{example}\label{ex:munn}
    We will now describe an analogue, in the setting of semigroupoids, of the canonical actions of a semigroup on its idempotent semilattice (the \emph{Munn representation}), and of 
    
    Let $\mathcal{S}$ be an inverse semigroupoid. We denote by $\operatorname{F}(\mathcal{S})$ the set of elements $b\in\mathcal{S}$ such that
    \begin{enumerate}[label=(\roman*)]
        \item $\so(b)=\ra(b)$;
        \item For all $e\leq b^*b$, we have $beb^*=e$;
    \end{enumerate}
    
    The verification that $\operatorname{F}(\mathcal{S})$ is an inverse sub-semigroupoid of $\mathcal{S}$ is straightforwards, using properties of the canonical order of $\mathcal{S}$. Note that $E(\mathcal{S})\subseteq\operatorname{F}(\mathcal{S})$. If $b\in\operatorname{F}(\mathcal{S})$, then $b^*b=bb^*$, and for all $e\in E(\mathcal{S})\cap\mathcal{S}^b$, $beb^*=eb^*b$
    
    If $a\in\mathcal{S}$, $b\in\operatorname{F}(\mathcal{S})$ and $ab$ is defined, then $aba^*\in\operatorname{F}(\mathcal{S})$ as well.
    
    Then $\mathcal{S}$ carries a natural action by conjugation) on $\operatorname{F}(\mathcal{S})$ as follows. Consider the bundle $\pi\defeq\so|_{\operatorname{F}(\mathcal{S})}=\ra|_{\operatorname{F}(\mathcal{S})}\colon\operatorname{F}(\mathcal{S})\to\mathcal{S}^{(0)}$. Given $a\in\mathcal{S}$, let $\dom(\mu_a)\defeq\left\{b\in\operatorname{F}(\mathcal{S}):bb^*\leq a^*a\right\}$, which is an ideal of $\operatorname{F}(\mathcal{S})$, and contained in $\pi^{-1}(\so(a))$. The map $\mu_a\colon\dom(\mu_a)\to\ran(\mu_{a^*})$, $\mu_a(b)=aba^*$, is an isomorphism, and the pair $(\pi,\tau)$ is a global action of $\mathcal{S}$ on $\operatorname{F}(\mathcal{S})$.
    
    If $\mathcal{S}$ is an étale inverse semigroupoid, then $\operatorname{int}(\operatorname{F}(\mathcal{S})$ is invariant, in the sense that $\mu(\mathcal{S}\ast\operatorname{int}(\operatorname{F}(\mathcal{S})))\subseteq \operatorname{int}(\operatorname{F}(\mathcal{S}))$. Indeed, if $A$ and $U$ are open bisections of $\mathcal{S}$ and $U\subseteq\operatorname{F}(\mathcal{S})$, then $\mu(A\ast U)=AUA^*$ is the set of products $aua^*$ (whenever defined), where $a\in A$ and $u\in U$. Since the product of open sets is open and $\mathcal{S}\ast\operatorname{int}(\operatorname{F}(\mathcal{S}))$ is the union of all such sets $A\ast U$, then $\operatorname{int}(\operatorname{F}(\mathcal{S}))$ is invariant.
    
    Thus we may restrict the action $(\pi,\mu)$ to $\operatorname{int}(\operatorname{F}(\mathcal{S}))$. In fact, the argument in the paragraph above proves that the restriction of $(\pi,\mu|_U)$ to any invariant open subset $U$ of $\operatorname{F}(\mathcal{S})$ is an open action. Since $E(\mathcal{S})$ is also open and invariant, then the semigroupoids $\mathcal{S}\ltimes\operatorname{int}(\operatorname{F}(\mathcal{S})$ and $\mathcal{S}\ltimes E(\mathcal{S})$ are étale as well.
    
    In classical cases, the action $\mu$ is well-known.
    
    \begin{itemize}
        \item If $\mathcal{S}$ is an étale groupoid, then $\operatorname{F}(\mathcal{S})$ is the isotropy subgroupoid of $\mathcal{S}$. The restriction of $(\pi,\mu)$ to $E(\mathcal{S})=\mathcal{S}^{(0)}$ is the canonical action of $\mathcal{S}$ on its unit space: $\mu_g$ is defined only as $\mu_g(\so(g))=\ra(g)$, for each $g\in\mathcal{S}$.
        \item if $\mathcal{S}$ is a discrete inverse semigroup then the restriction of $(\pi,\mu)$ to $E(\mathcal{S})$ is the \emph{Munn representation} of $\mathcal{S}$.
    \end{itemize}
\end{example}

We finish this section by proving that continuous $\land$-preactions of discrete inverse semigroups, and continuous non-degenerate global actions of étale groupoids, are always open. Of course, it is enough to consider only actions on topological spaces (unit groupoids).

\begin{example}\label{ex:semidirectproductofactionisetale}
Just as in Example \ref{ex:actionofsemigroup}, a topological partial (or global) action $\theta$ of an inverse semigroup $S$ on a topological space $X$ is precisely a continuous partial action of $S$ on $X$, as topological semigroupoids (as in \cite[Definition 2.3]{arxiv1804.00396}), where we regard $S$ as a discrete space.

In fact, more generally, if $\theta$ is a continuous $\land$-preaction of $S$ on $X$, then $\dom(\theta_a)$ is open in $X$ for all $a\in S$ (by a previous remark). Thus $S\ltimes X=\bigcup_{a\in S}\left\{a\right\}\times\dom(\theta_a)$ is open in $S\times X$ (since we regard $S$ as a discrete space). From the comment after Proposition \ref{prop:preactionisopeniffprojectionisopen}, it follows that $\theta$ is an open $\land$-preaction, and in particular $S\ltimes X$ is an étale inverse semigroupoid.
\end{example}

\begin{example}
As in Example \ref{ex:groupoidnondegenerateaction}, a continuous, non-degenerate global action $(\pi,\theta)\colon\G\curvearrowright X$ of a topological groupoid $\G$ on a topological space $X$ is the same notion as used in \cite[Definition 3.6]{MR2969047}. Suppose moreover that $\G$ is étale, and let us prove that $(\pi,\theta)$ is open. The map $p$ of Proposition \ref{prop:preactionisopeniffprojectionisopen} is precisely the source map $\so\colon\G\ltimes X\to X$, $\so(g,x)=x$, so we need only to prove that it is open. Let $A\subseteq\G$ and $U\subseteq X$ be open. From Example \ref{ex:groupoidnondegenerateaction}, we know that $\pi^{-1}(\so(a))=\dom(\theta_a)$ for all $a\in\G$, so it follows that
\[\so(A\ast U)=\pi^{-1}(\so(A))\cap U\]
which is open because $\G$ is étale. Therefore $(\pi,\theta)$ is an open global action.
\end{example}

We now refine Proposition \ref{prop:extensionoflandpreactiontopartialaction} to the topological setting. Recall that any $\land$-preaction $(\pi,\theta)$ may be extended, in a minimal manner, to a partial action $(\pi,\overline{\theta})$.

\begin{proposition}
If $\mathcal{S}$ is an étale semigroupoid, $\Lambda$ is a topological semigroupoid and $(\pi,\theta)$ is a continuous (resp.\ open) $\land$-preaction, then the minimal partial action $(\pi,\overline{\theta})$ of Proposition \ref{prop:extensionoflandpreactiontopartialaction} is also continuous (resp.\ open).
\end{proposition}

To avoid confusion, we denote $\mathcal{S}\ltimes\Lambda$ and $\mathcal{S}\overline{\ltimes}\Lambda$ the (underlying sets of the) semidirect products associated to $\theta$ and $\overline{\theta}$, respectively, and if $A\subseteq\mathcal{S}$, $U\subseteq\Lambda$, by $A\ast U\defeq(A\times U)\cap(\mathcal{S}\ltimes\Lambda)$ and $A\overline{\ast}U\defeq(A\times U)\cap(\mathcal{S}\overline{\ltimes}\Lambda)$. The action maps are still denoted $\theta$ and $\overline{\theta}$.

\begin{proof}
    If $U\subseteq\Lambda$ is open, then $\overline{\theta}^{-1}(U)$ is the union of the sets of the form $(A^{\uparrow,\leq})\overline{\ast} V$, where $A\subseteq \mathcal{S}$ and $V\subseteq\Lambda$ are open, and $A\ast V\subseteq\theta^{-1}(U)$. Thus, if $(\pi,\theta)$ is continuous then $(\pi,\overline{\theta})$ is also continuous.
    
    Similarly, if $A\subseteq\mathcal{S}$ and $V\subseteq\Lambda$ are open, then $\overline{\theta}(A\overline{\ast} V)=\theta(A^{\downarrow,\leq}\ast V)$, so $(\pi,\theta)$ being open implies that $(\pi,\overline{\theta})$ is also open. \end{proof}

\section{Constructions with inverse semigroupoids}\label{sec:categoricalconstructions}
    Throughout this section the letters $\mathcal{S}$, $\mathcal{T}$,\ldots will always denote inverse semigroupoids. We define $\cat{TopIS}$ as the category of topological inverse semigroupoids and their continuous homomorphisms, and $\cat{TopGr}$ as the full subcategory of $\cat{TopIS}$ of topological groupoids. By $\cat{EtIS}$ and $\cat{EtGr}$ we denote the full subcategories of $\cat{TopIS}$ of étale inverse semigroupoids and groupoids, respectively.

\subsection{Underlying groupoids}
\begin{definition}
The \emph{underlying} or \emph{restricted product groupoid} of a semigroupoid $\mathcal{S}$ is the groupoid $\mathscr{U}(\mathcal{S})$ obtained from $\mathcal{S}$ by restricting its product to the set $\mathscr{U}(\mathcal{S})^{(2)}=\left\{(a,b)\in \mathcal{S}\times \mathcal{S}:a^*a=bb^*\right\}$.
\end{definition}

It is straightforward to check that $\mathscr{U}(\mathcal{S})$ is indeed a groupoid. The unit space (identified with the object space) of $\mathscr{U}(\mathcal{S})$ is $E(\mathcal{S})$, and the inverse of $a\in\mathscr{U}(\mathcal{S})$ is $a^{-1}=a^*$. The source and range maps on $\mathscr{U}(\mathcal{S})$ are given, respectively, by
\[\so_{\mathscr{U}(\mathcal{S})}(a)=a^*a,\qquad\ra_{\mathscr{U}(\mathcal{S})}(a)=aa^*.\]

If $\mathcal{S}$ is a topological semigroupoid, then the same topology of $\mathcal{S}$ makes $\mathscr{U}(\mathcal{S})$ a topological groupoid: The product and inverse maps are immediately continuous.

If $\mathcal{S}$ is étale, then $\mathscr{U}(\mathcal{S})$ is étale as well. Indeed, suppose that $A$ is a bisection of the semigroupoid $\mathcal{S}$. If $a,b\in A$ and $a^*a=b^*b$, then $\so(a)=\ra(a^*a)=\ra(b^*b)=\so(b)$, thus $a=b$. Then the source map $\so_{\mathscr{U}(\mathcal{S})}\colon a\mapsto a^*a$ is injective on $A$. Moreover, $\so_{\mathscr{U}(\mathcal{S})}(A)=A^*A$ is open for every $A\in\mathbf{B}(\mathcal{S})$. Thus $\so_{\mathscr{U}(\mathcal{S})}$ is a continuous, open, locally injective map, i.e., a local homeomorphism. Therefore $\mathscr{U}(\mathcal{S})$ is an étale groupoid.

Note that if $\G$ is a groupoid and $\phi\colon\mathcal{G}\to\mathcal{S}$ is a homomorphism, then for all $(g,h)\in\G[2]$ we have
\[\phi(g)^*\phi(g)=\phi(g^{-1}g)=\phi(hh^{-1})=\phi(h)\phi(h)^*\]
so the same map $\phi$ may be seen as a groupoid homomorphism from $\G$ to $\mathscr{U}(\mathcal{S})$.

In particular, if $\phi\colon\mathcal{S}\to \mathcal{T}$ is a continuous semigroupoid homomorphism, then the same map induces a continuous groupoid homomorphism $\mathscr{U}(\phi)=\phi\colon\mathscr{U}(\mathcal{S})\to\mathscr{U}(\mathcal{T})$. This gives us a functor $\mathscr{U}\colon\cat{TopIS}\to\cat{TopGr}$. In fact, $\mathscr{U}$ is a retraction from $\cat{TopIS}$ to $\cat{TopGr}$, i.e., the restriction of $\mathscr{U}$ to $\cat{TopGr}$ is the identity functor.

The functor $\mathscr{U}$ can be easily seen to be terminal universal with respect to the inclusion of categories $\iota\colon\cat{TopGr}\hookrightarrow\cat{TopIS}$. The same facts are also true when restricted to $\cat{EtIS}$ and $\cat{EtGr}$.

\begin{proposition}
Let $\mathcal{S}$ be a topological (resp.\ étale) inverse semigroupoid. Then for every topological (resp.\ étale) groupoid $\G$ and for every continuous semigroupoid homomorphism $\phi\colon\G\to\mathcal{S}$, the associated map $\phi\colon\G\to\mathscr{U}(\mathcal{S})$ is a groupoid morphism. In other words, there exists a unique groupoid morphism $\psi\colon\G\to\mathscr{U}(S)$ (namely, the same function as $\phi$) such that the following diagram commutes:
\[\begin{tikzpicture}
\node (G) at (0,0) {$\G$};
\node (GS) at ([shift={+(0,-1.5)}]G) {$\mathscr{U}(\mathcal{S})$};
\node (iG) at ([shift={+(1.5,0)}]G) {$\iota(\G)$};
\node (iGS) at ([shift={+(1.5,-1.5)}]G) {$\iota(\mathscr{U}(\mathcal{S}))$};
\node (S) at ([shift={+(1.5,0)}]iGS) {$\mathcal{S}$};
\draw[->,dashed] (G)--(GS) node[left,midway] {$\psi$};
\draw[->] (iG)--(iGS) node[left,midway] {$\iota(\psi)$};
\draw[->] (iG)--(S) node[right,midway] {$\phi$};
\draw[->] (iGS)--(S) node[below,midway] {$\id$};
\end{tikzpicture}\]
\end{proposition}

\subsection{Quotients}

A somewhat general notion of quotient for discrete semigroupoids is considered in \cite{MR3597709}. We will consider only quotients of inverse semigroupoids which preserve their vertex sets.

\begin{definition}
If $G$ is a graph, an equivalence relation $R$ on $G^{(1)}$ is said to be \emph{graphed} the source and range maps are $R$-invariant (i.e., constant on all $R$-equivalence classes).

A \emph{graph congruence} on an inverse semigroupoid $\mathcal{S}$ is a graphed equivalence relation $R$ on $\mathcal{S}$ such that for all $a,\widetilde{a},b,\widetilde{b}\in \mathcal{S}$, if $(a,\widetilde{a}),(b,\widetilde{b})\in R$ and $(a,b)\in\mathcal{S}^{(2)}$, then $(ab,\widetilde{a}\widetilde{b})\in R$.
\end{definition}

Note that, in the definition above, the product $\widetilde{a}\widetilde{b}$ is defined since $\so(\widetilde{a})=\so(a)=\ra(b)=\ra(\widetilde{b})$, as the source and range maps are constant on $R$-equivalence classes.

Given a graph congruence $R$ on an inverse semigroupoid $\mathcal{S}$, we let $\mathcal{S}/R$ be the \emph{quotient semigroupoid}, which is a graphed semigroupoid constructed in the same manner as quotients of categories (see \cite[Section 2.8]{MR1712872}): The vertex space is $(\mathcal{S}/R)^{(0)}\defeq\mathcal{S}^{(0)}$. The arrow space is the usual quotient set $\mathcal{S}/R$, and we denote the $R$-equivalence class of $a\in\mathcal{S}$ as $[a]$. Since $\so$ and $\ra$ are constant on $R$-equivalence classes, then they factor (uniquely) to maps $\so,\ra\colon\mathcal{S}/R\to\mathcal{S}^{(0)}$, $\so([a])=\so(a)$ and $\ra([a])=\ra(a)$ for all $a\in\mathcal{S}$. The product is defined in the only manner which makes the natural quotient map $\mathcal{S}\to\mathcal{S}/R$ a semigroupoid homomorphism. Namely, given $(x,y)\in (\mathcal{S}/R)^{(2)}$, choose representatives $a\in x$ and $b\in y$. Then $\so(a)=\so(x)=\ra(y)=\ra(b)$, so we may define $xy=[ab]$. Since $R$ is a congruence, this product does not depend on the choice of representatives of $x$ and $y$. Associativity of the product is immediate.

Let us prove that quotients of inverse semigroupoids are also inverse semigroupodis. The following is an analogue of a well-known fact for semigroups, and its proof follows the same steps. See \cite[Theorem 5.1.1]{MR1455373}, for example.

\begin{lemma}[{\cite[Lemma 3.3.1]{MR3597709}}]\label{lem:regularisinverseiffidempotentscommute}
A regular graphed semigroupoid $\mathcal{S}$ is an inverse semigroupoid if and only if elements of $E(\mathcal{S})$ commute, i.e., for all $(e,f)\in (E(\mathcal{S})\times E(\mathcal{S}))\cap \mathcal{S}^{(2)}$, $ef=fe$.
\end{lemma}

The following is a particular case of \cite[Lemma 3.3.3]{MR3597709} and the Lemma above. It may also be proven directly as in the case of inverse semigroups -- see \cite[Proposition 2.1.1(iii)]{MR1724106}.

\begin{proposition}
If $R$ is a graph congruence on an inverse semigroupoid $\mathcal{S}$, then $\mathcal{S}/R$ is an inverse semigroupoid.
\end{proposition}

\subsubsection{The étale case}

Suppose now that $\mathcal{S}$ is a topological (or étale) inverse semigroupoid, and that $R$ is a graphed congruence on $\mathcal{S}$. We wish to endow $\mathcal{S}/R$ with a natural topology making it a topological (or étale) inverse semigroupoid as well, and for this we use the \emph{quotient topology}. However, in general we do not have control over the topology of $(\mathcal{S}/R)^{(2)}$, and thus we cannot guarantee that the multiplication map of $\mathcal{S}/R$ is continuous with respect to the quotient topology. We will therefore to make further topological assumptions on the congruence $R$ in order to obtain the desired result. Let us briefly recall some facts about quotient topologies.

\begin{itemize}
    \item If $X$ is a topological space, $Y$ is a set, and $\pi\colon X\to Y$ is a function, the \emph{quotient topology} of $Y$ (induced by $\pi$) has as open subsets the subsets $A$ of $Y$ such that $\pi^{-1}(A)$ is open in $X$. In this case, a function $p\colon Y\to Z$, where $Z$ is a topological space, is continuous if and only if $p\circ\pi$ is continuous. In other words, continuous maps from $Y$ are precisely the factors of continuous maps from $X$ through $\pi$.
    \item If $f\colon X\to Y$ is a continuous, open, surjective function between topological spaces, then the topology of $Y$ concides the quotient topology of $f$.
\end{itemize}

Suppose that $R$ is an equivalence relation on a topological space $X$. Given a subset $A\subseteq\mathcal{S}$, we let $R[A]$ denote the \emph{saturation} of $A$, i.e., $R[A]\defeq\left\{x\in X:xRa\text{ for some }a\in A\right\}$. 
We say that $R$ is \emph{open} if the saturation of every open subset of $X$ is open, or equivalently if the quotient map $X\to X/R$ is an open map.

\begin{proposition}\label{prop:equivalecesetalequotient}
Let $\mathcal{S}$ be a topological (resp.\ étale) inverse semigroupoid and $R$ a graphed open congruence on $\mathcal{S}$. Then the quotient topology of $\mathcal{S}/R$ makes it a topological (resp.\ étale) inverse semigroupoid, where we endow $\mathcal{T}^{(0)}=\mathcal{S}^{(0)}$ with its original topology.
\end{proposition}
\begin{proof}
    We first assume only that $\mathcal{S}$ is a topological inverse semigroupoid. Let us denote $\mathcal{T}\defeq\mathcal{S}/R$ and $\pi\colon\mathcal{S}\to\mathcal{T}$ the quotient map.
    
    We have the commutative diagram
    \[\begin{tikzpicture}
    \node (S1) at (0,0) {$\mathcal{S}$};
    \node (S2) at ([shift={+(2,0)}]S1) {$\mathcal{S}$};
    \node (T1) at ([shift={+(0,-1)}]S1) {$\mathcal{T}$};
    \node (T2) at ([shift={+(2,0)}]T1) {$\mathcal{T}$};
    \draw[->] (S1)--(S2) node[midway,above] {$(\ )^*$};
    \draw[->] (S1)--(T1) node[midway,left] {$\pi$};
    \draw[->] (S2)--(T2) node[midway,right] {$\pi$};
    \draw[->] (T1)--(T2) node[midway,above] {$(\ )^*$};
    \end{tikzpicture}
    \]
    where the horizontal arrows denote the inversion maps. Therefore the inversion map of $\mathcal{T}$ is simply the factor map of $\pi\circ(\ )^*$ through $\pi$, and is therefore continuous. Continuity of the source and range maps may be proven similarly with diagrams.
    
    We will prove that $\mathcal{T}^{(2)}$ has the quotient topology induced by the restriction of $\pi\times\pi$ to $\mathcal{T}^{(2)}$. After that, continuity or the product on $\mathcal{T}$ follows easily, since we have the commutative diagram
    \[\begin{tikzpicture}
    \node (S2) at (0,0) {$\mathcal{S}^{(2)}$};
    \node (S) at ([shift={+(2,0)}]S2) {$\mathcal{S}$};
    \node (T2) at ([shift={+(0,-1)}]S2) {$\mathcal{T}^{(2)}$};
    \node (T) at ([shift={+(2,0)}]T2) {$\mathcal{T}$};
    \draw[->] (S2)--(S) node[midway,above] {$\mu_{\mathcal{S}}$};
    \draw[->] (S2)--(T2) node[midway,left] {$\pi\times\pi$};
    \draw[->] (S)--(T) node[midway,right] {$\pi$};
    \draw[->] (T2)--(T) node[midway,above] {$\mu_{\mathcal{T}}$};
    \end{tikzpicture}
    \]
    where the horizontal arrows denote the product maps. This means that the product of $\mathcal{T}$ is the factor of a continuous map through $\pi\times\pi$, and is therefore continuous.
    
    Since $\pi$ is surjective, continuous and open, because $R$ is open, then $\pi\times\pi$ is also surjective, continuous and open, and therefore $\mathcal{T}\times\mathcal{T}$ has the quotient topology induced by $\pi\times\pi$.
    
    Moreover, we have $\mathcal{S}^{(2)}=(\pi\times\pi)^{-1}\left(\mathcal{T}^{(2)}\right)$ and so for every $A\subseteq \mathcal{S}\times\mathcal{S}$ we have $(\pi\times\pi)(A\cap\mathcal{S}^{(2)})=(\pi\times\pi)(A)\cap\mathcal{T}^{(2)}$. In particular the restriction $(\pi\times\pi)|_{\mathcal{S}^{(2)}}\colon\mathcal{S}^{(2)}\to\mathcal{T}^{(2)}$ is also continuous, surjective, and open, and therefore $\mathcal{T}^{(2)}$ has its quotient topology, and continuity of the product map follows.
    
    Now assume that $\mathcal{S}$ is étale. If $A$ is an open bisection of $\mathcal{S}$, then $\pi(A)$ is an open bisection of $\mathcal{T}$, because $\so_{\mathcal{T}}\circ\pi=\so_{\mathcal{S}}$ and $\ra_{\mathcal{T}}\circ\pi=\ra_{\mathcal{S}}$ are injective on $A$. Moreover, the source map of $\mathcal{T}$ is open: if $C\subseteq\mathcal{T}$ is open, then $\so_{\mathcal{T}}(C)=\so_{\mathcal{S}}(\pi^{-1}(C))$ is open in $\mathcal{T}^{(0)}=\mathcal{S}^{(0)}$. Therefore $\so_{\mathcal{T}}$ is continuous, open, and locally injective, hence a local homeomorphism.\qedhere
\end{proof}

Therefore, under the hypotheses of Proposition \ref{prop:equivalecesetalequotient}, $\pi\colon\mathcal{S}\to\mathcal{S}/R$ is a continuous homomorphism of topological inverse semigroupoid satisfying the usual universal property: If $\phi\colon\mathcal{S}\to\Lambda$ is any continuous (resp.\ open) homomorphism from $\mathcal{S}$ to a topological semigroupoid $\Lambda$ such that $R\subseteq\ker\phi$, then there exists a unique continuous (resp.\ open) semigroupoid homomorphism $\psi\colon\mathcal{S}/R\to\Lambda$ such that $\phi=\psi\circ\pi$.

\begin{remark}
A similar construction as in \cite{MR0412333} shows, by set-theoretic considerations, that the category of topological inverse semigroupoids admits arbitrary quotients in the following sense: If $\mathcal{S}$ is any topological inverse semigroupoid and $R$ is any congruence on $\mathcal{S}$, then there exists a topological inverse semigroupoid $\mathcal{S}/R$ and a continuous homomorphism $\pi_R\colon\mathcal{S}\to\mathcal{S}/R$ such that $\pi_R(a)=\pi_R(b)$ whenever $(a,b)\in R$, and for any other inverse semigroupoid $\mathcal{T}$ and any continuous homomorphism $\phi\colon\mathcal{S}\to\mathcal{T}$ such that $\phi(a)=\phi(b)$ whenever $(a,b)\in R$, then there exists a unique continuous groupoid homomorphism $\psi\colon\mathcal{S}/R\to\mathcal{T}$ such that $\psi\circ\pi_R=\phi$.

The main idea in the construction of $\mathcal{S}/R$ as above is to show that there is a set $\mathscr{Q}$ of all possible topological semigroupoid quotients of $\mathcal{S}$ which identify $R$-equivalent elements of $\mathcal{S}$. More precisely, $\mathscr{Q}$ is a set of topological inverse semigroupoids $Q$ endowed with continuou maps $\pi_Q\colon\mathcal{S}\to Q$ such that $\pi_Q(a)=\pi_Q(b)$ whenever $(a,b)\in R$, and which satisfies the following property: If $\phi\colon\mathcal{S}\to\mathcal{T}$ is any topological semigroupoid homomorphism such that $\phi(a)=\phi(b)$ whenever $(a,b)\in R$ and $\phi(\mathcal{S})$ generates $\mathcal{T}$ (as an inverse semigroupoid), then $(\mathcal{T},\phi)$ may be represented (faithfully) as $(Q,\pi_Q)$ for some $Q\in\mathscr{Q}$. We then take $(\mathcal{S}/R)'=\prod_{Q\in\mathscr{Q}}Q$, $\pi=\prod_{Q\in\mathscr{Q}}\pi_Q$, and $\mathcal{S}/R$ the sub-inverse semigroupoid of $(\mathcal{S}/R)'$ generated by $\pi(S)$.

However, the topology of $\mathcal{S}/R$ is not manageable in abstract terms, and a priori the inverse semigroupoid $\mathcal{S}/R$ might be trivial even if $R\neq\mathcal{S}\times\mathcal{S}$.
\end{remark}

\subsection{Semigroupoids of germs and the initial groupoid}\label{subsec:groupoidofgerms}

\begin{definition}
A poset $(P,\preceq)$ is \emph{conditionally downwards directed} for all $p\leq P$, the downset $p^{\downarrow,\preceq}=\left\{x\in P:x\preceq p\right\}$ is downwards directed. Explicitly, this means that whenever $x,y\preceq p$ in $P$, there exists $z\in P$ with $z\preceq x,y$.
\end{definition}

As an example, every inverse semigroupoid with its canonical order is conditionally downwards directed.

The \emph{germ relation} $\sim_{\preceq}$ on a poset $(P,\preceq)$ is defined as
\[a\sim_{\preceq} b\iff\text{there exists }z\in P\text{ such that } z\leq a\text{ and }z\leq b.\]

\begin{lemma}
$\sim_\preceq$ is always reflexive and symmetric, and it is transitive (and thus an equivalence relation) if and only if $(P,\preceq)$ is conditionally downwards directed.
\end{lemma}
\begin{proof}
    Only the last statement warrants a proof. If $(P,\preceq)$ is conditionally downwards directed, then suppose $a\sim_{\preceq} b$ and $b\sim_{\preceq} c$. Let $x$ and $y$ with $x\preceq a,b$ and $y\preceq b,c$. As both $x$ and $y$ are bounded above by $b$, we may take $z\preceq x,y$. Then $z\preceq x\preceq a$ and $z\preceq y\preceq c$, so $a\sim_\preceq c$. Hence $\sim_\preceq$ is transitive.
    
    Conversely, if $\sim_\preceq$ is transitive, suppose $x,y\preceq p$ in $P$. Then $x\sim_\preceq p$ and $p\sim_\preceq y$, so $x\sim_\preceq y$, which means that there exists $z\preceq x,y$, as desired.\qedhere
\end{proof}

The following example will be used in our duality result, and yields a large class 

\begin{definition}\label{def:compatibleorder}
    A preorder $\preceq$ on a semigroupoid $\mathcal{S}$ is \emph{compatible} (with $\mathcal{S}$) if
    \begin{enumerate}[label=(\roman*)]
        \item\label{def:compatibleorder1} $a\preceq b$ implies $a\leq b$;
        \item\label{def:compatibleorder2} $a\preceq b$ implies $ax\preceq bx$ and $ya\preceq yb$ whenever $(a,x),(y,a)\in\mathcal{S}^{(2)}$.
    \end{enumerate}
    (note that $bx$ and $yb$ in \ref{def:compatibleorder2} are defined because, by \ref{def:compatibleorder1}, $a\leq b$, so $\so(a)=\so(b)$ and $\ra(a)=\ra(b)).$
\end{definition}

Any compatible preorder $\preceq$ on a semigroupoid $\mathcal{S}$ is preserved by inverses: If $a\preceq b$, then $a\leq b$, so
\[a^*=b^*ab^*\preceq b^*bb^*=b^*\]

\begin{example}\label{ex:setorderofbisections}.
    Let $\mathcal{S}$ be an étale inverse semigroupoid and let $\mathbf{B}(\mathcal{S})$ be its semigroup of open bisections. Then set inclusion $\subseteq$ is a compatible preorder compatible with $\mathbf{B}(\mathcal{S})$. Set inclusion coincides with the canonical order of $\mathbf{B}(\mathcal{S})$ if and only if $\mathcal{S}$ is a groupoid.
    
    Indeed, suppose that $\mathcal{S}$ is not a groupoid. Then the canonical order of $\mathcal{S}$ is not equality, i.e., there are $a\neq b$ in $\mathcal{S}$ with $a\leq b$. Take any two open bisections $A_0,B_0\in\mathbf{B}(\mathcal{S})$ containing $a$ and $b$, respectively. Then $A\defeq A_0\cap B_0^{\uparrow,\leq}$ and $B\defeq B_0\cap A^{\downarrow,\leq}$ are also open bisections of $\mathcal{S}$ containing $a$ and $b$, respectively (see Corollary \ref{cor:propertiesopeninversesemigroupoid}). However, $B\leq A$ in the canonical order of $\mathbf{B}(\mathcal{S})$, but $B$ is not contained in $A$.
    
    The verification of the other direction -- that if $\G$ is an étale groupoid then the canonical order of $\mathbf{B}(\mathcal{G})$ is set inclusion -- is straightforward.
\end{example}

\begin{lemma}\label{lem:relation.of.germs.is.a.graphed.congruence}
If $\preceq$ is a compatible preorder on an inverse semigroupoid $\mathcal{S}$ then it is conditionally downwards directed, and $\sim_{\preceq}$ is a graphed congruence.
\end{lemma}
\begin{proof}
    To prove that $\preceq$ is conditionally downwards directed, suppose $x,y\preceq a$. Then $x,y\leq a$ as well. The product $z\defeq xy^*y$ is thus defined, and
    \[z=xy^*y\preceq ay^*y=y,\qquad\text{and}\qquad z=xy^*y\preceq xa^*a=x.\]
    (In fact, note that any $\preceq$-lower bound $w$ of $\left\{x,y\right\}$ satisfies $w=ww^*w\preceq xy^*y=z$, so $z=\inf_{\preceq}\left\{x,y\right\}$.)
    
    Property \ref{def:compatibleorder}\ref{def:compatibleorder1} implies that the source and range maps are $\sim_{\preceq}$-invariant, so $\sim_{\preceq}$ is graphed. If $a_i\sim_{\preceq} b_i$ ($i=1,2$) and $(a_1,a_2)\in\mathcal{S}^2$, we choose $c_i\preceq a_i,b_i$, so applications of \ref{def:compatibleorder}\ref{def:compatibleorder1} imply $c_1c_2\preceq a_1a_2$ and $c_1c_2\preceq b_1b_2$, thus $a_1a_2\sim_{\preceq} b_1b_2$, which proves that $\sim_{\preceq}$ is a congruence.\qedhere
\end{proof}

\begin{remark}
The germ relation associated to a compatible order does not determine the order completely. For example, let $L_3=\left\{0,1,2\right\}$ be the lattice with $0<1<2$. Set $x\preceq y$ if and only if $x=y$ or $x=a$. Then $\preceq$ is a compatible order on $L_3$, different from the canonical order $\leq$ but with the same germ relation (namely, all elements of $L_3$ are equivalent).
\end{remark}

Therefore, given an inverse semigroupoid $\mathcal{S}$ and a compatible order $\preceq$, we may construct the quotient inverse semigroupoid $\mathcal{S}/\!\!\sim_{\preceq}$. If $\mathcal{S}$ is étale, we wish to apply Proposition \ref{prop:equivalecesetalequotient} to make $\mathcal{S}/\!\!\sim_{\preceq}$ étale as well.

\begin{definition}
A \emph{topologically compatible} order $\preceq$ on a topological inverse semigroupoid $\mathcal{S}$ is a compatible order $\preceq$ such that upper and lower closures of open sets are also open, i.e., if $A\subseteq\mathcal{S}$ is open, then
\[A^{\uparrow,\preceq}=\left\{z\in\mathcal{S}:a\preceq z\text{ for some }a\in A\right\}\quad\text{and}\quad A^{\downarrow,\preceq}=\left\{z\in\mathcal{S}:z\preceq a\text{ for some }a\in A\right\}\]
are open in $\mathcal{A}$.
\end{definition}

\begin{example}
Let $X=[0,1]$ as a unit topological groupoid and its usual topology. Consider the discrete lattice $L_2=\left\{0,1\right\}$ with $0<1$ and the topological semigroupoid $\mathcal{S}\defeq L_2\times X$. Then $\mathcal{S}$ is an étale inverse semigroupoid. The order
\[x\preceq y\iff x=y\text{ or }x=(0,1)\text{ and }y=(1,1)\]
is compatible but not topologically compatible with $\mathcal{S}$. Moreover, it is not hard to verify that the quotient $\mathcal{S}/\!\!\sim_{\preceq}$ is a topological, non-étale, inverse semigroupoid when endowed with the quotient topology.
\end{example}

At this moment, we do not have an example of a compatible order $\preceq$ on a topological inverse semigroupoid $\mathcal{S}$ such that the product map on $\mathcal{S}/\!\!\sim_{\preceq}$ is not continuous with respect to the quotient topology.

Given a topologically compatible order $\preceq$ on a topological inverse semigroupoid $\mathcal{S}$, the $\sim_{\preceq}$-saturation of an open subset $A\subseteq\mathcal{S}$ is
\[\sim_{\preceq}[A]=(A^{\downarrow,\preceq})^{\uparrow,\preceq}.\]
which is open. Therefore $\sim_{\preceq}$ is open, and as a particular case of Proposition \ref{prop:equivalecesetalequotient}, we obtain the following result.

\begin{proposition}
Suppose that $\preceq$ is a topologically compatible order on a topological (resp.\ étale) inverse semigroupoid $\mathcal{S}$. Then the quotient topology on $\mathcal{T}=\mathcal{S}/\!\!\sim_{\preceq}$ makes it a topological (resp.\ étale) inverse semigroupoid. (We endow $\mathcal{S}^{(0)}=\mathcal{T}^{(0)}$ with its original topology.)
\end{proposition}

Given an étale inverse semigroupoid $\mathcal{S}$, the canonical order $\leq$ of $\mathcal{S}$ is topologically compatible, by Corollary \ref{cor:propertiesopeninversesemigroupoid}\ref{cor:propertiesopeninversesemigroupoid2}, and thus we may apply the result above.

\begin{definition}
    The \emph{initial groupoid} of an étale inverse semigroupoid is the topological (semi)groupoid $\IG(\mathcal{S})=\mathcal{S}/\!\!\sim_{\leq}$.
\end{definition}

To check that $\IG(\mathcal{S})$ is indeed a groupoid, we may simply verify that its canonical order is trivial. Denote by $\pi_{\mathcal{S}}\colon\mathcal{S}\to\IG(\mathcal{S})$ the (usual) quotient map. If $\pi_{\mathcal{S}}(a)\leq\pi_{\mathcal{S}}(b)$, then $\pi_{\mathcal{S}}(a)=\pi_{\mathcal{S}}(ba^*a)$. Since $ba^*a\leq b$ then $ba^*a$ and $b$ are $\sim_{\leq}$ related, hence $\pi_{\mathcal{S}}(b)=\pi_{\mathcal{S}}(ba^*a)=\pi_{\mathcal{S}}(a)$.

Moreover, $\IG(\mathcal{S})$ satisfies a universal property for semigroupoid homomorphisms from $\mathcal{S}$ to topological groupoids (and hence its name): Let $\pi_{\mathcal{S}}\colon\mathcal{S}\to\IG(\mathcal{S}$ be the canonical quotient map. Suppose that $\phi\colon\mathcal{S}\to\G$ is a continuous semigroupoid homomorphism, where $\G$ is a topological groupoid. As $\phi$ preserves the orders and the order of $\G$ is trivial, then $\phi$ is $\sim_{\leq}$-invariant, and thus $\phi$ factors through $\pi_\mathcal{S}$ to a continuous groupoid homomorphism $\psi\colon \IG(\mathcal{S})\to\G$, i.e., $\phi=\psi\circ\pi_{\mathcal{S}}$.

In particular, if $\mathcal{T}$ is another étale inverse semigroupoid and $\phi\colon\mathcal{S}\to\mathcal{T}$ is a continuous semigroupoid homomorphism, then the composition $\pi_T\circ\phi\colon\mathcal{S}\to\IG(\mathcal{T})$ factors uniquely through $\IG(\mathcal{S})$, so we obtain a continuous groupoid homomorphism $\IG(\phi)\colon\IG(\mathcal{S})\to\IG(\mathcal{T})$. Therefore we have a functor $\IG\colon\cat{EtIS}\to\cat{EtGr}$, and the previous paragraph proves that it is initial with respect to the inclusion of categories $\iota\colon\cat{EtGr}\hookrightarrow\cat{EtIS}$.

\begin{proposition}
Let $\mathcal{S}$ be an étale inverse semigroupoid. Then for every étale groupoid $\G$, every continuous semigroupoid morphism $\phi\colon \mathcal{S}\to\G$ factors through a continuous groupoid homomorphism $\IG(\mathcal{S})\to\G$. In other words, there exists a unique continuous groupoid homomorphism $\psi\colon\IG(\mathcal{S})\to\G$ such that the following diagram commutes:
\[\begin{tikzpicture}
\node (S) at (0,0) {$S$};
\node (igS) at ([shift={+(2.5,0)}]S) {$\iota(\IG(S))$};
\node (iG) at ([shift={+(0,-1.5)}]igS) {$\iota(\G)$};
\node (gS) at ([shift={+(2.5,0)}]igS) {$\IG(S)$};
\node (G) at ([shift={+(0,-1.5)}]gS) {$\G$};
\draw[->,dashed] (gS)--(G) node[right,midway] {$\psi$};
\draw[->] (igS)--(iG) node[right,midway] {$\iota(\psi)$};
\draw[->] (S)--(igS) node[above,midway] {$\pi_S$};
\draw[->] (S)--(iG) node[below,midway] {$\phi$};
\end{tikzpicture}\]
\end{proposition}

If $\G$ is a groupoid, then its order is trivial and the associated germ relation is the identity, hence the restriction of $\IG$ to $\cat{EtGr}$ is naturally isomorphic to the identity functor. In this manner, $\IG$ may also be regarded as a retraction from $\cat{EtIS}$ onto its full subcategory $\cat{EtGr}$.

\begin{example}
If $S$ is a discrete inverse semigroup, then $\IG(S)$ is called the \emph{maximal group homomorphic image} of $S$. See \cite[Proposition 2.1.2]{MR1724106}.
\end{example}

\begin{example}
If $\theta\colon S\curvearrowright X$ is a continuous $\land$-preaction of an inverse semigroup $S$ on a topological space $X$, then we have seen in Example \ref{ex:semidirectproductofactionisetale} the semidirect product $S\ltimes X$ is an étale inverse semigroupoid. The groupoid $\IG(S\ltimes X)$ is the \emph{groupoid of germs} of $\theta$, as considered in \cite{arxiv1804.00396}.
\end{example}

\subsubsection*{Quotients of semidirect products}

Let $(\pi,\theta)\colon\mathcal{S}\curvearrowright\mathcal{T}$ be an open, continuous $\land$-preaction, where $\mathcal{S}$ and $\mathcal{T}$ are topological inverse semigroupoids. We will now analyse how the operations of ``taking quotients'' and ``taking semidirect products'' behave with respect to each other. Abusing language, we might simply say that ``quotients and semidirect products commute''.

Suppose that $R_1$ and $R_2$ are graphed congruences on $\mathcal{S}$ and $\mathcal{T}$, respectively. We denote equivalence classes, either in $\mathcal{S}$ or $\mathcal{T}$, simply by brackets, if $a\in\mathcal{S}$ then $[a]$ denotes its $R_1$-equivalence class, and similarly in $\mathcal{T}$.

On one direction, consider the equivalence relation $R_1\times R_2$ on $\mathcal{S}\ltimes\mathcal{T}$ -- namely $(s_1,t_1)$ is $(R_1\times R_2)$-equivalent to $(s_2,t_2)$ if and only if $s_1$ is $R_1$-equivalent to $s_2$ and $t_1$ is $R_2$-equivalent to $t_2$. We wish that $R_1\times R_2$ is a congruence on $\mathcal{S}\ltimes\mathcal{T}$, and for this we need to impose conditions on the congruences $R_1$ and $R_2$.

The $(R_1\times R_2)$-class of $(a,x)\in\mathcal{S}\ltimes\mathcal{T}$ is denoted $[a,x]$.

\begin{definition}
    A graphed congruence $R$ on an inverse semigroupoid $R$ is \emph{idempotent pure} if $(a,e)\in R$ and $e\in E(\mathcal{S})$ implies $a\in\mathcal{E(\mathcal{S})}$ (i.e., the saturation of $E(\mathcal{S})$ is $E(\mathcal{S})$).
\end{definition}

The following equivalences are well-known in the case of inverse semigroups, and the proofs are easy enough, so we omit them.

\begin{proposition}\label{prop:equivalencesidempotentpure}
The following are equivalent:
\begin{enumerate}[label=(\roman*)]
    \item\label{prop:equivalencesidempotentpure1} $R$ is idempotent pure;
    \item\label{prop:equivalencesidempotentpure2} The canonical quotient map $\pi\colon\mathcal{S}\to\mathcal{S}/R$ is idempotent pure, in the sense that $\pi^{-1}(E(\mathcal{S}/R))=E(\mathcal{S})$.
    \item\label{prop:equivalencesidempotentpure3} If $(a,b)\in R$, then $a^*b\in E(\mathcal{S})$ (and also $ab^*\in E(\mathcal{S})$);
\end{enumerate}
\end{proposition}

We therefore make two standing hypotheses on the congruence $R_1$ and $R_2$, and the action $(\pi,\theta)$.

\begin{enumerate}[label=(H\arabic*)]
\item\label{hyp:onrelationsforsemidirectproduct1} $R_1$ is idempotent pure;
\item\label{hyp:onrelationsforsemidirectproduct2} $\theta$ is an $R_2$-morphism: For all $a\in\mathcal{S}$ and $x,y\in\dom(\theta_a)$, if $(x,y)\in R_2$ then $(\theta_a(x),\theta_a(y))\in R_2$.
\end{enumerate}

\begin{lemma}\label{lem:hypmakeformulasgood}
Under hypotheses \ref{hyp:onrelationsforsemidirectproduct1} and \ref{hyp:onrelationsforsemidirectproduct2}, $(R_1\times R_2)$ is a graphed congruence on $\mathcal{S}\ltimes\mathcal{T}$. Moreover, if $[a,x]=[b,y]$ then $[\theta_a(x)]=[\theta_b(y)]$.
\end{lemma}
\begin{proof}
    We start with the last statement. Suppose $[a,x]=[b,y]$. Then
    $[x]=[y]=[xy^*y]$, and $xy^*y$ belongs to both $\dom(\theta_a)$ and $\dom(\theta_b)$. As $R_1$ is idempotent pure, $\theta_a$ and $\theta_b$ are compatible, and thus coincide on the intersection of their domains. Since $\theta$ is an $R_2$-morphism then
    \[[\theta_a(x)]=[\theta_a(xy^*y)]=[\theta_b(xy^*y)]=[\theta_b(y)],\]
    as desired.
    
    We can now prove that the source and range maps of $\mathcal{S}\ltimes\mathcal{T}$ are $R_1\times R_2$-invariant. Assuming $[a,x]=[b,y]$, we have $[\theta_a(x)]=[\theta_b(y)]$ by the previous paragraph. As $R_2$ is graphed,
    \[\ra(a,x)=\ra(\theta_a(x))=\ra(\theta_b(y))=\ra(b,y).\]
    and $\so(a,x)=\so(x)=\so(y)=\so(b,y)$.
    
    It remains only to prove that if $[a,x]=[b,y]$, $[c,z]=[d,w]$, and the product $(a,x)(c,z)$ is defined, then $[(a,x)(c,z)]=[(b,y)(d,w)]$. More specifically, we need to prove that
    \[[ac,\theta_c^{-1}(x\theta_c(z))]=[bd,\theta_d^{-1}(y\theta_d(w))]\]
    Again using the last part of the lemma, we have $[\theta_c(z)]=[\theta_d(w)]$, so $[x\theta_c(z)]=[y\theta_d(w)]$ because $R_2$ is a congruence. Applying the last part one more time yields
    \[[\theta_c^{-1}(x\theta_c(z))]=[\theta_d^{-1}(y\theta_d(w))]\]
    and also $[ac,bd]$ because $R_1$ is a congruence.\qedhere
\end{proof}

The lemma above allows us to construct the quotient of the semidirect product, $(\mathcal{S}\ltimes\mathcal{T})/R_1\times R_2$. We now do the opposite, namely we construct a semidirect product of the quotients, $(\mathcal{S}/R_1)\ltimes(\mathcal{T}/R_2)$. First we need to describe the $\land$-preaction of $\mathcal{S}/R_1$ on $\mathcal{T}/R_2$.

The anchor map $\pi\colon\mathcal{T}\to\mathcal{S}^{(0)}$ is a homomorphism, so if $\ra(x)=\ra(y)$, then $x^*y$ is defined, so $\pi(x)=\pi(x^*)=\pi(y)$. Since $R_2$ is graphed then $\pi$ factors unique to a homomorphism, $\mathcal{T}/R_2\to\mathcal{S}^{(0)}$, which we also denote $\pi$.

We now define the action of $\mathcal{S}/R_1$ on $\mathcal{T}/R_2$, which we denote by $\Theta$. Given an $R_1$-class $\alpha\in\mathcal{S}/R_1$, consider the subsets of $\mathcal{T}/R_2$
\[D_\alpha\defeq\left\{[x]:x\in\dom(\theta_b)\text{ for some }b\in\mathcal{S}\text{ with }(b,\alpha)\in R_1\right\}.\]
and
\[R_\alpha\defeq\left\{[x]:x\in\ran(\theta_b)\text{ for some }b\in\mathcal{S}\text{ with }(b,\alpha)\in R_1\right\}.\]
Note that $R_\alpha=D_{\alpha^*}$, and that $D_\alpha$ is a subset of $\pi^{-1}(\so(\alpha))$.

By Lemma \ref{lem:hypmakeformulasgood}, we may define a map $\Theta_{\alpha}\colon D_\alpha\to R_\alpha$ by
\[\Theta_\alpha([x])=[\theta_a(x)],\qquad\text{ whenever }a\in\alpha\text{ and }x\in\dom(\theta_a).\]
We have already seen that if $a,b\in\alpha$ then $\theta_a$ and $\theta_b$ coincide on their common domain, and from this it follows that $\Theta_\alpha$ is a bijection from $D_\alpha$ to $R_\alpha$. We will prove that it is a $\land$-preaction, in the sense of Definition \ref{def:partialactiononsemigroupoid}. For this, let us make a small detour to the theory of \emph{$\lor$-prehomomorphisms}, which are dual to $\land$-prehomomorphisms.

\begin{definition}
A \emph{$\lor$-prehomomorphism} between inverse semigroupoids is a map $\theta\colon\mathcal{S}\to\mathcal{T}$ which satisfies $(\theta\times\theta)(\mathcal{S}^{(2)})\subseteq\mathcal{T}^{(2)}$ and $\theta(ab)\leq\theta(a)\theta(b)$ whenever $(a,b)\in\mathcal{S}^{(2)}$.
\end{definition}

\begin{proposition}\label{prop:invertiblelorprehomarehom}
If $\theta\colon\mathcal{S}\to\mathcal{T}$ is a $\land$-prehomomorphism, then $\theta$ preserves idempotents, the canonical order, and inverses.

If $\theta$ is invertible and $\theta^{-1}$ is also a $\lor$-prehomomorphism, then $\theta$ is an isomorphism.
\end{proposition}
\begin{proof}
    The statements in the first paragraph may be proven just as in the case of inverse semigroups; see \cite[Proposition 3.1.5]{MR1694900}. Let us prove the last statement. Suppose $\theta$ is invertible and $\theta^{-1}$ is also a $\lor$-prehomomorphism.
    
    Given $(a,b)\in\mathcal{S}^{(2)}$, we have
    \[\theta(ab)\leq\theta(a)\theta(b).\]
    Since $\theta^{-1}$ is a $\lor$-prehomomorphism, and in particular it preserves the order,
    \[ab\leq\theta^{-1}(\theta(a)\theta(b))\leq\theta^{-1}(\theta(a))\theta^{-1}(\theta(b))=ab\]
    so $ab=\theta^{-1}(\theta(a)\theta(b))$, or equivalently, $\theta(ab)=\theta(a)\theta(b)$.\qedhere
\end{proof}

\begin{proposition}
If $(\pi,\theta)$ is a $\land$-preaction then $(\pi,\Theta)$ is a $\land$-preaction.
\end{proposition}
\begin{proof}
    We simply need to verify that if $(\pi,\theta)$ satisfies any of properties \ref{def:dualprehomomorphismandpartialhomomorphism}\ref{def:dualprehomomorphismandpartialhomomorphism1}, \ref{def:dualprehomomorphismandpartialhomomorphism2} or \ref{def:dualprehomomorphismandpartialhomomorphism3}, then $(\pi,\Theta)$ satisfies the same property.
    
    The work for \ref{def:dualprehomomorphismandpartialhomomorphism}\ref{def:dualprehomomorphismandpartialhomomorphism1} is straightforward, i.e., if $\theta_a^{-1}=\theta_{a^*}$ for all $a\in\mathcal{S}$, then $\Theta_\alpha^{-1}=\Theta_{\alpha^*}$.
    
    Assume then that $(\pi,\theta)$ satisfies \ref{def:dualprehomomorphismandpartialhomomorphism}\ref{def:dualprehomomorphismandpartialhomomorphism2}: For all $(a,b)\in\mathcal{S}^{(2)}$, $\theta_a\circ\theta_b\leq\theta_{ab}$.
    
    Suppose $(\alpha,\beta)\in(\mathcal{S}/R_1)^{(2)}$ and $\gamma\in\dom(\Theta_\alpha\circ\Theta_\beta)$. Then we can find $b\in\beta$ and $x\in\gamma\cap\dom(\theta_b)$ such that $\Theta_\beta(\gamma)=[\theta_b(x)]$, and this belongs to $\dom(\Theta_\alpha)$. Then, find $a\in\alpha$ and $y\in\dom(\theta_a)$ such that $[\theta_b(x)]=[y]$ and $\Theta_\alpha(\Theta_\beta(\gamma))=[\theta_a(y)]$.
    
    We have $[y]=[\theta_b(x)]=[\theta_b(x)y^*y]$, and the element $\theta_b(x)y^*y$ belongs to both $\ran(\theta_b)$ and $\dom(\theta_a)$. Then we can consider $z\in\dom(\theta_b)$ (namely, $z=\theta_{b^*}(\theta_b(x)y^*y)$) such that $\theta_b(x)y^*y=\theta_b(z)$. Moreover, as $[\theta_b(x)]=[\theta_b(z)]$ then $[x]=[z]$, because $\theta$ is an $R_2$-morphism. Then $z\in\dom(\theta_a\circ\theta_b)$, so $\gamma=[z]\in\dom(\Theta_\alpha\circ\Theta_\beta)$, and
    \[\Theta_\alpha(\Theta_\beta(\gamma))=\Theta_\alpha([\theta_b(z)])=[\theta_a(\theta_b(z))]=[\theta_{ab}(z)]=\Theta_{\alpha\beta}(\gamma).\]
    This proves that $\Theta_\alpha\circ\Theta_\beta\leq\Theta_{\alpha\beta}$, as we wanted.
    
    It remain only to prove that each map $\Theta_\alpha$ is an isomorphism between ideals of $\mathcal{T}/R_2$. It should be clear that $D_\alpha$ is an ideal for each $\alpha\in\mathcal{S}/R_1$. Moreover, recall that as $R_2$ is graphed, then a product $[x][y]$ is defined in $\mathcal{T}/R_2$ if and only if $xy$ is defined in $\mathcal{T}$, in which case $[x][y]=[xy]$.
    
    Let $\alpha\in\mathcal{S}/R_1$. First we prove that $\Theta_\alpha$ preserves idempotents. Indeed, if $\gamma\in\dom(\Theta_\alpha)$ is idempotent, take $a\in\alpha$ and $x\in\gamma\cap\dom(\theta_a)$. As $[x]=\gamma=\gamma^*\gamma=[x^*x]$, we may assume that $x\in E(\mathcal{S})$, so $\theta_a(x)\in\mathcal{E}(\mathcal{T})$, because $\theta_a$ is an isomorphism, and thus $\Theta_\alpha(\gamma)=[\theta_a(x)]\in E(\mathcal{T}/R_2)$.
    
    We may now prove that $\Theta_{\alpha}$ is a $\lor$-prehomomorphism. Suppose that $\gamma,\delta\in\dom(\Theta_\alpha)$. Then are $a,b\in\alpha$ and $x\in\gamma\cap\dom(\theta_a)$ and $y\in\delta\cap\dom(\theta_b)$. If $\gamma\delta$ is defined and it also belongs to $\dom(\Theta_\alpha)$, so there is $c\in\alpha$ and $z\in(\gamma\delta)\cap\dom(\theta_c)$. We then apply Lemma \ref{lem:hypmakeformulasgood} several times, which allows us to compute
    \begin{align*}
        \Theta_\alpha(\gamma\delta)&=[\theta_c(z)]=[\theta_a(xy)]=[\theta_a(x)][\theta_a(x^*xy)]=\Theta_\alpha(\gamma)[\theta_b(x^*xy)]=\Theta_\alpha(\gamma)[\theta_b(x^*xyy^*)][\theta_b(y)]\\
        &=\Theta_\alpha(\gamma)\Theta_\alpha([x^*xyy^*])\Theta_\alpha(\delta).
    \end{align*}
    Since $\Theta_\alpha$ preserves idempotents, then $\Theta_\alpha(\gamma\delta)\leq\Theta_\alpha(\gamma)\Theta_\alpha(\delta)$, so $\Theta_\alpha$ is a $\lor$-prehomomorphism. The same holds for $\Theta_{\alpha^*}=\Theta_\alpha^{-1}$, so by Proposition \ref{prop:invertiblelorprehomarehom}, $\Theta_\alpha$ is an isomorphism.\qedhere.
\end{proof}

\begin{remark}
    The analogous result to the proposition above for partial actions is not true. The following example was considered in \cite[Example 4.5]{Mikola2017}. Let $E=\left\{0,a,b\right\}$ be the semilattice with $x\leq y$ if and only if $x=0$ or $x=y$, acting on the lattice $L_2=\left\{0,1\right\}$, where $0<1$, as follows: $\theta_0$ and $\theta_a$ are the identity of $\left\{0\right\}$ and $\theta_b=\id_{L_2}$. Then, in fact, $\theta$ is a global action.
    
    We let $R_1$ be the congruence which identifies $0$ and $b$ in $E$, and $R_2$ the identity of $L_2$. Again denoting by $\Theta$ the $\land$-preaction of $E/R_1$ on $L_2/R_2\cong T$, we have $[0]\leq[a]$, but $\Theta_{[0]}=\theta_b$, $\Theta_{[a]}=\theta_a$, and $\theta_a<\theta_b$.
\end{remark}

We will now consider the topological setting. The two standing hypotheses \ref{hyp:onrelationsforsemidirectproduct1} and \ref{hyp:onrelationsforsemidirectproduct2} are still maintained.

\begin{lemma}\label{lem:topologicalpropertiesofacctionpasstoquotients}
Suppose that $\mathcal{S}$ is étale, $\mathcal{T}$ is topological, $(\pi,\theta)$ is a continuous open $\land$-preaction, and $R_1$ and $R_2$ are open. Let $\pi_1\colon\mathcal{S}\to\mathcal{S}/R_1$, $\pi_2\colon\mathcal{T}\to\mathcal{T}/R_2$ be the canonical quotient maps. Then
\begin{enumerate}[label=(\alph*)]
    \item\label{lem:topologicalpropertiesofacctionpasstoquotients1} $\pi_1\times\pi_2\colon\mathcal{S}\ltimes\mathcal{S}\to(\mathcal{S}/R_1)\ltimes(\mathcal{T}/R_2)$ is surjective, continuous and open. In particular, $(\mathcal{S}/R_1)\ltimes(\mathcal{T}/R_2)$ carries the quotient topology of $\pi_1\times\pi_2$.
    \item\label{lem:topologicalpropertiesofacctionpasstoquotients2} $R_1\times R_2$ is open.
    \item\label{lem:topologicalpropertiesofacctionpasstoquotients3} $(\pi,\Theta)$ is open and continuous.
\end{enumerate}
\end{lemma}
\begin{proof}
    Consider the map $p\colon\mathcal{S}\ltimes\mathcal{T}\to\mathcal{T}$, $p(a,x)=x$, which we know to be a from Proposition \ref{prop:preactionisopeniffprojectionisopen}.
    \begin{enumerate}[label=\ref{lem:topologicalpropertiesofacctionpasstoquotients\arabic*}]
    \item Surjectivity and continuity of $\pi_1\times\pi_2$ should be clear, so we prove that it is open. A basic open set of $\mathcal{S}\ltimes\mathcal{T}$ has the form $A\ast U$, where $A$ is an open bisection of $\mathcal{S}$ and $U$ is open in $\mathcal{T}$. Let us prove that
    \[(\pi_1\times\pi_2)(A\ast U)=\pi_1(A)\ast\pi_2(p(A\ast U)).\ntag\label{eq:proofofisomorphismofquotientandsemidirectproduct}\]
    Indeed, the inclusion ``$\subseteq$'' is immediate. Conversely, an element of the right-hand side may be writtens in the form $([a],[x])$, where $x\in\dom(\theta_b)$ for some $b\in A$, and also in the form $([c],[z])$ where $z\in\dom(\theta_c)$. As $R_2$ is graphed then $\ra(x)=\ra(z)$, so $\pi(x)=\pi(z)$. As $R_1$ is also graphed, then
    \[\so(a)=\so(c)=\pi(z)=\pi(x)=\so(b)\]
    and therefore $a=b$ as $A$ is a bisection. Therefore $(a,x)\in\mathcal{S}\ltimes \mathcal{T}$ and $([a],[x])=(\pi_1\times\pi_2)(a,x)$.
    
    Since $\pi_1$, $\pi_2$ and $p$ are open maps then $(\pi_1\times\pi_2)$ is also open.
    \item $R_1\times R_2$ is the kernel of $\pi_1\times\pi_2$, so the $R_1\times R_2$-saturation of an open subset $U\subseteq\mathcal{S}\times\mathcal{T}$ is $(\pi_1\times\pi_2)^{-1}(\pi_1\times\pi_2(U))$, which is open by the previous item.
    \item We have the commutative diagram
        \[\begin{tikzpicture}
        \node (ST) at (0,0) {$\mathcal{S}\ltimes\mathcal{T}$};
        \node (T) at ([shift={+(3,0)}]ST) {$\mathcal{T}$};
        \node (SRTR) at ([shift={+(0,-1)}]ST) {$(\mathcal{S}/R_1)\ltimes(\mathcal{T}/R_2)$};
        \node (TR) at ([shift={+(0,-1)}]T) {$\mathcal{T})/R_2$};
        \draw[->] (ST)--(T) node[midway,above] {$\theta$};
        \draw[->] (ST)--(SRTR) node[midway,left] {$\pi_1\times\pi_2$};
        \draw[->] (T)--(TR) node[midway,right] {$\pi_2$};
        \draw[->] (SRTR)--(TR) node[midway,above] {$\Theta$};
    \end{tikzpicture}\]
    where the horizontal arrows are the action maps and the vertical ones are topological quotient maps. Since $\theta$ is continuous and open then $\Theta$ is continuous and open.\qedhere
    \end{enumerate}
\end{proof}

\begin{theorem}
Assuming hypotheses \ref{hyp:onrelationsforsemidirectproduct1} and \ref{hyp:onrelationsforsemidirectproduct2}, that $\mathcal{S}$ is étale, $\mathcal{T}$ is topological, $(\pi,\theta)$ is a continuous open $\land$-preaction, and that $R_1$ and $R_2$ are open, the map
\[\Phi\colon(\mathcal{S}\ltimes\mathcal{T})/(R_1\times R_2)\to(\mathcal{S}/R_1)\ltimes(\mathcal{T}/R_2),\qquad\Phi([a,x])=([a],[x])\]
is a topological semigroupoid isomorphism.
\end{theorem}
\begin{proof}
    The fact that $\Phi$ is a semigroupoid isomorphism is immediate. To check that it is also a homeomorphism, consider the commutative diagram
    \[\begin{tikzpicture}
        \node (1ST) at (0,0) {$\mathcal{S}\ltimes\mathcal{T}$};
        \node (SRTR) at ([shift={+(0,-1)}]1ST) {$(\mathcal{S}/R_1)\ltimes(\mathcal{T}/R_2)$};
        \node (2ST) at ([shift={+(-4,0)}]1ST) {$\mathcal{S}\ltimes\mathcal{T}$};
        \node (STRR) at ([shift={+(0,-1)}]2ST) {$(\mathcal{S}\ltimes\mathcal{T})/(R_1\times R_2)$};
        \draw[->] (2ST)--(1ST) node[midway,above] {$\id$};
        \draw[->] (1ST)--(SRTR) node[midway,right] {$\pi_1\times\pi_2$};
        \draw[->] (2ST)--(STRR) node[midway,left] {$\pi$};
        \draw[->] (STRR)--(SRTR) node[midway,above] {$\Phi$};
    \end{tikzpicture}\]
    where $\pi_1$, $\pi_2$ and $\pi$ are canonical quotient maps. The vertical arrows are topological quotient maps, so and $\id$ is continuous and open, therefore $\Phi$ is continuous and open as well, hence a homeomorphism.\qedhere
\end{proof}

\subsection{Global actions}

In this subsection, let us consider global open actions of étale inverse semigroupoids on sets.

\begin{example}\label{ex:actiononvertexspace}
Every étale inverse semigroupoid $\mathcal{S}$ admits a canonical continuous, open action on its vertex space.

Let $\id_{\mathcal{S}^{(0)}}\colon\mathcal{S}^{(0)}\to\mathcal{S}^{(0)}$ be the trivial bundle. Define  $\tau\colon\mathcal{S}\to\mathcal{I}(\id_{\mathcal{S}^{(0)}})$ as $\tau_a(\so(a))=\ra(a)$ for all $a\in\mathcal{S}^{(0)}$ (that is, $\dom(\tau_a)=\left\{\so(a)\right\}$ and $\ran(\tau_a)=\left\{\ra(a)\right\})$.

As $\mathcal{S}$ is étale, it is easy to verify that $\mathbb{T}(\mathcal{S})\defeq(\id_{\mathcal{S}^{(0)}},\tau)$ is a continuous open action.
\end{example}

Denote by $\cat{Act}$ the category of continuous open actions of étale inverse semigroupoids on topological spaces. A morphism between actions $(\pi^i,\theta^i)\colon\mathcal{S}_i\curvearrowright X_i$ ($i=1,2$) is a pair $(f,\phi)$, where $f\colon X_1\to X_2$ is a continuous function and $\phi\colon\mathcal{S}_1\to\mathcal{S}_2$ is a continuous homomorphism, which are equivariant in the sense that for all $a\in\mathcal{S}$, $f\circ\theta^1_a\leq \theta^2_{\phi(a)}\circ f$ (i.e., $\theta^2_{\phi(a)}\circ f$ is an extension of $f\circ\theta_1^a$).

The composition of two morphisms is simply $(f_1,\phi_1)(f_2,\phi_2)=(f_1\circ f_2,\phi_1\circ\phi_2)$.

We will construct a pair of adjunct functors between $\cat{Act}$ and $\cat{EtIS}$. Let us denote by $\mathbb{T}(\mathcal{S})=(\id_{\mathcal{S}^{(0)}},\tau)$ the action constructed in Example \ref{ex:actiononvertexspace}, for an étale inverse semigroupoid $\mathcal{S}$. We already know that any continuous étale semigroupoid homomorphism $\phi\colon\mathcal{S}_1\to\mathcal{S}_2$ induces a continuous map on the vertex sets $\phi^{(0)}\colon\mathcal{S}_1^{(0)}\to\mathcal{S}_2^{(0)}$, by $\phi^{(0)}(\so(a))=\so(\phi(a))$. It is easy to see that $\mathbb{T}(\phi)\defeq(\phi^{(0)},\phi)$ is equivariant, and thus a morphism between the actions $\mathbb{T}(\mathcal{S}_1)$ and $\mathbb{T}(\mathcal{S}_2)$. We thus have a functor $\mathbb{T}\colon\cat{EtIS}\to\cat{Act}$.

In the other direction, given a morphism $(f,\phi)\colon(\pi_1,\theta_1)\to(\pi_2,\theta_2)$ of continuous open actions $(\pi_i,\theta_i)\colon\mathcal{S}_i\curvearrowright \mathcal{T}_i$, we define a continuous semigroupoid homomorphism $\ast(f,\phi)\colon\mathcal{S}_1\ltimes\mathcal{T}_1\to\mathcal{S}_2\ltimes \mathcal{T}_2$ by
\[\ast(f,\phi)(a,x)=(\phi(a),f(x)).\]
Therefore we have a functor $\ast\colon\cat{Act}\to\cat{EtIS}$, taking any action to the associated semidirect product.

\begin{proposition}
$(\mathbb{T},\ast)$ is a pair of adjoint functors. Moreover, $\ast\mathbb{T}$ is (naturally isomorphic to) the identity of $\cat{EtIS}$.
\end{proposition}
\begin{proof}
    We first define a natural transformation $\epsilon\colon\mathbb{T}\ast\to\id_{\cat{Act}}$. Given an action $(\pi,\theta)\colon\mathcal{S}\curvearrowright\mathcal{T}$, consider the homomorphism $\epsilon^{(1)}\colon\mathcal{S}\ast\mathcal{T}\to \mathcal{S}$, $\epsilon^{(1)}(a,x)=a$. Since the vertex set $(\mathcal{S}\ltimes X)^{(0)}$ is (a subset of) $X$, let $\epsilon^{(0)}\defeq\id_X$. Then $\epsilon_{(\pi,\theta)}=(\epsilon^{(0)},\epsilon^{(1)})$ is a morphism of actions, and this defines a natural transformation $\epsilon\colon \mathbb{T}\ast\to\id_{\cat{Act}}$.
    
    In the other direction, given an inverse semigroupoid $\mathcal{S}$, consider $\eta_{\mathcal{S}}\colon\mathcal{S}\to\mathcal{S}\ast\mathcal{S}^{(0)}$ as $\eta_{\mathcal{S}}(a)=(a,\so(a))$. Then $\eta_{\mathcal{S}}$ is actually an isomorphism, and this defines a natural isomorphism $\eta\colon\id_{\cat{EtIS}}\to\ast\mathbb{T}$.
    
    Elementary, although long, computations prove that $(\epsilon,\eta)$ is an adjunction between $\mathbb{T}$ and $\ast$.\qedhere
\end{proof}

\section{Duality}\label{sec:duality}
    By making an analogy with the Non-Commutative Stone Duality of Lawson-Lenz, between certain classes of groupoids and inverse semigroups, we look at the following question: How can one recover an inverse semigroupoid from its set of bisections?

Let $\mathcal{S}$ be a discrete inverse semigroupoid. Then the set $\mathbf{B}(\mathcal{S})$ of bisections of $\mathcal{S}$ is an inverse semigroup under product of sets. However, the natural order in this semigroup does not coincide with set inclusion: Given $A,B\in\mathbf{B}(\mathcal{S})$, we have $A\leq B$ if and only if for all $a\in A$, there exists $b\in B$ with $(\ra(b),\so(b))=(\ra(a),\so(a))$ and $a\leq b$.

As a more concrete example, if $S$ is an inverse semigroup, then the non-empty bisections of $S$ are simply singleton sets, and thus set inclusion coincides with identity: $\left\{a\right\}\subseteq\left\{b\right\}$ if and only if $a=b$. However, as long as $S$ is not a group, there will be distinct elements $a\neq b$ with $a\leq b$.

Therefore, in order to recover $\mathcal{S}$ from $\mathbf{B}(\mathcal{S})$ we will need to consider information about set containment as well.

Following the approach in classical Stone duality, as well as in the initial versions of non-commutative Stone duality for groupoids, we will consider only locally compact étale semigroupoids with Hausdorff and zero-dimensional idempotent spaces.

\begin{lemma}\label{lem:ES.Hausdorff.and.S0.Hausdorff}
    Suppose that $\mathcal{S}$ is a topological inverse semigroupoid.
    \begin{enumerate}[label=(\alph*)]
        \item\label{lem:ES.Hausdorff.and.S0.Hausdorff1} If $E(\mathcal{S})$ is Hausdorff and , then $\mathcal{S}^{(0)}$ is Hausdorff as well.
        \item\label{lem:ES.Hausdorff.and.S0.Hausdorff2} If $\mathcal{S}^{(0)}$ is Hausdorff. then the product of compact subsets of $\mathcal{S}$ is compact.
    \end{enumerate}
\end{lemma}
\begin{proof}
    \begin{enumerate}[label={\ref{lem:ES.Hausdorff.and.S0.Hausdorff\theenumi}}]
    \item Suppose $E(\mathcal{S})$ is Hausdorff. Given $x_1\neq x_2$ in $\mathcal{S}^{(0)}$, choose $e_1,e_2\in E(\mathcal{S})$ with $\so(e_i)=x_i$, so $e_1\neq e_2$. Take any two open bisections $U_1,U_2$ with $e_i\in U_i$. 
    \item We have $\mathcal{S}^{(2)}=\left\{(a,b)\in\mathcal{S}\colon \so(a)=\ra(b)\right\}$, which is closed in $\mathcal{S}\times\mathcal{S}$ as $\so$ and $\ra$ are continuous and $\mathcal{S}^{(0)}$ is Hausdorff.
    
    If $\mu\colon\mathcal{S}^{(2)}\to\mathcal{S}$ denotes the (continuous) product, then for any two compact subsets $K,L\subseteq\mathcal{S}$, $(K\times L)\cap\mathcal{S}^{(2)}$ is closed in the compact $K\times L$, hence compact itself. Therefore $KL=\mu((K\times L)\cap\mathcal{S}^{(2)})$ is compact.\qedhere
    \end{enumerate}
\end{proof}

\begin{example}
Let $X$ be any Hausdorfff topological space with a non-isolated point $x_0$., and let $L_2=\left\{0,1\right\}$ be the lattice with $0\leq 1$. Let $\mathcal{S}$ be the inverse semigroupoid obtained from $L_2\times X$ by identifying $(0,x)$ with $(1,x)$ whenever $x\neq x_0$. Then $\mathcal{S}=E(\mathcal{S})$ is a non-Hausdorff, étale inverse semigroupoid, and $\mathcal{S}^{(0)}=X$ is Hausdorff.

Now consider $Y=\mathcal{S}$ simply as a topological space. Let $E=\left\{0,a,b\right\}$ be the semilattice with $x\leq y\iff x=0\text{ or }x=y$, and $\mathcal{T}=E\times X\setminus\left\{(0,x_0)\right\}$, which we make an inverse semigroupoid over $Y$ as follows: If $\so(a,x_0)$ is the class of $(0,x_0)$ in $Y$, $\so(b,x_0)$ is the class of $(1,x_0)$ in $Y$, and if $x\neq x_0$, then $\so(e,x)$ is the class of $(0,x)$ in $Y$. Then the product of $E\times X$ restricted to $\mathcal{T}$ makes it a Hausdorff, étale inverse semigroupoid, with non-Hausdorff vertex space.

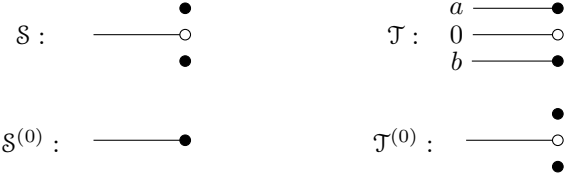
\[\begin{tikzpicture}[scale=0.7]
\node (SS0) at (0,0) {};
\draw[fill=black] ([shift={+(-0.1,0)}]SS0) circle (0.1);
\node (start_X) at ([shift={+(-2,-0.5)}]SS0) {};
\node (end_X) at ([shift={+(2,0)}]start_X) {};
\draw[-] (start_X)--(end_X);
\draw[fill=white] ([shift={+(-0.1,0)}]end_X) circle (0.1);
\node (Y1) at ([shift={+(0,-0.5)}]end_X) {};
\draw[fill=black] ([shift={+(-0.1,0)}]Y1) circle (0.1);
\node (start_S0) at ([shift={+(0,-2)}]start_X) {};
\node (end_S0) at ([shift={+(2,0)}]start_S0) {};
\draw[-] (start_S0)--(end_S0);
\draw[fill=black] ([shift={+(-0.1,0)}]end_S0) circle (0.1);

\node (S) at ([shift={+(-1,0)}]start_X) {$\mathcal{S}:$} ;
\node (S0) at ([shift={+(-1,0)}]start_S0) {$\mathcal{S}^{(0)}:$} ;

\node (start_a) at ([shift={+(5,0)}]SS0) {$a$};
\node (end_a) at ([shift={+(2,0)}]start_a)  {};
\draw[-] (start_a)--(end_a);
\draw[fill=black] ([shift={+(-0.1,0)}]end_a) circle (0.1);
\node (start_0) at ([shift={+(0,-0.5)}]start_a) {$0$} ;
\node (end_0) at ([shift={+(2,0)}]start_0)  {};
\draw[-] (start_0)--(end_0) ;
\draw[fill=white] ([shift={+(-0.1,0)}]end_0) circle (0.1);
\node (start_b) at ([shift={+(0,-0.5)}]start_0) {$b$} ;
\node (end_b) at ([shift={+(2,0)}]start_b)  {};
\draw[-] (start_b)--(end_b) ;
\draw[fill=black] ([shift={+(-0.1,0)}]end_b) circle (0.1);
\node (Y0) at ([shift={+(0,-1)}]end_b) {};
\draw[fill=black] ([shift={+(-0.1,0)}]Y0) circle (0.1);
\node (start_X) at ([shift={+(-2,-0.5)}]Y0) {};
\node (end_X) at ([shift={+(2,0)}]start_X) {};
\draw[-] (start_X)--(end_X);
\draw[fill=white] ([shift={+(-0.1,0)}]end_X) circle (0.1);
\node (Y1) at ([shift={+(0,-0.5)}]end_X) {};
\draw[fill=black] ([shift={+(-0.1,0)}]Y1) circle (0.1);

\node (T) at ([shift={+(-1,0)}]start_0) {$\mathcal{T}:$} ;
\node (T0) at ([shift={+(-1,0)}]start_X) {$\mathcal{T}^{(0)}:$} ;
\end{tikzpicture}
\]
\end{example}
\begin{definition}
    An étale inverse semigroupoid $\mathcal{S}$ is \emph{ample} if $\mathcal{S}^{(0)}$ is locally compact, Hausdorff and zero-dimensional. We denote by $\mathbf{KB}(\mathcal{S})$ the inverse semigroup of compact-open bisections of $\mathcal{S}$.
\end{definition}

We now axiomatize set inclusion for $\mathbf{KB}(\mathcal{S})$. For a preorder $\subseteq$ on an inverse semigroup $S$, we denote respective joins and meets by $a\cup b=\sup_{\subseteq}\left\{a,b\right\}$ and $a\cap b=\inf_{\subseteq}\left\{a,b\right\}$, whenever either exists.

\begin{definition}\label{def:sigmaorder}
A \emph{$\Sigma$-ordered} inverse semigroup is an inverse semigroup with zero $S$ equipped with a compatible order $\subseteq$ satisfying
\begin{enumerate}[label=({$\Sigma$-}\roman*)]
    \item\label{def:sigmaorder.0minimum} $0\subseteq a$ for all $a\in S$.
    \item\label{def:sigmaorder.conditionaljoins} $(S,\subseteq)$ admits \emph{conditional joins}: If $a_1,a_2\subseteq c$, then the join $a_1\cup a_2=\sup_{\subseteq}\left\{a_1,a_2\right\}$ exists.
    \item\label{def:sigmaorder.distributivity} $(S,\subseteq)$ is (finitely) \emph{distributive}: If a join $c_1\cup c_2$ exists and $a\in S$, then $a(c_1\cup c_2)=(ac_1)\cup (ac_2)$ (the latter term exists by the previous item).
    \item\label{def:sigmaorder.ESrelativecomplements} $(E(S),\subseteq)$ admits relative complements: If $e\subseteq f$ in $E(S)$, then there exists a $c\in E(S)$ such that $e\cup c=f$ and $e\cap c=0$. Such $c$ is necessarily unique (see next paragraph), and denoted by $f\setminus e$.
    \item\label{def:sigmaorder.zerojoins} If $a^*b=ab^*=0$, then the join $a\cup b$ exists.
    \item\label{def:sigmaorder.interpolation} If $t\leq a$ in $S$, then there exists $z\subseteq a$ such that $t\leq z\subseteq a$ and
    \[tx=0\iff zx=0\qquad\text{ for all }x\in S.\]
\end{enumerate}
\end{definition}

The proof that relative complements of $(E(S),\subseteq)$ are unique follows as in the classical setting by distributivity: If $e\subseteq f$ in $E(S)$ and $c_1,c_2$ are relative complements of $e$ in $f$, then $c_1e=c_1\cap e=0$ (see Lemma \ref{lem:technicalpropertiesofsigmaorder}\ref{lem:technicalpropertiesofsigmaorder3}), so
\[c_1=c_1f=c_1(e\cup c_2)=c_1e\cup c_1c_2=0\cup c_1c_2=c_1c_2\]
i.e., $c_1\leq c_2$, and symmetrically we conclude that $c_1=c_2$.

Since compatible orders are preserved by taking inverses, then property \ref{def:sigmaorder.distributivity} implies $(c_1\cup c_2)a=(c_1a)\cup(c_2a)$ as well.

In the next few technical lemmas, we compile the technical properties of $\Sigma$-ordered inverse semigroups which will be necessary, and follow the discussion with two canonical examples. Let us fix, throughout this subsection, a $\Sigma$-ordered inverse semigroup $(S,\subseteq)$.

\begin{lemma}\label{lem:technicalpropertiesofsigmaorder}
Suppose that $(S,\subseteq)$ is a $\Sigma$-ordered inverse semigroup.
    \begin{enumerate}[label=(\alph*)]
        \item\label{lem:technicalpropertiesofsigmaorder1} If $x\leq y\leq z$ and $x\subseteq z$, then $x\subseteq y$.
        \item\label{lem:technicalpropertiesofsigmaorder2} If $a\cap b$ exists and $c\subseteq b$, then $a\cap c$ exists, and $a\cap c=(a\cap b)c^*c$.
        \item\label{lem:technicalpropertiesofsigmaorder3} If $x,y\subseteq a$, then $x\cap y$ exists, and $x\cap y=xy^*y=yy^*x$.
        \item\label{lem:technicalpropertiesofsigmaorder4} If $x,y\subseteq a$ and $b\in S$, then $b(x\cap y)=bx\cap by$.
        \item\label{lem:technicalpropertiesofsigmaorder5} If $a\cup b$ exists, then $(a\cup b)^*(a\cup b)=a^*a\cup b^*b$
        \item\label{lem:technicalpropertiesofsigmaorder6} If $b\cup c$ exists, then
        \[a\cap(b\cup c)=(a\cap b)\cup(a\cap c)\]
        in the sense that one side is defined if and only if the other one is, in which case they coincide.
        \item\label{lem:technicalpropertiesofsigmaorder7} $S$ admits relative complements. More precisely, if $a\subseteq b$, then $b\setminus a=b(b^*b\setminus a^*a)$. (Note that the relative complement $b^*b\setminus a^*a$ is defined by \ref{def:sigmaorder.ESrelativecomplements}.)
        \item\label{lem:technicalpropertiesofsigmaorder8} If $a\cup b$, $a\cup c$ and $b\cap c$ are defined, then
        \[a\cup(b\cap c)=(a\cup b)\cap (a\cup c)\]
        in the sense that both sides are defined and coincide.
        \item\label{lem:technicalpropertiesofsigmaorder9} If $a\subseteq b$ and $c\in S$, then $c(b\setminus a)=(cb)\setminus (ca)$.
        \item\label{lem:technicalpropertiesofsigmaorder10} (De Morgan's Laws) If $a_1,a_2\subseteq b$ then $(b\setminus a_1)\cap(b\setminus a_2)=b\setminus(a_1\cup a_2)$ and $(b\setminus a_1)\cup(b\setminus a_2)=(b\setminus a_1\cap a_2)$.
    \end{enumerate}
\end{lemma}
\begin{proof}
    \begin{enumerate}[label={\ref{lem:technicalpropertiesofsigmaorder\arabic*}}]
        \item As $x\subseteq z$, we multiply both sides by $y^*y$ to obtain $x=xy^*y\subseteq zy^*y=y$, because $x\leq y\leq z$.
        \item On one hand, $(a\cap b)c^*c\subseteq (a\cap b)b^*b=a\cap b\subseteq a$, and $(a\cap b)c^*c\subseteq bc^*c=c$, so $(a\cap b)c^*c$ is a $\subseteq$-lower bound of $\left\{a,c\right\}$. Simple computations prove that it is the greatest lower bound.
        \item This was already proven in Lemma \ref{lem:relation.of.germs.is.a.graphed.congruence}.
        \item As $bx,by\subseteq ba$, item \ref{lem:technicalpropertiesofsigmaorder2} yields $bx\cap by=bx(by)^*(by)$. As $x,y\leq a$, then $xy^*\in E(S)$, so we may commute $xy^*$ and $b^*b$ to obtain
        \[bx\cap by=b(xy^*)(b^*b)y=b(b^*b)(xy^*)y=b(x\cap y)\]
        where we used item \ref{lem:technicalpropertiesofsigmaorder2} in the last equality, as $x\cap y=xy^*y$.
        
        \item By distributivity, Property \ref{def:sigmaorder.distributivity},
        \[(a\cup b)^*(a\cup b)=(a^*\cup b^*)(a\cup b)=(a^*a)\cup(a^*b)\cup(b^*a)\cup(b^*b).\]
        But since $a^*b\subseteq a^*(a\cup b)=a^*a$, and similarly for $b^*a$, then
        \[(a\cup b)^*(a\cup b)=a^*a\cup b^*b.\]
        
        \item If $a\cap(b\cup c)$ exists, then item \ref{lem:technicalpropertiesofsigmaorder5} implies that both $a\cap b$ and $a\cap c$ exist, Since both are $\subseteq a$, their $\subseteq$-join exists by Property \ref{def:sigmaorder.conditionaljoins}.
        
        Assume now that $(a\cap b)\cup(a\cap c)$ exists, and let us prove that it is the $\subseteq$-meet of $\left\{a,b\cup c\right\}$. It is a lower bound, since
        \[(a\cap b)\cup(a\cap c)\subseteq a\cup a=a\quad\text{and}\quad (a\cap b)\cup(a\cap c)\subseteq b\cup c.\]
        To prove that it is the largest lower bound, suppose $p\subseteq a$ and $p\subseteq (b\cap c)$. In particular, $p\cap b=b$. As $a\cap b$ exists, item \ref{lem:technicalpropertiesofsigmaorder5} implies $(a\cap b)p^*p=p\cap b=p$. Similarly, $(a\cap c)p^*p=p$. By distributivity, Property \ref{def:sigmaorder.distributivity}, we have
        \[((a\cap b)\cup(a\cap c))p^*p=((a\cap b)p^*p)\cap ((a\cap c)p^*p)=p\cap p=p,\]
        that is, $p\leq(a\cap b)\cup(a\cap c)\leq b\cup c$. Since $p\subseteq b\cup c$, item \ref{lem:technicalpropertiesofsigmaorder1} gives us $p\subseteq (a\cap b)\cup(a\cap c)$, as we wanted.
        
        \item Suppose $a\subseteq b$. Let $e=b^*b\setminus a^*a$, i.e., $a^*a\cup e=b^*b$ and $a^*a\cap e=0$. Multiplying both equalities by $b$ on the left on both sides and using distributivity (item \ref{lem:technicalpropertiesofsigmaorder3} and Property \ref{def:sigmaorder.distributivity}) yields $a\cup(be)=b$ and $a\cap(be)=0$, so $be$ is the complement of $a$ relative to $b$.
        
        \item Since both $a$ and $b\cap c$ are $\subseteq$-bounded above by $a\cup b$, then $a\cup(b\cap c)$ is defined by Property \ref{def:sigmaorder.conditionaljoins}, and is a $\subseteq$-lower bound of $a\cup b$. Symmetrically, changing the roles of $b$ and $c$, we see that it is a $\subseteq$-lower bound of $\left\{a\cup b,a\cup c\right\}$. Again, let us prove that it is the largest one. Suppose $p\subseteq a\cup b,a\cup c$. Let us verify some cases:
        \begin{enumerate}[label=(\arabic*)]
            \item If $p\subseteq a$, then of course $p\subseteq a\cup(b\cap c)$.
            \item Suppose $p\cap a=0$. Since both $p$ and $a$ are $\subseteq a\cup b$, then this means that $0=ap^*p$, so $p=(a\cup b)p^*p=ap^*p\cup bp^*p=bp^*p$, and thus $p\leq bp^*p$. Both $p$ and $b$ are $\subseteq a\cup b$, so $p\subseteq b$. Similarly, $p\subseteq c$, so $p\subseteq b\cap c\subseteq a\cup(b\cap c)$.
            \item For a general lower bound $p$ of $\left\{a\cup b,a\cup c\right\}$, we apply the first case to $p\cap a$ (which exists by item \ref{lem:technicalpropertiesofsigmaorder3}, since $a,p\subseteq a\cup b$) and the second one to $p\setminus(p\cap a)$, to conclude that
            \[p=(p\cap a)\cup(p\setminus(p\cap a))\subseteq a\cup(b\cap c).\qedhere\]
        \end{enumerate}
        The last case proves that $a\cup(b\cap c)$ is the meet of $(a\cup b)\cap(a\cup c)$, as we wanted.
        \item If $a\subseteq b$ and $c\in S$, then it is easy enough, with the previous items, to verify that $c(b\setminus a)$ satisfies the properties of the complement of $ca$ in $cb$, so $c(b\setminus a)=(cb)\setminus(ca)$.
        \item Similarly to the previous item, it is only a matter of verifying that $(b\setminus a_1)\cap(b\setminus a_2)$ satisfies the defining properties of the complement $b\setminus (a_1\cup a_2)$, so these elements coincide, and similarly $b\setminus (a_1\cap a_2)=(b\setminus a_1)\cup(b\setminus a_2)$.\qedhere
    \end{enumerate}
\end{proof}

The interpolation property \ref{def:sigmaorder.interpolation} is very important and will be used heavily during the proof of the duality result we aim to produce.

\begin{lemma}
    Suppose $t\leq a$ in $S$. Then there exists a unique $p\in S$ such that $t\leq p\subseteq a$ and for all $x\in S$
    \[tx=0\iff px=0\qquad\text{and}\qquad xt=0\iff xp=0.\]
\end{lemma}
\begin{proof}
    We have $t\leq a$ and $t^*\leq a^*$. By \ref{def:sigmaorder.interpolation}, consider $z,w\in S$ such that $t\leq z,w\subseteq a$ and
    \[tx=0\iff zx=0\qquad\text{and}\qquad xt=0\iff xw=0.\]
    Let $p=z\cap w$ (Lemma \ref{lem:technicalpropertiesofsigmaorder}\ref{lem:technicalpropertiesofsigmaorder3}). Then $t\leq p\subseteq a$, $p=zw^*w$ and for all $x\in S$,
    \[px=0\iff zw^*wx=0\iff tw^*wx=0\iff tx=0,\]
    where the second equivalence follows from the choice of $z$ and the third one from $t\leq w$. Similarly, as $p=w\cap z=zz^*w$, $xp=0\iff xt=0$.
    
    As for the uniqueness of $p$, suppose that $q$ is another element with the same property. Then $p^*p,q^*q\subseteq a^*a$, and we have $0=p(a^*a\setminus p^*p)=0$, hence $x(a^*a\setminus p^*p)=0$ and $q(a^*a\setminus p^*p)=0$ by the given properties of $p$ and $q$. Therefore
    \[q^*q=q^*q(a^*a)=q^*q((a^*a\setminus p^*p)\cup p^*p)=q^*qp^*p,\]
    so $q^*q\leq p^*p$. By symmetry, we conclude that $q^*q=p^*p$, and as both $p,q\leq a$ then $p=q$.
    \qedhere
\end{proof}

\begin{definition}
    The unique element $p$ of the Lemma above is called the \emph{interpolator} of $t$ and $a$, and is denoted by $a|t$.
\end{definition}

Interpolators may be described precisely in well-known cases.

\begin{example}
    Suppose $\mathcal{S}$ is an ample inverse semigroupoid. Then $(\mathbf{KB}(\mathcal{S}),\subseteq)$ is a $\Sigma$-ordered inverse semigroupoid. If $A\leq B$ in $\mathbf{KB}(\mathcal{S})$, then the interpolator $B|A$ is $\so^{-1}(\so(B))\cap A$.
\end{example}

\begin{example}\label{ex:weaklybooleanaresigma}
    Let $S$ be a weakly Boolean distributive inverse semigroup as in \cite{MR3077869}. Then the canonical order of $S$ makes it a $\Sigma$-ordered inverse semigroup. Interpolators are trivial, in the sense that if $a\leq b$ in $S$ then $b|a=a$.
\end{example}

\subsection{The general procedure}\label{subsec:generalprocedure}

Let us describe the general procedure to reconstruct an ample inverse semigroupoid $\mathcal{S}$ from the pair $(\mathbf{KB}(\mathcal{S}),\subseteq)$. First, we ``extend'' the canonical action $\mathcal{S}$ on $\mathcal{S}^{(0)}$ to an action of $\mathbf{KB}(\mathcal{S})$ on $\mathcal{S}^{(0)}$ in the intuitive way: a bisection of $\mathcal{S}$ is actually a collection of arrows between points of $\mathcal{S}^{(0)}$, and thus describes a partial function of $\mathcal{S}^{(0)}$. Then set inclusion on $\mathbf{KB}(\mathcal{S})$ induces a compatible order on the semidirect product $\mathbf{KB}(S)\ltimes\mathcal{S}^{(0)}$, and the quotient (semigroupoid of germs) is isomorphic to $\mathcal{S}$.

So the problem at hand now is to describe $\mathcal{S}^{(0)}$ in terms of $\mathbf{KB}(\mathcal{S})$, which we follow by constructing the relevant categories and the dual equivalence between them.

\subsection{Reconstruction of \texorpdfstring{$\mathcal{S}^{(0)}$}{S0} from \texorpdfstring{$\mathbf{KB}(\mathcal{S})$}{KB(S)}}\label{subsec:reconstructionofvertexspace}

Let $E$ be a semilattice. Recall that a \emph{filter} on $E$ is a nonempty subset $\mathfrak{F}\subseteq E$ which is \emph{downwards directed} and \emph{upwards closed} - i.e., it satisfies
\begin{enumerate}[label=(\roman*)]
    \item If $e,f\in\mathfrak{F}$, then $ef\in\mathfrak{F}$;
    \item If $e\in\mathfrak{F}$, $f\in E$, and $e\leq f$, then $f\in\mathfrak{F}$.
\end{enumerate}

A filter is \emph{proper} if $\mathfrak{F}\neq E$. If $E$ has a zero (minimum), this is equivalent to say that $0\not\in\mathfrak{F}$. An \emph{ultrafilter} is a maximal proper filter. Every nonzero element $e\in E$ belongs to the filter $e^\uparrow=\left\{f\in E:e\leq f\right\}$, and Zorn's Lemma implies that every filter is contained in some ultrafilter. Therefore every element of $E$ belongs to some ultrafilter.

The following alternative description of ultrafilters is well-known and useful.

\begin{lemma}\label{lem:description.of.ultrafilters.by.nonzero.products}
    Let $\mathfrak{F}$ be a proper filter in a semilattice with zero $E$. The following are equivalent:
    \begin{enumerate}[label=(\arabic*)]
        \item\label{lem:description.of.ultrafilters.by.nonzero.products.1} $\mathfrak{F}$ is an ultrafilter;
        \item\label{lem:description.of.ultrafilters.by.nonzero.products.2} For every $e\in E$, if $0\not\in e\mathfrak{F}$ then $e\in\mathfrak{F}$.
    \end{enumerate}
\end{lemma}
\begin{proof}
    If $0\not\in e\mathfrak{F}$, then the set $(e\mathfrak{F})^{\uparrow}=\left\{u\in E:ef\leq u\text{ for some }f\in\mathfrak{F}\right\}$ is a proper filter containing $\mathfrak{F}\cup\left\{e\right\}$. If $\mathfrak{F}$ is maximal, then $\mathfrak{F}=(e\mathfrak{F})^{\uparrow}$, which contains $e$. This proves \ref{lem:description.of.ultrafilters.by.nonzero.products.1}$\Rightarrow$\ref{lem:description.of.ultrafilters.by.nonzero.products.2}. 
    
    Conversely, if \ref{lem:description.of.ultrafilters.by.nonzero.products.2} is valid, consider a proper filter $\mathcal{G}$ containing $\mathfrak{F}$. As $0\not\in\mathfrak{G}$, then every $e\in G$ satisfies the condition of \ref{lem:description.of.ultrafilters.by.nonzero.products.2}, so $\mathfrak{G}\subseteq\mathfrak{F}$. Hence $\mathfrak{F}$ is valid, i.e., \ref{lem:description.of.ultrafilters.by.nonzero.products.1}.\qedhere
\end{proof}

\begin{definition}
    The \emph{spectrum} of $E$ is the topological space $\Omega(E)$ of all ultrafilters in $E$, with the topology generated by the sets
    \[X[e]=\left\{\mathfrak{F}\in\Omega(E):e\in\mathfrak{F}\right\},\qquad (e\in E).\]
\end{definition}

Note that $X[e]\cap X[f]=X[ef]$ for all $e,f\in E$, so $\left\{X[e]:e\in E\right\}$ is indeed a basis for the topology of $\Omega(E)$.

Suppose now that $(S,\subseteq)$ is a $\Sigma$-ordered inverse semigroup, and $E=E(S)$ is its idempotent semilattice (we use the \emph{canonical order} $\leq$ as the lattice structure of $E(S)$, and not $\subseteq$).

\begin{lemma}\label{lem:ultrafilters.are.prime}
    Every element $\mathfrak{F}$ of $\Omega(E(S))$ is a prime $\subseteq$-filter, in the sense that if $e_1\cup e_2\in\mathfrak{F}$ then $e_1\in\mathfrak{F}$ or $e_2\in\mathfrak{F}$.
    
    In other words, $X[e_1\cup e_2]=X[e_1]\cup X[e_2]$ whenever $e_1\cup e_2$ is defined in $E(S)$.
\end{lemma}
\begin{proof}
    If $e_1\in\mathfrak{F}$ we are done. If $e_1\not\in\mathfrak{F}$, then by Lemma \ref{lem:description.of.ultrafilters.by.nonzero.products} there exists $f\in\mathfrak{F}$ such that $e_1f=0$. For any other $f'\in\mathfrak{F}$, as $e_1\cup e_2\in\mathfrak{F}$, we have
    \[0\neq (f'f)(e_1\cup e_2)=(f'fe_1)\cup (f'fe_2)=0\cup (f'fe_2)=f'fe_2,\]
    and in particular $f'e_2\neq 0$ (as $E(S)$ is commutative). By Lemma \ref{lem:description.of.ultrafilters.by.nonzero.products}, $e_2\in\mathfrak{F}$.\qedhere
\end{proof}

\begin{proposition}
$\Omega(E(S))$ is a zero-dimensional, locally compact Hausdorff space.
\end{proposition}
\begin{proof}
    First we prove that $\Omega(E(S))$ is Hausdorff. Suppose that $F\neq G$ in $\Omega(E(S))$. As $F$ and $G$ are ultrafilters, Lemma \ref{lem:description.of.ultrafilters.by.nonzero.products} allows us to take $f\in F$ and $g\in G$ such that $fg=0$, so $X[f]$ and $X[g]$ are disjoint neighbourhoods of $\mathfrak{F}$ and $\mathfrak{G}$, respectively.
    
    To prove that $\Omega(E(S))$ is locally compact and zero-dimensional, it suffices to prove that each basic open set $X[e]$ is compact. Suppose that $X[e]=\bigcup_{i\in I}X[r_i]$. Since $X[er_i]=X[e]\cap X[r_i]$, we may assume that $r_i\leq e$. Using interpolation, let $s_i=e|r_i$, i.e.,
    \[r_i\leq s_i\subseteq e\qquad\text{and}\qquad s_ix=0\iff r_ix=0\text{ for all }x\in E(S).\]
    By Lemma \ref{lem:description.of.ultrafilters.by.nonzero.products}, we have $X[r_i]=X[s_i]$ for all $i$, so it suffices to prove that $e=\bigcup_{i\in F}s_i$ for some finite subset $F\subseteq I$ and apply Lemma \ref{lem:ultrafilters.are.prime}.
    
    Suppose otherwise, and let us arrive at a contradiction. For every finite subset $F\subseteq I$, let $s(F)=\bigcup_{i\in F}s_i$. Then $e\setminus s(F)\neq 0$ for all $F$. Using Lemma \ref{lem:technicalpropertiesofsigmaorder},
    \[(e\setminus s(F))(e\setminus s(G))=(e\setminus s(F))\cap(e\setminus s(G))=e\setminus(s(F)\cup s(G))=e\setminus s(F\cup G), \]
    so the family $B=\left\{x\in E(S):x\geq s(F)\text{ for some finite }F\subseteq I\right\}$ is a $\leq$-filter containing $e$. Let $\mathfrak{F}$ be any $\leq$-ultrafilter containing $B$ (which exists by Zorn's Lemma). Then for all $i$, $e\setminus s_i\in B\subseteq \mathfrak{F}$, so $s_i\not\in F$, i.e., $\mathfrak{F}\in X(e)\setminus\bigcup_{i\in I}X[s_i]$, a contradiction.\qedhere
\end{proof}

Suppose now that $\mathcal{S}$ is an ample inverse semigroupoid. For every $x\in\mathcal{S}^{(0)}$, the set \[\psi(x)\defeq\left\{U\in E(\mathbf{KB}(\mathcal{S})):x\in\so(U)\right\}\]
is clearly a proper filter in $E(\mathbf{KB}(\mathcal{S}))$.

\begin{lemma}\label{lem:homeomorphisms0spectrum}
    For every $x\in\mathcal{S}^{(0)}$, the set $\psi(x)$ is an ultrafilter in $E(\mathbf{KB}(\mathcal{S}))$. In fact, the map $\psi\colon \mathcal{S}^{(0)}\to\Omega(E(\mathbf{KB}(\mathcal{S})))$ is a homeomorphism. The inverse $\psi^{-1}\colon\Omega(E(\mathbf{KB}(\mathcal{S})))$ is the unique function satisfying $\left\{\psi^{-1}(\mathfrak{F}\right\}=\bigcap_{U\in\mathfrak{F}}\so(U)$ for all $\mathfrak{F}\in\Omega(E(\mathbf{KB}(\mathcal{S})))$.
\end{lemma}
\begin{proof}
    First note that, as $\mathcal{S}^{(0)}$ is Hausdorff, then $\psi(x)\subseteq\psi(y)$ if and only if $x=y$. 
    
    We thus only need to prove that every proper filter $\mathfrak{F}$ of $E(\mathbf{KB}(\mathcal{S}))$ is contained in $\psi(x)$ for some $x$. As $\mathfrak{F}$ is downwards directed, then the family $\left\{\so(U):U\in\mathfrak{F}\right\}$ of compact subsets of the Hausdorff space $\mathcal{S}^{(0)}$ has the finite intersection property, so Cantor's Intersection Theorem implies that its intersection is nonempty. Any element $x\in\bigcap_{U\in\mathfrak{F}}\so(U)$ satisfies $\mathfrak{F}\subseteq\psi(x)$.
    
    This in fact implies that $\psi(x)$ is an ultrafilter for any $x\in\mathcal{S}^{(0)}$: By Zorn's lemma and the previous paragraph, there are an ultrafilter $\mathfrak{F}$ and $y\in\mathcal{S}^{(0)}$ such that $\psi(x)\subseteq\mathfrak{F}\subseteq\psi(y)$. By the first paragraph of the proof and maximality of $\mathcal{F}$, $\psi(x)=\psi(y)=\mathfrak{F}$ is an ultrafilter. The argument in the previous paragraph implies the statement about $\psi^{-1}$.\qedhere
\end{proof}

\subsection{The category \texorpdfstring{$\Sigma\cat{-Ord}$}{Σ-Ord}} A \emph{morphism} of $\Sigma$-ordered inverse semigroups is a semigroup homomorphism $\theta\colon S\to T$ which further satisifies:
\begin{enumerate}[label=(\roman*)]
    \item\label{def:morphismofsigmaordered.zero} $\theta(0)=0$;
    \item\label{def:morphismofsigmaordered.monotone} $\theta$ is $\subseteq$-monotone: if $a\subseteq b$ in $S$ then $\phi(a)\subseteq\phi(b)$ in $T$;
    \item\label{def:morphismofsigmaordered.cup} $\theta$ is a $\cup$-morphism: if $a\cup b$ exists in $S$, then $\theta(a)\cup\theta(b)$ exists in $T$, and $\theta(a\cup b)=\theta(a)\cup\theta(b)$;
    \item\label{def:morphismofsigmaordered.weaklymeet} For all $a,b\in S$ and $t\in T$, if $t\subseteq \theta(a),\theta(b)$, then there exists $c\subseteq a,b$ such that $t\subseteq\theta(c)$;
    \item\label{def:morphismofsigmaordered.proper} $\theta$ is \emph{proper}: For all $t\in T$, there exist $t_1,\ldots,t_n\in T$ and $s_1,\ldots,s_n\in S$ such that
    \[t=\bigcup_{i=1}^n t_i\qquad\text{and}\qquad t_i\subseteq\theta(s_i)\text{ for all }i=1,\ldots,n.\]
    \item\label{def:morphismofsigmaordered.interpolators} $\theta$ preserves interpolators: If $a\leq b$, then $\theta(b|a)=\theta(b)|\theta(a)$.
\end{enumerate}

It is easy to check that morphisms of $\Sigma$-ordered inverse semigroups are stable under composition. We thus define the category $\cat{$\Sigma$-Ord}$ of $\Sigma$-ordered inverse semigroups and their morphisms.

Property \ref{def:morphismofsigmaordered.weaklymeet} above was considered in \cite[p.\ 134]{MR3077869}, as part of the definition of \emph{callitic morphisms} for distributive inverse semigroups. It is a strengthening of the condition that $\theta$ preserves $\land$, and is necessary since we consider non-Hausdorff inverse semigroupoids.

To see that $\theta$ preserves $\land$, suppose that $a\cap b$ exists in $S$. Then $\theta(a\cap b)\subseteq\theta(a),\theta(b)$, as $\theta$ is $\subseteq$-monotone. If $u\subseteq \theta(a),\theta(b)$, then by \ref{def:morphismofsigmaordered.weaklymeet} there exists $c\subseteq a\cap b$ such that $\theta(u)\subseteq\theta(c)\subseteq\theta(a\cap b)$, so $\theta(a\cap b)$ is the largest $\subseteq$-lower bound of $\left\{a,b\right\}$, i.e., $\theta(a\cap b)=\theta(a)\cap\theta(b)$.

It is important to note that there are semigroup homomorphisms, between $\Sigma$-ordered inverse semigroups, which satisfy Properties \ref{def:morphismofsigmaordered.zero}-\ref{def:morphismofsigmaordered.proper} but not \ref{def:morphismofsigmaordered.interpolators}.

The next example shows that \ref{def:morphismofsigmaordered.weaklymeet} and \ref{def:morphismofsigmaordered.interpolators} are necessary.

\begin{example}
Let $L_3=\left\{0,1,2\right\}$ be the lattice with $0<1<2$, and the compatible order $x\subseteq y\iff x=0$ or $x=y$, and let $L_2=\left\{0,1\right\}$, as an ideal of $L_3$, and the restriction of $\subseteq$. Both $L_3$ and $L_2$ are $\Sigma$-ordered inverse semigroups. ($L_3$ is isomorphic to $\mathbf{KB}(L_2)$, where $L_2$ is seen simply as an inverse semigroup, and $L_2$ is isomorphic to $\mathbf{KB}(\left\{0\right\})$.)

The map $\theta\colon L_3\to L_2$, $\theta(0)=0$, $\theta(1)=\theta(2)=1$, satisfies all of \ref{def:morphismofsigmaordered.zero}-\ref{def:morphismofsigmaordered.interpolators} except \ref{def:morphismofsigmaordered.weaklymeet}.

Similarly, the map $\eta\colon L_3\to L_2$, $\eta(0)=\eta(1)=0$, $\eta(2)=1$, satisfies all of \ref{def:morphismofsigmaordered.zero}-\ref{def:morphismofsigmaordered.proper} but not \ref{def:morphismofsigmaordered.interpolators}
\end{example}

\subsection{The category
\texorpdfstring{$\cat{Amp}_\star$}{Amp*}}
A homomorphism $\phi\colon\mathcal{S}\to\mathcal{T}$ of inverse semigroupoids is \emph{star-injective} if for all $a,b\in\mathcal{S}$, $\so(a)=\so(b)$ and $\phi(a)=\phi(b)$ implies $a=b$ (i.e., $\theta$ is injective on all fibers), and \emph{star-surjective} if for all $t\in\mathcal{T}$ and all $a\in\mathcal{S}$, if $\so(t)=\so(\phi(a))$, then there exists $b\in\mathcal{S}$ with $\so(b)=\so(a)$ and $\phi(b)=t$.

If $\phi$ is both star-injective and star-surjective, we say that $\phi$ is a covering homomorphism. We define $\cat{Amp}_\star$ as the category of ample inverse semigroupoids and continuous, proper covering homomorphisms.

\subsection{The functor \texorpdfstring{$\mathbb{K}\colon\cat{Amp}_\star^{\mathrm{op}}\to\cat{$\Sigma$-Ord}$}{K:Amp*op→Σ-Ord}}

On objects, to each ample inverse semigroup $\mathcal{S}$, we set $\mathbb{K}(\mathcal{S})=(\mathbf{KB}(\mathcal{S}),\subseteq)$. To define $\mathbb{K}$ on morphisms, let us first state the necessary lemma.

\begin{lemma}\label{lem:functorkonmorphisms}
    Let $\phi\colon\mathcal{S}\to\mathcal{T}$ is a proper continuous covering homomorphism of ample semigroupoids. If $A\in\mathbf{KB}(\mathcal{T})$ then $\phi^{-1}(A)\in\mathbf{KB}(\mathcal{S})$. The map $A\mapsto\phi^{-1}(A)$ is a morphism of $\Sigma$-ordered inverse semigroups.
\end{lemma}

We thus define the functor $\mathbb{K}$ on a morphism $\phi\colon A\to B$ of $\cat{Amp}_\star$ as $\mathbb{K}(\phi)(A)=\phi^{-1}(A)$.

\begin{proof}[Proof of Lemma \ref{lem:functorkonmorphisms}]
Suppose $A\in\mathbf{KB}(\mathcal{T})$. As $\phi$ is continuous and proper, then $\phi^{-1}(A)$ is open and compact. Let us prove that it is a bisection of $\mathcal{S}$.

Suppose $a,b\in\phi^{-1}(A)$ and $\so(a)=\so(b)$. Then $\so(\phi(a))=\so(\phi(b))$, andas $\phi(a)$ and $\phi(b)$ belong to the bisection $A$ then $\phi(a)=\phi(b)$. As $\phi$ is star-injective then $a=b$. Thus the source map is injective on $\phi^{-1}(A)$, and similarly the range map is injective on $\phi^{-1}(A)$.

We now need to prove that $\mathbb{K}(\phi)\colon A\mapsto\phi^{-1}(A)$ is a semigroup homomorphism, i.e., that $\phi^{-1}(AB)=\phi^{-1}(A)\phi^{-1}(B)$ for all $A,B\in\mathbf{KB}(\mathcal{T})$. The inclusion $\phi^{-1}(A)\phi^{-1}(B)\subseteq\phi^{-1}(AB)$ is immediate as $\phi$ is a semigroupoid homomorphism. Conversely, suppose that $z\in\phi^{-1}(AB)$. Then there exist $(a,b)\in (A\times B)\cap\mathcal{T}^{(2)}$ such that $\phi(z)=ab$, and in particular $\so(\phi(z))=\so(b)$. As $\phi$ is star-surjective, there exists $p_b\in\mathcal{S}$ with $\so(p_b)=\so(z)$ and $\phi(p_b)=b$. Then $\so(\phi(pp_b^*))=\so(b^*)=\so(a)$, so again since $\phi$ is star-surjective then there exists $p_a\in\mathcal{S}$ with $\so(p_a)=\so(p_b^*)$ and $\phi(p_a)=a$.

Then $\so(p_ap_b)=\so(p_b)=\so(z)$, and $\phi(p_ap_b)=ab=\phi(x)$. As $\phi$ is star-injective then $x=p_ap_b\in\phi^{-1}(A)\cap\phi^{-1}(B)$.

The first non-trivial property that we need to prove for $\mathbb{K}(\phi)$, to conclude that it is a morphism of $\Sigma$-ordered inverse semigroups, is properness. Suppose $K\in\mathbf{KB}(\mathcal{S})$. Then $\phi(K)$ is compact in $\mathcal{T}$, so we may cover it by finitely many bisections $A_1,\ldots,A_n\in\mathbf{KB}(\mathcal{T})$. Setting $K_i=\phi^{-1}(A_i)\cap K$, we have $K=\bigcup_{i=1}^n K_i$ and $K_i\subseteq\mathbb{K}(\phi)(A_i)$. This proves that $\mathbb{K}(\phi)$ is proper.

The fact that $\mathbb{K}(\phi)$ preserves interpolators it equivalent to the equality \[\phi^{-1}(\so^{-1}(\so(B))\cap A)=\so^{-1}(\so(\phi^{-1}(B)))\cap\phi^{-1}(A)\]
for arbitrary $A,B\in\mathbf{KB}(\mathcal{T})$, and this is easily proven using star-surjectivity.\qedhere
\end{proof}

\subsection{The functor \texorpdfstring{$\mathbb{P}\colon\cat{$\Sigma$-Ord}^{\mathrm{op}}\to\cat{Amp}$}{P:Σ-Ordop→Amp*}}

As a motivation, suppose that the functor $\mathbb{P}$ is defined in a manner that $\mathbb{P}(\mathbb{KB}(\mathcal{S}))\cong\mathcal{S}$ for every ample semigroupoid $\mathcal{S}$. Given another ample semigroupoid $\mathcal{T}$, a morphism $\theta\colon\mathbf{KB}(\mathcal{S})\to\mathbf{KB}(\mathcal{T}))$ thus induces a morphism $\mathbb{P}(\theta)\colon\mathcal{T}\to\mathcal{S}$, and in particular a map between the underlying vertex spaces $\mathcal{T}^{(0)}$ and $\mathcal{S}^{(0)}$, or alternatively by the content of Subsection \ref{subsec:generalprocedure}, between $\Omega(E(\mathbf{KB}(\mathcal{T})))$ and $\Omega(E(\mathbf{KB}(\mathcal{S})))$. This is the map we deal with first.

Let us temporarily fix a morphism of $\Sigma$-inverse semigroups $\theta\colon(S,\subseteq)\to(T,\subseteq)$. Then $\theta$ restricts to a semigroup homomorphism $\theta|_{E(S)}\colon E(S)\to E(T)$. The construction of the maps $\Omega(E(T))\to\Omega(E(S))$ cannot be done just by taking pre-images, contrary to as in more classical cases (e.g. \cite{MR3077869}).

\begin{lemma}\label{lem:maponspectra}
    If $\mathfrak{F}\in\Omega(E(T))$, then $\theta^{-1}(\mathfrak{F})$ is nonempty, and the set \[\mathbb{P}(\theta)^{(0)}(\mathfrak{F})=\left\{e\in E(S):0\not\in e\theta|_{E(S)}^{-1}(\mathfrak{F})\right\}\]
    is an ultrafilter containing $\theta|_{E(S)}^{-1}(\mathfrak{F})$. Moreover, the map $\mathbb{P}(\theta)^{(0)}$ thus defined is continuous and proper.
\end{lemma}
\begin{proof}
    Choose any $t\in\mathfrak{F}$, so as $\theta$ is proper, there exist $t_1,\ldots,t_n\in E(T)$ and $s_1,\ldots,s_n\in E(S)$ such that $t=\bigcup_{i=1}^n t_i$ and $t_i\subseteq\phi(s_i)$ for each $i$. As $\mathfrak{F}$ is $\subseteq$-prime (Lemma \ref{lem:ultrafilters.are.prime}), for some $i$ we have $t_i\in\mathfrak{F}$, so $s_i\in\theta|_{E(S)}^{-1}(\mathfrak{F})$.
    
    Of course, $\mathbb{P}(\theta)^{(0)}(\mathfrak{F})$ is upwards closed, does not contain $0$, and contains $\theta^{-1}(\mathfrak{F})$, so to prove that it is a filter we need to prove that it is closed under products. For this, fix an arbitrary $p\in\theta|_{E(S)}^{-1}(\mathfrak{F})$. Suppose that $f,g\in\mathbb{P}(\theta)^{(0)}(\mathfrak{F})$. Then $fp\leq p$, so we may consider the interpolator $q_f\defeq p|(fp)$.
    
    Since $\theta(p)=\theta(q_f)\cup\theta(p\setminus q_f)$, one of $q_f$ or $(p\setminus q_f)$ belongs to $\theta^{-1}(\mathfrak{F})$ (by Lemma \ref{lem:ultrafilters.are.prime}). We have $fp(p\setminus q_f)\leq q_f(p\setminus q_f)=q_f\cap(p\setminus q_f)=0$, and because $f\in\mathbb{P}(\theta)^{(0)}(\mathfrak{F})$ then $p(p\setminus q_f)\not\in\theta^{-1}(\mathfrak{F})$, so $(p\setminus q_f)\not\in\theta^{-1}(\mathfrak{F})$ as well.
    
    We therefore have $q_f\in\theta^{-1}(\mathfrak{F})$. Suppose now, in order to obtain a contradiction, that $fg\not\in\mathbb{P}(\theta)^{(0)}(\mathfrak{F})$, so that we may find $e\in\theta^{-1}(\mathfrak{F})$ such that $fge=0$. Then $fp(eg)=0$, implying $q_f(eg)=0$, contradicting the facts that $e,q_f\in\theta^{-1}(\mathfrak{F})$ and $g\in\mathbb{P}(\theta)^{(0)}(\mathfrak{F})$.
    
    From all of this, we conclude that $\mathbb{P}(\theta)^{(0)}(\mathfrak{F})$ is a proper filter of $E(S)$. In fact, its definition makes it clear, by Lemma \ref{lem:description.of.ultrafilters.by.nonzero.products}, that it is an ultrafilter in $E(S)$, i.e., an element of $\Omega(E(S))$.
    
    To prove that $\mathbb{P}(\theta)^{(0)}$ is continuous and proper, it is sufficient to prove that $\left(\mathbb{P}(\theta)^{(0)}\right)^{-1}(X[e])=X[\theta(e)]$, as it shows that preimages of basic compact open subsets of $\Omega(E(\mathbf{KB}(\mathcal{S})))$ are compact and open in $\Omega(E(\mathbf{KB}(\mathcal{T})))$. The inclusion $X[\theta(e)]\subseteq\left(\mathbb{P}(\theta)^{(0)}\right)^{-1}(X[e])$ is simple enough, however the converse is harder. Namely, we need to prove that if $e\in\mathbb{P}(\theta)^{(0)}(\mathfrak{F})$, then $\theta(e)\in\mathfrak{F}$.
    
    Suppose that this was not the case, so there exists $a\in\mathfrak{F}$ with $\theta(e)a=0$. By Property \ref{def:sigmaorder}\ref{def:sigmaorder.zerojoins}, the join $\theta(e)\cup a$ exists. As $\theta$ is proper, the same considerations as in the proof that $\theta^{-1}(\mathfrak{F})$ is nonempty allow us to assume that $\theta(e)\cup a\subseteq\theta(g)$, and in particular $\theta(e)=\theta(eg)$. Letting $p=g|(eg)$, we have $\theta(p)=\theta(g)|\theta(e)$. In particular, $\theta(e)a=0$ implies $\theta(p)a=0$, so $\theta(p)\not\in\mathfrak{F}$. On the other hand, $e(g\setminus p)=ge(g\setminus p)\leq p(g\setminus p)=0$, and since $e\in\mathbb{P}(\theta)^{(0)}(\mathfrak{F})$ then $g\setminus p\not\in\theta^{-1}(\mathfrak{F})$, i.e., $\theta(g\setminus p)\not\in\mathfrak{F}$.
    
    This is a contradiction, as neither $\theta(p)$ nor $\theta(g\setminus p)$ belong to the ultrafilter $\mathfrak{A}$ but $\theta(g)=\theta(p)\cup\theta(g\setminus p)$ does.
    
    We therefore conclude that $\left(\mathbb{P}(\theta)^{(0)}\right)^{-1}(X[e])=X[\theta(e)]$, as we wanted.\qedhere
\end{proof}

To finish the construction of the functor $\mathbb{P}$, let us work in a more general setting. Suppose that $\theta\colon S\curvearrowright E$ is an action of an inverse semigroup $S$ acts on a semilattice $E$ by order isomorphisms of ideals of $E$ (equivalently, this is action between semigroupoids). Then we may construct a ``dual'' action $\widehat{\theta}$ of $S$ on $\Omega(E)$. Namely, for $a\in S$, we set $\dom(\widehat{\theta}_a)\defeq\left\{\mathfrak{F}\in\Omega(E):\mathfrak{F}\cap\dom(\theta_a)\neq\varnothing\right\}$, and a bijection $\widehat{\theta}_a\colon\dom(\widehat{\theta}_a)\to\dom(\widehat{\theta}_{a^*})$, \[\widehat{\theta}_a(\mathfrak{F})=\theta_a(\mathfrak{F}\cap\dom(\theta_a))^{\uparrow}=\left\{u\in E:u\geq\theta_a(e)\text{ for some }e\in\mathfrak{F}\cap\dom(\theta_a)\right\}\]
Note that $\dom(\widehat{\theta}_a)=\bigcup_{e\in\dom(\theta_a)}X[e]$ is open in $\Omega(E)$, and that $\widehat{\theta}_a*(X[e])=X[\theta(e)]$, whenever $X[e]\subseteq$

In particular, the Munn representation $\mu\colon S\curvearrowright E(S)$ induces a continuous action $\widehat{\mu}\colon S\curvearrowright \Omega(E(S))$. Note that $\dom(\widehat{\mu}_a)=X[a^*a]$ for all $a\in S$. We then construct the semidirect product $S\ltimes_{\widehat{\mu}}\Omega(E(S))$.

As we are considering a $\Sigma$-ordered inverse semigroup $(S,\subseteq)$, the order $\subseteq$ induces a compatible order, also denoted $\subseteq$, on $S\ltimes_{\widehat{\mu}}\Omega(E(S))$, namely
\[(s,\mathfrak{F})\subseteq(t,\mathfrak{G})\iff e=f\text{ and }s\subseteq t.\]
Note that a basic open set of $S\ltimes_{\widehat{\mu}}\Omega(E(S))$ is of the form $\left\{s\right\}\times X[e]$ for some $e\in E(S)$, and
\[\left(\left\{s\right\}\times X[e]\right)^{\uparrow,\subseteq}=\bigcup\left\{\left\{t\right\}\times X[e]:t\in S,s\subseteq t\right\},\]
so upper closures (relative to $\subseteq$) of open sets are open, and similarly for lower closures. This means that $\subseteq$ is topologically compatible in $S\ltimes_{\widehat{\mu}}\Omega(E(S))$, hence the semigroupoid of germs
\[\mathbb{P}(S,\subseteq)\defeq \left(S\ltimes_{\widehat{\mu}}\Omega(E(S))\right)/\!\!\sim_{\subseteq}\]
is an étale inverse semigroupoid. The vertex space $\Omega(E(S))$ is Hausdorff, locally compact and zero-dimensional, and therefore $\mathbb{P}(S,\subseteq)$ is ample. This defined the functor $\mathbb{P}$ on objects.

To define $\mathbb{P}$ on morphisms, we need a lemma. Fix a morphism  of $\Sigma$-ordered inverse semigroups $\theta\colon(S,\subseteq)\to(T,\subseteq)$.

\begin{lemma}
    For every $[t,\mathfrak{F}]\in\mathbb{P}(T,\subseteq)$, there exists $s\in S$ such that $[\theta(s),\mathfrak{F}]=[t,\mathfrak{F}]$. If $[\theta(s_1),\mathfrak{F}]=[\theta(s_2),\mathfrak{F}]$, then $[s_1,\mathbb{P}^{(0)}(\theta)(\mathfrak{F}))]=[s_2,\mathbb{P}^{(0)}(\theta)(\mathfrak{F})$.
\end{lemma}
\begin{proof}
    Since $\theta$ is proper, choose $t_1,\ldots,t_n\in T$ and $s_1,\ldots,s_n\in S$ such that $t=\bigcup_{i=1}^n t_i$ and $t_i\subseteq \theta(s_i)$ for each $i$. As $t^*t=\bigcup_{i=1}^n t_i^*t_i$ belongs to $\mathfrak{F}$, we have $t_i^*t_i\in\mathfrak{F}$ for some $i$, and of course
    \[[t,\mathfrak{F}]=[t_i,\mathfrak{F}]=[\theta(s_i),\mathfrak{F}]\]
    
    For the second part, first note that if $[\theta(s),\mathfrak{F}]$ is a well-defined element of $\mathbb{P}(T,\subseteq)$, i.e., if $\mathfrak{F}\in X[\theta(s)^*\theta(s)]$ then the definition of $\mathbb{P}(\theta)^{(0)}$ readily yields $\mathbb{P}(\theta)^{(0)}(\mathfrak{F})\in[s^*s]$, so $[s,\mathbb{P}(\theta)^{(0)}(\mathfrak{F})]$ is a well-defined element of $\mathbb{P}(S,\subseteq)$.
    
    Now assume that $[\theta(s_1),\mathfrak{F}]=[\theta(s_2),\mathfrak{F}]$, i.e., there exists $t\subseteq\theta(s_1),\theta(s_2)$ such that $t^*t\in\mathfrak{F}$. Property \ref{def:morphismofsigmaordered.weaklymeet} of morphisms of $\Sigma$-ordered inverse semigroups yields $u\subseteq s_1,s_2$ such that $t\subseteq\theta(u)$. In particular $\mathfrak{F}\in X[\theta(u)^*\theta(u)]$, so the previous paragraph gives us $\mathbb{P}(\theta)^{(0)}(\mathfrak{F})\in X[u^*u]$. As $u\subseteq s_1,s_2$, we conclude that $[s_1,\mathbb{P}(\theta)^{(0)}(\mathfrak{F})]=[s_2,\mathbb{P}(\theta)^{(0)}(\mathfrak{F})]$.\qedhere
\end{proof}

We then define $\mathbb{P}(\theta)\colon\mathbb{P}(T,\subseteq)\to\mathbb{P}(S,\subseteq)$ as $\mathbb{P}(\theta)[t,\mathfrak{F}]=[s,\mathbb{P}(\theta)^{(0)}(\mathfrak{F})]$, where $s\in S$ is chosen so that $[t,\mathfrak{F}]=[\theta(s),\mathfrak{F}]$. The verification that $\mathbb{P}(\theta)$ is indeed a star-injective, star-surjective, homomorphism of inverse semigroupoids, is somewhat long and uneventful, so we omit it. The vertex map of $\mathbb{P}(\theta)$ is indeed $\mathbb{P}^{(0)}$ which we have already proven to be proper and continuous, so $\mathbb{P}(\theta)$ is also proper and continuous.

\subsection{The natural equivalence}

First consider an ample inverse semigroupoid $\mathcal{S}$. Let $\psi\colon\mathcal{S}^{(0)}\to\Omega(E(\mathbf{KB}(\mathcal{S})))$, $\psi(x)=\left\{U\in E(\mathbf{KB}(\mathcal{S})):x\in \so(U)\right\}$ be the homeomorphism constructed in Lemma \ref{lem:homeomorphisms0spectrum}.

Thus we define $\zeta_{\mathcal{S}}\colon\mathcal{S}\to\mathbb{P}(\mathbb{K}(\mathcal{S})$ as $\zeta_{\mathcal{S}}(a)=[A,\psi(\so(a))]$, where $A\in\mathbf{KB}(\mathcal{S})$ is chosen so that $a\in A$. Then $\zeta_{\mathcal{S}}$ is an isomorphism of étale inverse semigroupoids, and the associated vertex map $\zeta_{\mathcal{S}}^{(0)}=\psi$ is a homeomorphism, so $\zeta_{\mathcal{S}}$ is an isomorphism as well. This defines a natural isomorphism $\zeta\colon\id_{\cat{Amp}_\star}\to\mathbb{P}\mathbb{K}$. (The proof of the naturality of $\zeta$ is straightforward, although uninspiring, by carefully following all the definitions, and thus is ommited.)

In the other direction, given $(S,\subseteq)$ a $\Sigma$-ordered inverse semigroup, we wish to construct an isomorphism $\kappa_{(S,\subseteq)}\colon S\to\mathbb{K}(\mathbb{P}(S,\subseteq))$. For this, let us introduce some notation: Given $s\in S$ and a subset $U\subseteq\Omega(E(S))$, let $[s,U]=\left\{[s,x]:x\in U\cap\dom(\mu_s)\right\}$. Then the sets $[s,U]$, where $s\in S$ and $U\subseteq\Omega(E(S))$ is open, form a basis for the topology of $\mathbb{P}(S,\subseteq)$. Moreover, if $U$ is compact and open in $\Omega(E(S))$ then $[s,U]$ is a compact-open bisection of $\mathbb{P}(S,\subseteq)$.

We may therefore define the semigroup homomorphism
\[\kappa_{(S,\subseteq)}\colon S\to \mathbb{K}(\mathbb{P}(S,\subseteq)),\qquad\kappa(s)=[s,X[s^*s]].\]
Let us prove that it is an isomorphism.

\begin{lemma}\label{lem:subsetesofbasicarebasic}
Let $s\in S$ and $U$ be a compact-open subset of $X[s^*s]$. Then there exists $z\subseteq s$ such that $X[z^*z]=U$. In particular $[s,U]=[z,U]$.
\end{lemma}
\begin{proof}
    First we prove that $U=X[e]$ for some $e\in E(S)$. As $U$ is compact-open, we may write it as a finite union of basic open subsets of $\Omega(E(S))$, $U=\bigcup_{i=1}^n X[e_i]$, where $e_i\in E(S)$. Proceeding inductively, it is enough to assume $n=2$.
    
    For $i=1,2$, let $f_i=e_i|(e_1e_2)$, so $X[f_i]=X[e_1e_2]=X[e_1]\cap X[e_2]$. We have $X[e_i]=X[f_i]\cup X[e_i\setminus f_i]$, and the latter sets are disjoint, so $X[e_i\setminus f_i]=X[e_i]\setminus X[f_i]$. We thus have a partition
    \[U=X[e_1\setminus f_1]\cup X[e_1e_2]\cup X[e_2\setminus f_2].\]
    The sets in the right-hand side being disjoint means that pairwise products among $e_1\setminus f_1$, $e_1e_2$ and $e_2\setminus f_2$, always yields $0$. As we consider $\Sigma$-orders, the join $e\defeq(e_1\setminus f_1)\cup (e_1e_2)\cup (e_2\setminus f_2)$ exists, and $U=X[e]$ by Lemma \ref{lem:ultrafilters.are.prime}.
    
    Therefore, as we can assume $U=X[e]$. Since $U\subseteq X[s^*s]$ then $U=X[e]\cap X[s^*s]=X[es^*s]$, i.e., we may assume $U=X[e]$ with $e\leq s^*s$. Let $z=s|(se)$, so \[X[z^*z]=X[(se)^*(se)]=X[es^*se]=X[e]=U.\qedhere\]
\end{proof}

\begin{proposition}
    $\kappa_{(S,\subseteq)}$ is surjective, i.e., every compact-open bisection $U$ of $\mathbb{P}(S,\subseteq)$ is of the form $[s,U]$ for some $s\in S$ and some $U\subseteq X[s^*s]$.
\end{proposition}
\begin{proof}
    First we write $U$ in terms of the basic open sets: $U=\bigcup_{i=1}^n[s_i,U_i']$, each $U_i'$ being compact-open and hence clopen in $\Omega(E(S))$.
    
    Let $U_1=U_1'$ and for $i\geq 2$, $U_i=U_i'\setminus\bigcup_{j=1}^{i-1}U_j$. Then we have disjoint compact-open bisection $[s_i,U_i]$ (as their sources are disjoint), and also $\bigcup_{i=1}^n[s_i,U_i]\subseteq U$. However, the source map is injective on $U$, and $\so(\bigcup_{i=1}^n[s_i,U_i])=\bigcup_{i=1}^n U_i=\bigcup_{i=1}^n U_i'=\so(U)$, and thus we must have $U=\bigcup_{i=1}^n[s_i,U_i]$. By Lemma \ref{lem:subsetesofbasicarebasic}, we may assume $U_i=\dom(\kappa_{s_i})$.
    
    As the sets $[a_i,U_i]$ are disjoint, and contained in a bisection, then their sources and ranges are also disjoint. It follows that $a_i^*a_j=a_ia_j^*=0$ whenever $i\neq j$, and thus the supremum $a=\bigcup_{i=1}^n a_i$ exists. It readily follows that $X[a^*a]=\bigcup_{i=1}^n U_i$, and thus that $[a,X[a^*a]]=\bigcup_{i=1}^n[a_i,U_i]=U$.\qedhere
\end{proof}

\begin{proposition}
$\kappa_{(S,\subseteq)}$ is a $\subseteq$-isomorphism: $s\subseteq t$ in $(S,\subseteq)$ if and only if $\kappa_{(S,\subseteq)}(s)\subseteq \kappa_{(S,\subseteq)}(t)$ (as sets).
\end{proposition}
\begin{proof}
    Of course $s\subseteq t$ implies $\kappa_{(S,\subseteq)}(s)\subseteq\kappa_{(S,\subseteq)}(t)$, so we deal with the converse implication.
    
    Assume that $\kappa_{(S,\subseteq)}(s)\subseteq\kappa_{(S,\subseteq)}(t)$, i.e., $[s,X[s^*s]]\subseteq[t,X[t^*t]]$. For every $\mathfrak{F}\in X[s^*s]$, we have $[s,\mathfrak{F}]=[t,\mathfrak{F}]$, so we may find $a\subseteq s,t$ with $\mathfrak{F}\in X[a^*a]$. This means that $\left\{X[a^*a]:a\subseteq s,t\right\}$ is an open cover of $X[s^*s]$, so we may find a finite subcover, i.e., $a_1,\ldots,a_n\subseteq s,t$ such that $X[s^*s]=\bigcup_{i=1}^nX[a_i^*a_i]$. Letting $a=\bigcup_{i=1}^n a_i$, we have $a\subseteq s,t$, and $X[a^*a]=X[s^*s]$. Thus we just need to prove that $a=s$.
    
    Indeed, if $a\neq s$ then $s\setminus a\neq 0$, so $X[(s\setminus a)^*(s\setminus a)]$ is a nonempty subset of $X[s^*s]$ which does not intersect $X[a^*a]=X[s^*s]$, a contradiction.\qedhere
\end{proof}

Therefore $\kappa_{(S,\subseteq)}\colon(S,\subseteq)\to(\mathbb{K}(\mathbb{P}(S)),\subseteq)$ is an isomorphism of $\Sigma$-ordered inverse semigroups. Just as for $\zeta$, the proof of the naturality of $\kappa\colon\id_{\Sigma\cat{-Ord}}\to\mathbb{K}\mathbb{P}$ is straightforward by carefully applying the definitions of the functors $\mathbb{P}$ and $\mathbb{K}$, so we also omit it.

We thus conclude:

\begin{theorem}[Stone duality for ample inverse semigroupoids]
    The category $\cat{Amp}_\star$ of ample inverse semigroupoids and their proper continuous covering homomorphisms is dually equivalent to the category $\Sigma\cat{-Ord}$ of $\Sigma$-ordered inverse semigroups and their morphisms.
\end{theorem}

\bibliographystyle{amsplain}
\bibliography{library}

\end{document}